\documentclass[english]{article}
\usepackage[T1]{fontenc}
\usepackage[latin9]{inputenc}
\usepackage{geometry}
\usepackage{color}
\usepackage{array}
\usepackage{float}
\usepackage{mathtools}
\usepackage{bm}
\usepackage{amsmath}
\usepackage{amsthm}
\usepackage{amssymb}
\usepackage{graphicx}
\usepackage{esint}

\makeatletter

\providecommand{\tabularnewline}{\\}

\theoremstyle{plain}
\newtheorem{thm}{\protect\theoremname}
\theoremstyle{plain}
\newtheorem{lem}[thm]{\protect\lemmaname}
\theoremstyle{definition}
\newtheorem{defn}[thm]{\protect\definitionname}
\theoremstyle{plain}
\newtheorem{prop}[thm]{\protect\propositionname}
\theoremstyle{remark}
\newtheorem{rem}[thm]{\protect\remarkname}

\usepackage{algorithmic}
\usepackage{wasysym}
\usepackage{bbm}

\makeatother

\usepackage{babel}
\providecommand{\definitionname}{Definition}
\providecommand{\lemmaname}{Lemma}
\providecommand{\propositionname}{Proposition}
\providecommand{\remarkname}{Remark}
\providecommand{\theoremname}{Theorem}

\begin{document}
\global\long\def\conj{*}%

\global\long\def\Z{\mathbb{Z}}%

\global\long\def\R{\mathbb{R}}%

\global\long\def\C{\mathbb{C}}%

\global\long\def\H{{\cal H}}%

\global\long\def\X{{\cal X}}%

\global\long\def\Q{{\cal Q}}%

\global\long\def\Y{{\cal Y}}%

\global\long\def\e{{\mathbf{e}}}%

\global\long\def\et#1{{\e(#1)}}%

\global\long\def\ef{{\mathbf{\et{\cdot}}}}%

\global\long\def\x{{\mathbf{x}}}%

\global\long\def\w{{\mathbf{w}}}%

\global\long\def\xt#1{{\x(#1)}}%

\global\long\def\xf{{\mathbf{\xt{\cdot}}}}%

\global\long\def\d{{\mathbf{d}}}%

\global\long\def\b{{\mathbf{b}}}%

\global\long\def\u{{\mathbf{u}}}%

\global\long\def\y{{\mathbf{y}}}%

\global\long\def\yt#1{{\y(#1)}}%

\global\long\def\yf{{\mathbf{\yt{\cdot}}}}%

\global\long\def\z{{\mathbf{z}}}%

\global\long\def\v{{\mathbf{v}}}%

\global\long\def\h{{\mathbf{h}}}%

\global\long\def\s{{\mathbf{s}}}%

\global\long\def\c{{\mathbf{c}}}%

\global\long\def\p{{\mathbf{p}}}%

\global\long\def\f{{\mathbf{f}}}%

\global\long\def\g{{\mathbf{g}}}%

\global\long\def\a{{\mathbf{a}}}%

\global\long\def\rb{{\mathbf{r}}}%

\global\long\def\rt#1{{\rb(#1)}}%

\global\long\def\rf{{\mathbf{\rt{\cdot}}}}%

\global\long\def\mat#1{{\ensuremath{\bm{\mathrm{#1}}}}}%

\global\long\def\valpha{\mat{\alpha}}%

\global\long\def\vbeta{\mat{\beta}}%

\global\long\def\vtheta{\mat{\theta}}%

\global\long\def\veta{\mat{\eta}}%

\global\long\def\vmu{\mat{\mu}}%

\global\long\def\vrho{\mat{\rho}}%

\global\long\def\matN{\ensuremath{{\bm{\mathrm{N}}}}}%

\global\long\def\matA{\ensuremath{{\bm{\mathrm{A}}}}}%

\global\long\def\matB{\ensuremath{{\bm{\mathrm{B}}}}}%

\global\long\def\matC{\ensuremath{{\bm{\mathrm{C}}}}}%

\global\long\def\matD{\ensuremath{{\bm{\mathrm{D}}}}}%

\global\long\def\matP{\ensuremath{{\bm{\mathrm{P}}}}}%

\global\long\def\matU{\ensuremath{{\bm{\mathrm{U}}}}}%

\global\long\def\matV{\ensuremath{{\bm{\mathrm{V}}}}}%

\global\long\def\matM{\ensuremath{{\bm{\mathrm{M}}}}}%

\global\long\def\matR{\mat R}%

\global\long\def\matW{\mat W}%

\global\long\def\matK{\mat K}%

\global\long\def\matQ{\mat Q}%

\global\long\def\matS{\mat S}%

\global\long\def\matY{\mat Y}%

\global\long\def\matX{\mat X}%

\global\long\def\matI{\mat I}%

\global\long\def\matJ{\mat J}%

\global\long\def\matZ{\mat Z}%

\global\long\def\matL{\mat L}%

\global\long\def\S#1{{\mathbb{S}_{N}[#1]}}%

\global\long\def\IS#1{{\mathbb{S}_{N}^{-1}[#1]}}%

\global\long\def\PN{\mathbb{P}_{N}}%

\global\long\def\TNormS#1{\|#1\|_{2}^{2}}%

\global\long\def\TNorm#1{\|#1\|_{2}}%

\global\long\def\InfNorm#1{\|#1\|_{\infty}}%

\global\long\def\FNorm#1{\|#1\|_{F}}%

\global\long\def\UNorm#1{\|#1\|_{\matU}}%

\global\long\def\UNormS#1{\|#1\|_{\matU}^{2}}%

\global\long\def\UINormS#1{\|#1\|_{\matU^{-1}}^{2}}%

\global\long\def\ANorm#1{\|#1\|_{\matA}}%

\global\long\def\BNorm#1{\|#1\|_{\mat B}}%

\global\long\def\HNormS#1{\|#1\|_{\H}^{2}}%

\global\long\def\XNormS#1#2{\|#1\|_{#2}^{2}}%

\global\long\def\AINormS#1{\|#1\|_{\matA^{-1}}^{2}}%

\global\long\def\BINormS#1{\|#1\|_{\matB^{-1}}^{2}}%

\global\long\def\BINorm#1{\|#1\|_{\matB^{-1}}}%

\global\long\def\ONorm#1#2{\|#1\|_{#2}}%

\global\long\def\T{\textsc{T}}%

\global\long\def\pinv{\textsc{+}}%

\global\long\def\Expect#1{{\mathbb{E}}\left[#1\right]}%

\global\long\def\ExpectC#1#2{{\mathbb{E}}_{#1}\left[#2\right]}%

\global\long\def\dotprod#1#2#3{(#1,#2)_{#3}}%

\global\long\def\dotprodsqr#1#2{(#1,#2)^{2}}%

\global\long\def\Trace#1{{\bf Tr}\left(#1\right)}%

\global\long\def\realpart#1{{\bf Re}\left(#1\right)}%

\global\long\def\nnz#1{{\bf nnz}\left(#1\right)}%

\global\long\def\range#1{{\bf range}\left(#1\right)}%

\global\long\def\nully#1{{\bf null}\left(#1\right)}%

\global\long\def\vecmat#1{{\bf vec}\left(#1\right)}%

\global\long\def\vol#1{{\bf vol}\left(#1\right)}%

\global\long\def\rank#1{{\bf rank}\left(#1\right)}%

\global\long\def\diag#1{{\bf diag}\left(#1\right)}%

\global\long\def\erfc#1{{\bf erfc}\left(#1\right)}%

\global\long\def\grad#1{{\bf grad}#1}%

\global\long\def\st{\,\,\,\text{s.t.}\,\,\,}%

\global\long\def\KL#1#2{D_{{\bf KL}}\left(#1,#2\right)}%

\title{Gauss-Legendre Features for Gaussian Process Regression}
\author{Paz Fink Shustin, Haim Avron}
\maketitle
\begin{abstract}
Gaussian processes provide a powerful probabilistic kernel learning
framework, which allows learning high quality nonparametric regression
models via methods such as Gaussian process regression. Nevertheless,
the learning phase of Gaussian process regression requires massive
computations which are not realistic for large datasets. In this paper,
we present a Gauss-Legendre quadrature based approach for scaling
up Gaussian process regression via a low rank approximation of the
kernel matrix. We utilize the structure of the low rank approximation
to achieve effective hyperparameter learning, training and prediction.
Our method is very much inspired by the well-known random Fourier
features approach, which also builds low-rank approximations via numerical
integration. However, our method is capable of generating high quality
approximation to the kernel using an amount of features which is poly-logarithmic
in the number of training points, while similar guarantees will require
an amount that is at the very least linear in the number of training
points when random Fourier features. Furthermore, the structure of
the low-rank approximation that our method builds is subtly different
from the one generated by random Fourier features, and this enables
much more efficient hyperparameter learning. The utility of our method
for learning with low-dimensional datasets is demonstrated using numerical
experiments.
\end{abstract}

\section{Introduction}

Gaussian processes (GPs)~\cite{williams2006gaussian} provide a powerful
probabilistic kernel learning framework, which allows learning high
quality nonparametric regression models via methods such as Gaussian
process regression (GPR). Indeed, GP based methods are widely used
in machine learning and statistics. They have been applied to a wide
variety of problems, such as data visualization, Bayesian optimization
\cite{snoek2012practical}, modeling dynamics and spatial data analysis~\cite{stein1999interpolation}.
One of the key advantages of the GP formulation of kernel regression
is that the marginal likelihood is a function of the kernel hyperparameters,
and that it can be computed via a closed-form formula. By maximizing
the marginal likelihood, one can learn the hyperparameters from the
data, thereby tuning the method in a principled manner.

However, learning GPs comes with an hefty computational price-tag.
Given a training set of $n$ points of dimension $d$, exact GPR requires
solving a (usually dense) linear equation, and thus requires $O(n^{3})$
FLOPs. Prediction costs $O(nd)$ FLOPs per test point. Such costs
are problematic for datasets with more than a few thousand points.
The situation is even more severe if we consider the hyperparameter
learning phase: here the cost is $O(n^{3})$ FLOPs per hyperparameter
in a learning iteration (assuming we use a first-order optimization
method). Hyperparameter learning of exact kernel models on large-scale
data is even more unrealistic than training such models.

Given the ubiquity of GPs, it is unsurprising that there is a rich
literature on scaling GP-based method, e.g. \cite{quinonero2005unifying}
and \cite{wilson2015kernel}. One attractive approach is to approximate
the kernel matrix (also known as covariance matrix) $\matK_{\vtheta}$
as a sum of a diagonal matrix (often a multiple of the identity) and
a low rank matrix~\cite{stein2014limitations}:
\begin{equation}
\matK_{\vtheta}\approx\matZ\matW(\vtheta)\matZ^{\conj}+\matD(\vtheta)\label{eq:low-rank-gram}
\end{equation}
In the above, $\matK_{\vtheta}\in\R^{n\times n}$ denotes the kernel
matrix, where the subscript $\vtheta$ denotes the dependence of the
kernel matrix on the hyperparameters $\vtheta$ (discussion of our
notation appears in Section~\ref{subsec:notation}), $\matZ$ has
$s\ll n$ columns, and $\matW(\vtheta),\matD(\vtheta)$ are diagonal
matrices. The various steps of GPR can be much more efficiently conducted
on a kernel whose kernel matrix has the structure of the righthand
side of Eq.~(\ref{eq:low-rank-gram}), e.g. training takes $O(ns^{2})$
(see Section~(\ref{subsec:efficient-gpr}) for details on efficient
GPR with low-rank approximations with an even more restricted structure
in which $\matD(\vtheta)$ is a multiple of the identity).

In the kernel learning literature, methods for forming a low rank
approximation of kernels can be roughly split into two approaches:
methods that use data-dependent basis functions, and methods that
use independent basis functions. An example for the first kind is
the Nyström method \cite{WilliamsSeeger01}. Such methods utilize
the given training data, and thus may outperform methods that use
independent basis functions, especially when there is a large gap
in the eigenspectrum. However, data dependence can incur additional
costs. For example, the Nyström method requires keeping some of the
data as part of the model.

Another class of methods for building low rank approximations of kernel
matrices are methods that use independent basis functions, and thus
approximate the kernel function directly. One such important and highly
influential method is the \emph{random Fourier features} approach
suggested by Rahimi and Recht in 2007~\cite{rahimi2008random}. Following
the publication of~\cite{rahimi2008random}, there has been extensive
research on random features, including works that attempt to improve
the approximation quality of the method (e.g.,~\cite{sutherland2015error,ChoromanskiEtAl18}),
works that focused on using in random features to learn huge datasets
(e.g.~\cite{HuangEtAl14,AvronSindhwani16}), and works that focused
on theoretical analysis of random features (e.g.~\cite{yang2012nystrom,sriperumbudur2015optimal,avron2017random}).
The previous list is far from exhaustive. In the context of our work,
worth mentioning is \cite{avron2017random} which showed that if the
kernel matrix of the approximate kernel spectrally approximates the
kernel matrix of the true kernel then the excess risk when using kernel
ridge regression with the approximate kernel is not much larger than
the excess risk when using the true kernel.

Random Fourier features, and random features methods in general, are
based on writing the kernel function as an integral and then using
numerical integration schemes in order to construct a low rank approximation
of that function\footnote{A rank $k$ bivariate function $f(\x,\y)$ is a function that can
be written as $f(\x,\y)=\sum_{j=1}^{k}\sigma_{j}\phi_{j}(\x)\psi_{j}(\y)$
for some $\sigma_{1},\dots,\sigma_{k},\phi_{1},\dots,\phi_{k},\psi_{1},\dots,\psi_{k}$~\cite{TownsendTrefethen13}.}. In random Fourier features, a shift-invariant kernel is rewritten
as an integral via an application of Bochner's theorem, and Monte-Carlo
integration is used to build the low rank approximation. The use of
Quasi Monte-Carlo in lieu of Monte-Carlo integration was explored
in \cite{avron2016quasi}. Bach explored the connection between random
Fourier features and kernel quadrature rules in \cite{bach2017equivalence},
however without providing any practically useful explicit mappings
for kernels. Monte-Carlo and Quasi-Monte Carlo integration admit only
slow convergence rate. As a consequence, the number of features required
for spectral approximation when using Monte-Carlo or Quasi Monte-Carlo
integration must be polynomial in quality parameter of the spectral
approximation. In this paper we argue that in the context of Gaussian
process regression a stronger notion of spectral equivalence is required.
The slow convergence rate of Monte-Carlo or Quasi Monte-Carlo integration
implies that at best the number of features required for spectral
equivalence is linear in the number of training points, which is obviously
undesirable.

Random features approaches based on Monte-Carlo and Quasi Monte-Carlo
suffer from another serious defect when it comes to GPR: the low rank
approximation they build has the form
\begin{equation}
\matK_{\vtheta}\approx\matZ(\vtheta)\matZ(\vtheta)^{\conj}+\matD(\vtheta)\label{eq:low-rank-RFF}
\end{equation}
While for training and prediction, this structure works equally as
well as the structure in Eq.~(\ref{eq:low-rank-gram}), when it comes
to hyperparameter learning this is no longer the case; see Section~\ref{subsec:rff-problematic}.

One can construct faster converging low-rank approximations using
numerical quadrature rules such as Gaussian quadrature. Dao et al.
considered the use of Gaussian quadrature in the context of kernel
learning~\cite{dao2017gaussian}. Gaussian quadrature rule is a method
for numerically approximating \emph{weighted }integrals (i.e., integrals
of the form $\int_{-\infty}^{\infty}f(x)w(x)dx$ where $w(x)\geq0$
is a weight function) that is optimal in some formal sense. In the
context of approximating kernel functions, the weight function $w(\cdot)$
is determined by the kernel function and the value of the hyperparameters.
Once the weight function has been determined, in order to use a Gaussian
quadrature the nodes and weights corresponding to that particular
weight function must be computed. Efficient algorithms exist, but
these algorithms require the computation of integrals as well. For
a single kernel, that is when using a fixed value of the hyperparameters,
and when using a fixed number of quadrature features, computing the
nodes and weights is a one-time offline task. However, if the hyperparameters
are not fixed, e.g. when they are set using hyperparameter learning,
Gaussian quadrature becomes unrealistic. Furthermore, the fact that
the nodes and weights change with the hyperparameters implies that
we must use an approximation of the form of Eq.~(\ref{eq:low-rank-RFF})
and not of Eq.~(\ref{eq:low-rank-gram}), which is less desirable.
The connection between random Fourier features and quadrature rules
was also explored in \cite{munkhoeva2018quadrature}.

Low rank approximations for kernels matrices have also been widely
used in the statistics literature, and in particular the spatial statistics
literature \cite{cressie2008fixed,eidsvik2012approximate,BanerjeeEtAl08,finley2009improving,katzfuss2012bayesian}.
Possible limitations of the low rank approximation approach in the
context of spatial statistics have been noted in \cite{quinonero2005unifying,BanerjeeEtAl08,stein2007spatial,sang2011covariance},
and analyzed mathematically in \cite{stein2014limitations}. The use
of random features in the context of spatial statistics was explored
in \cite{ton2018spatial}.

In this paper we propose a quadrature based low-rank approximation
approach for efficient GPR involving a wide class of kernels which
includes shift-invariant kernels (i.e., stationary covariance functions).
Unlike previous literature which uses quadrature features in the context
of GPR, our method forms an approximation of the form of Eq.~(\ref{eq:low-rank-gram}),
and so is able to efficiently perform hyperparameter learning in addition
to training and prediction. Our method achieves this by using a fixed
set of quadrature nodes and weights, and designing the approximation
so that varying the hyperparameters corresponds to only changing the
integrand.

Specifically, our method uses a Gauss-Legendre quadrature. Gauss-Legendre
quadrature is a Gaussian quadrature for the uniform weight function
on a finite interval. Thus, the weight function does not change with
the hyperparameters, and with it the quadrature nodes and weight stay
fixed, whereas only the integrand varies. Changing only the integrand
translates to a simplified parametric form for the approximate kernel
matrix (Eq.~(\ref{eq:low-rank-gram})) which is more amenable to
efficient computations. Our proposed method, which we call \emph{Gauss-Legendre
Features, }is described in Section~\ref{sec:GL-Features}.

The Gauss-Legendre quadrature is designed to approximate integrals
with an integration area which is a finite interval. However, for
most widely-used kernels the integrand has infinite support. We address
this issue by utilizing the fact that for such kernels the integrand
decays quickly, so we can approximate the integral by truncating the
integration area. The truncation cutoff is determined by a parameter
of our method. Another parameter is the number of features (i.e.,
quadrature nodes) used in the approximation. In order to set these
two parameters correctly, we need a method for assessing the quality
of one kernel function approximation by another. To that end, we introduce
the notion of \emph{spectral equivalence}, and argue that if one parameterized
family of kernels is spectrally equivalent to another one, then that
first family is a good surrogate for the second family in the context
of GPR. These results are summarized in Section~\ref{sec:spectral}.

We rigorously analyze how to set the truncation cutoff and the number
of features to achieve spectral equivalence (these results are reported
in Section~\ref{sec:parameters}). Here another advantage of using
Gauss-Legendre quadrature becomes evident: the Gauss-Legendre quadrature
converges much faster the Monte-Carlo or Quasi Monte-Carlo integration,
so typically the number of features is sublinear in the training size.
Indeed, for widely used kernels like the Gaussian kernel and the Matèrn
kernel, the number of features required when using Gauss-Legendre
features is poly-logarithmic in the training size (see Section~\ref{sec:examples}).
Sublinear number of features is also likely achievable using Gaussian
quadrature (kernel learning using Gaussian quadrature is suggested
in \cite{dao2017gaussian}, however without proving spectral equivalence).
Yet, as explained this is rather problematic for hyperparameter learning,
and in general requires a large overhead for computing the quadrature
nodes and weights.

Finally, empirical results (Section~\ref{sec:Experiments}) clearly
demonstrate the superiority of our proposed method over classical
random Fourier features when conducting Gaussian process regression
on low-dimensional datasets.

\section{Preliminaries}

\subsection{\label{subsec:notation}Notation and Basic Definitions}

We consider all vectors as column vectors, unless otherwise stated.
For a vector $\x$ or a matrix $\mat A$, the notation $\x^{*}$ or
$\mat A^{*}$ denotes the Hermitian transpose. The $n\times n$ identity
matrix is denoted by $\mat I_{n}$. A Hermitian matrix $\mat A$ is
positive semidefinite (PSD) if $\x^{*}\mat A\x\geq0$ for every vector
$\x$. Also, for any Hermitian matrices $\mat A,\,\mat B$ of the
same size, the notation $\mat A\preceq\mat B$ means that $\mat B-\mat A$
is PSD.

We consider $n$ pairs of training data $(\x_{1},y_{1}),\ldots,(\x_{n},y_{n})\in{\cal X}\times{\cal Y}\subset\mathbb{R}^{d}\times\mathbb{R}$,
where $\x$ denotes the input vector of dimension $d$ and $y$ denotes
a scalar response. A \emph{kernel function }(aka \emph{covariance
function}) is a function $k:{\cal X}\times{\cal X}\to\mathbb{R}$
which is \emph{positive definite}, i.e. for every $m\in\mathbb{N}$
and $\x_{1},\ldots,\x_{m}\in\mathbb{R}^{d}$, the matrix $\matK\in\R^{m\times m}$
defined by $\mat K_{ij}=k(\x_{i},\x_{j})$ is PSD. The matrix $\matK$
is known by various names: \emph{kernel matrix, Gram matrix, covariance
matrix. }Given a dataset $\x_{1},\dots,\x_{n}\in\R^{d}$, we will
conveniently use $\matX$ to denote the $n$-by-$d$ matrix whose
rows are $\x_{1},\dots,\x_{n}$, and use $\matK(\matX,\matX)$ to
denote the kernel matrix corresponding to the kernel $k$ with data
$\matX$. For another kernel $\tilde{k}$ we will use $\tilde{\matK}(\matX,\matX)$
to denote the kernel matrix.

In many cases we will deal with parameterized families of kernels
$\{k_{\vtheta}\}_{\vtheta\in\Theta}$, where $\vtheta$ represents
the hyperparameters vector, and $\Theta$ is a set of possible parameters
values. The kernel matrix corresponding to $k_{\vtheta}$ is denoted
by $\matK_{\vtheta}(\matX,\matX)$. We also group the responses $y_{1},\dots,y_{n}$
into a single vector $\y\in\R^{n}$.

The \emph{Kullback--Leibler divergence} (abbreviated K\emph{L-divergence}
henceforth) is a well established metric for how much one distribution
is different from a reference distribution. We denote the KL-divergence
between two probability distributions ${\cal P}$ on ${\cal Q}$ by
$\KL{{\cal P}}{{\cal Q}}$, and recall the following is a well established
result\footnote{The exact definition of the KL-divergence is not important, since
we always use Eq.~(\ref{eq:KL_div}) when working with it.}: if ${\cal N}_{1}={\cal N}(\vmu_{0},\Sigma_{0})$ and ${\cal N}_{2}={\cal N}(\vmu_{1},\Sigma_{1})$
are two multivariate normal distributions, we have
\begin{equation}
\KL{{\cal N}_{1}}{{\cal N}_{2}}=\frac{1}{2}\Trace{\mat{\Sigma}_{1}^{-1}\mat{\Sigma}_{0}}+\frac{1}{2}\left(\mat{\mu}_{1}-\mat{\mu}_{0}\right)^{\T}\mat{\Sigma}_{1}^{-1}\left(\mat{\mu}_{1}-\mat{\mu}_{0}\right)+\frac{1}{2}\left(\log\det\mat{\Sigma}_{1}-\log\det\mat{\Sigma}_{0}\right)-\frac{n}{2}\label{eq:KL_div}
\end{equation}

\subsection{Gaussian Process Regression}

Gaussian Process Regression (GPR) is a Bayesian nonparametric approach
for regression. First, the following regression model is assumed:
\[
y=f(\x)+\varepsilon,\quad\varepsilon\overset{i.i.d}{\sim}\mathcal{N}\left(\mat 0,\sigma_{n}^{2}\right)
\]
($\sigma_{n}^{2}$ is a (hyper)parameter). Additionally, it is assumed
that $f$ is a \emph{Gaussian Process}, $f(\x)\sim{\cal GP}\left(\vmu\left(\x\right),k(\x,\x')\right)$,
where $k$ is the kernel function. This means that for any set of
data points $\x_{1},\dots,\x_{m}\in\mathbb{R}^{d}$ the vector $\f\in\R^{m}$
defined by $\f_{j}=f(\x_{j})$ ($j=1,\dots,m$) is a Gaussian random
vector with mean defined by $\mu(\matX)=[\mu(\x_{1})\cdots\mu(\x_{m})]^{\T}$
and covariance matrix $\matK(\matZ,\matZ)$. Throughout the paper
we assume, for the sake of simplicity, that the mean function $\mu(\x)$
is $0$. This simplifies the formulas while not really restricting
generality (a nonzero mean can be easily handled). Under these assumptions
$y\sim{\cal N}(0,\matK(\matX,\matX)+\sigma_{n}^{2}\matI)$. Under
these priors, the expected predictive value for $f(\x)$ at a test
$\x$ is 
\[
f(\x)\approx\matK(\x,\matX)\left(\matK(\matX,\matX)+\sigma_{n}^{2}\matI_{n}\right)^{-1}\y
\]
Consequently, training is conducted by computing the vector 
\[
\valpha\coloneqq\left(\matK(\matX,\matX)+\sigma_{n}^{2}\matI_{n}\right)^{-1}\y
\]
From these formulas we see that assuming that evaluating the kernel
function takes $O(d)$ operations and that $\valpha$ is computed
using direct factorization, training takes $O(n^{3})$ operations
and prediction takes $O(nd)$ operations.

The previous description is for a fixed kernel $k$. Typically, the
kernel $k_{\vtheta}$ has \textit{hyperparameters} which we represent
throughout the paper by the vector $\mat{\theta}$. The hyperparameters
are usually constrainted to some possible set of hyperparameters values
$\Theta$, and thus defines a parameterized family of kernels $\{k_{\vtheta}\}_{\vtheta\in\Theta}$.
\emph{Hyperparameter learning }refers to the process of determining
the value of the hyperparameters directly from the training data,
and is considered one of the important advantages of the GP framework.
This is typically conducted by maximizing the log marginal likelihood:

\begin{equation}
{\cal L}(\vtheta)\coloneqq-\frac{1}{2}\y^{\T}\left(\mat K_{\vtheta}(\matX,\matX)+\sigma_{n}^{2}\matI_{n}\right)^{-1}\y-\frac{1}{2}\log\det(\mat K_{\vtheta}(\matX,\matX)+\sigma_{n}^{2}\matI_{n})-\frac{n}{2}\log2\pi\label{eq:likelihood}
\end{equation}
In order to maximize ${\cal L}(\vtheta)$ using a first-order optimization
method it is required to compute its gradients. Using direct methods,
computing the gradient takes $O(n^{3}|\vtheta|)$ where $|\vtheta|$
represents the number of hyperparameters in $\vtheta$.

\subsection{Random Fourier Features}

Random Fourier Features (RFF)~\cite{rahimi2008random}, is one of
the most popular methods for constructing a low rank approximation
of kernels and scaling up kernel methods. The method targets shift-invariant
kernels, i.e. kernels of the form $k(\x,\x')=k_{0}(\x-\x')$ for a
positive definite function $k_{0}(\cdot)$.

RFF is motivated by a simple consequence of Bochner\textquoteright s
Theorem: for every shift-invariant kernel for which $k_{0}(0)=\sigma_{f}^{2}$
there is a probability measure $\mu$ and possibly a corresponding
probability density function $p(\cdot)$, both on $\R^{d}$, such
that
\[
k\left(\x,\x'\right)=\sigma_{f}^{2}\int_{\mathbb{R}^{d}}e^{-2\pi i\mat{\eta}^{\T}\left(\x-\x'\right)}d\mu(\mat{\eta})=\sigma_{f}^{2}\int_{\mathbb{R}^{d}}e^{-2\pi i\mat{\eta}^{\T}\left(\x-\x'\right)}p(\mat{\eta})d\mat{\eta}
\]
Let us assume that the density $p(\cdot)$ exist. If one chooses $\mat{\eta}_{1},\ldots,\mat{\eta}_{s}$
randomly according to $p(\cdot)$, and defines $\varphi(\x)=\frac{1}{\sqrt{s}}\left(e^{-2\pi i\mat{\eta}_{1}^{\T}\x},\ldots,e^{-2\pi i\mat{\eta}_{s}^{\T}\x}\right)^{*}$,
then
\[
k\left(\x,\x'\right)=\sigma_{f}^{2}\mathbb{E}_{\veta_{1},\dots\veta_{s}}\left[\varphi\left(\x\right)^{*}\varphi\left(\x'\right)\right]\,.
\]
So, an approximated kernel can be defined:
\[
\tilde{k}^{\text{(RFF)}}(\x,\x')\coloneqq\varphi(\x)^{*}\varphi(\x')=\frac{\sigma_{f}^{2}}{s}\sum_{j=1}^{s}e^{-2\pi i\mat{\eta}_{j}^{\T}(\x-\x')}\,.
\]
The kernel matrix corresponding to the approximate kernel is
\[
\tilde{\matK}^{\text{(RFF)}}(\matX,\matX)=\matZ\matZ^{\conj}
\]
where $\mat Z\in\mathbb{C}^{n\times s}$ to be the matrix whose $m^{th}$
row is $\varphi(\x_{m})^{*}$. The low rank structure of $\tilde{\matK}^{\text{(RFF)}}(\matX,\matX)$
allows more efficient training ($O(ns^{2})$) and predictions ($O(sd$)),
which are attractive if $s\ll n$.

The previous description is for a fixed kernel (and fixed hyperparameters).
When using GPR with hyperparameter learning we are dealing parameterized
family of kernels $\{k_{\vtheta}\}_{\vtheta\in\Theta}$. This case
has not been considered in Rahimi and Recht original work~\cite{rahimi2008random}.
We discuss it in Section~\ref{subsec:rff-problematic}.

\section{\label{sec:spectral}Spectrally Equivalent Kernel Approximations}

Our strategy for scaling up GPR is based on approximating the kernel
$k_{\vtheta}$ by an approximate kernel $\tilde{k}_{\vtheta}$ that
is low-rank in some sense which will become apparent in the next section.
This raises the question: how can we determine whether $\tilde{k}_{\vtheta}$
indeed approximates $k_{\vtheta}$ well? In \cite{avron2017random},
the authors suggested that in the context of kernel ridge regression,
\emph{spectral approximations} of the kernel matrices allows us to
reason about how well one kernel is approximated by another. The argument
in~\cite{avron2017random} is based on risk bounds for given fixed
hyperparameters, and so is less appropriate for GPR where hyperparameter
learning is common practice. In this section, we introduce the notion
\emph{spectral equivalence}, a stronger form of spectral approximation,
and connect it to hyperparameter learning in GPR.

Assume a bounding set $\X\subseteq\R^{d}$ for the data, then given
a dataset $(\x_{1},y_{1}),\dots,(\x_{n},y_{n})\in\X\times\R$, the
general assumption when using a kernel $k$ is that 
\[
\y\sim{\cal N}(\vmu,\matK(\matX,\matX))
\]
If, however, we would have used the kernel $\tilde{k}$, then the
assumption would have been
\[
\y\sim{\cal N}(\vmu,\tilde{\matK}(\matX,\matX))
\]
Thus, a measure on how well $\tilde{k}$ approximates $k$ might be
devised by measuring how much ${\cal N}(\vmu,\tilde{\matK}(\matX,\matX))$
is different from ${\cal N}(\vmu,\matK(\matX,\matX))$. The KL-divergence
is a well-established measure on how different one probability distribution
is from a reference distribution, so arguably, $\tilde{k}$ approximates
$k$ well if the KL-divergence $\KL{{\cal N}(\vmu,\matK(\matX,\matX))}{{\cal N}(\vmu,\tilde{\matK}(\matX,\matX))}$
is small. Indeed, the use $\KL{{\cal N}(\vmu,\matK(\matX,\matX))}{{\cal N}(\vmu,\tilde{\matK}(\matX,\matX))}$
as such a measure was suggested in the literature on spatial data
analysis~\cite{BanerjeeEtAl08,SangHuang12,stein2014limitations}.

The notion of spectral equivalence, which we develop below, is a measure
on how two matrices are close to one another. To connect it to the
KL-divergence, which we use to measure how well $\tilde{k}$ approximates
$k$, we have the following lemma, which implies that if the covariance
matrices of two multivariate distributions are close, then the KL-divergence
is small.
\begin{lem}
\label{lem:kl-small}Suppose that $\vmu\in\R^{n}$ and $\mat{\Sigma}_{0,},\mat{\Sigma}_{1}\in\R^{n\times n}$
are two symmetric positive definite matrices. Suppose that 
\begin{equation}
(1-n^{-1})\mat{\Sigma}_{0}\preceq\mat{\Sigma}_{1}\preceq(1+n^{-1})\mat{\Sigma}_{0}\label{eq:spec_eq-1}
\end{equation}
Then,
\[
\KL{{\cal N}(\vmu,\Sigma_{0})}{{\cal N}(\vmu,\Sigma_{1})}\leq1+O(n^{-1})
\]
\end{lem}

We first need the following Lemma.
\begin{lem}
\label{lem:spec_logdet-1}Suppose that $\matA$ and $\matB$ are two
symmetric positive definite matrices of order $n\times n$, such that
\begin{equation}
(1-n^{-1})\matB\preceq\matA\preceq(1+n^{-1})\matB\,.\label{eq:spec-eq-AB}
\end{equation}
Then, there exists $\gamma_{1},\dots,\gamma_{n}\in[-n^{-1},n^{-1}]$
such that 
\[
\log\det\matA-\log\det\matB=\sum_{i=1}^{n}\log(1+\gamma_{i})\,\text{.}
\]
\end{lem}

\begin{proof}
Let $\lambda_{1},\dots,\lambda_{n}$ denote the sorted eigenvalues
of $\matB$, and $\tilde{\lambda}_{1},\dots,\tilde{\lambda}_{n}$
denote the sorted eigenvalues of $\matA$, so
\[
\log\det\matB=\sum_{i=1}^{n}\log\lambda_{i},\quad\log\det\matA=\sum_{i=1}^{n}\log\tilde{\lambda}_{i}\,.
\]
 Eq.~(\ref{eq:spec-eq-AB}) implies that there exist $\gamma_{1},\dots,\gamma_{n}\in[-n^{-1},n^{-1}]$
such that $\tilde{\lambda}_{i}=(1+\gamma_{i})\lambda_{i}$. Hence,
\begin{eqnarray*}
\log\det\matA & = & \sum_{i=1}^{n}\log\tilde{\lambda}_{i}\\
 & = & \sum_{i=1}^{n}\log(1+\gamma_{i})\lambda_{i}\\
 & = & \log\det\matB+\sum_{i=1}^{n}\log(1+\gamma_{i})
\end{eqnarray*}
and that completes the proof.
\end{proof}
\begin{proof}
[Proof of Lemma \ref{lem:kl-small}]Since for two symmetric positive
definite matrices $\matA$ and $\matB$, $\matA\preceq\matB$ implies
$\matB^{-1}\preceq\matA^{-1}$, Eq.~(\ref{eq:spec_eq-1}) implies
that
\[
\left(1-\frac{1}{n+1}\right)\mat{\Sigma}_{0}^{-1}=\frac{1}{1+n^{-1}}\mat{\Sigma}_{0}^{-1}\preceq\mat{\Sigma}_{1}^{-1}\preceq\frac{1}{1-n^{-1}}\mat{\Sigma}_{0}^{-1}=\left(1+\frac{1}{n-1}\right)\mat{\Sigma}_{0}^{-1}\,.
\]
Multiplying by $\mat{\Sigma}_{0}^{1/2}$ on the right and left sides
gives
\[
\left(1-\frac{1}{n+1}\right)\mat I_{n}\preceq\mat{\Sigma}_{0}^{1/2}\mat{\Sigma}_{1}^{-1}\mat{\Sigma}_{0}^{1/2}\preceq\left(1+\frac{1}{n-1}\right)\mat I_{n}
\]
i.e., the eigenvalues of $\mat{\Sigma}_{0}^{1/2}\mat{\Sigma}_{1}^{-1}\mat{\Sigma}_{0}^{1/2}$
are bounded in the interval $[1-(n+1)^{-1},1+(n-1)^{-1}]$. Thus,
\[
\Trace{\mat{\Sigma}_{1}^{-1}\mat{\Sigma}_{0}}=\Trace{\mat{\Sigma}_{1}^{-1}\mat{\Sigma}_{0}^{1/2}\mat{\Sigma}_{0}^{1/2}}=\Trace{\mat{\Sigma}_{0}^{1/2}\mat{\Sigma}_{1}^{-1}\mat{\Sigma}_{0}^{1/2}}\leq n\left(1+\frac{1}{n-1}\right)\,.
\]
Also, from Lemma \ref{lem:spec_logdet-1}, there exist $\gamma_{1},\dots,\gamma_{n}\in[-n^{-1},n^{-1}]$
such that
\[
\log\det\mat{\Sigma}_{1}-\log\det\mat{\Sigma}_{0}=\sum_{i=1}^{n}\log(1+\gamma_{i})\,.
\]
 Using Eq. (\ref{eq:KL_div}), we obtain
\begin{eqnarray*}
\KL{{\cal N}(\vmu,\Sigma_{0})}{{\cal N}(\vmu,\Sigma_{1})} & = & \frac{1}{2}\Trace{\mat{\Sigma}_{1}^{-1}\mat{\Sigma}_{0}}+\frac{1}{2}(\log\det\mat{\Sigma}_{1}-\log\det\mat{\Sigma}_{0})-\frac{n}{2}\\
 & \leq & \frac{n}{2}\left(1+\frac{1}{n-1}\right)+\frac{1}{2}\sum_{i=1}^{n}\log\left(1+\gamma_{i}\right)-\frac{n}{2}\\
 & \leq & \frac{n}{2}\left(\frac{1}{n-1}+\log\left(1+\frac{1}{n}\right)\right)\\
 & = & \frac{1}{2}+\frac{1}{2(n-1)}+\frac{n}{2}\left(\frac{1}{n}+O\left(\frac{1}{n^{2}}\right)\right)\\
 & = & 1+O\left(\frac{1}{n}\right)\,.
\end{eqnarray*}
\end{proof}
Lemma~\ref{lem:kl-small} motivates the following definitions:
\begin{defn}
We say that a $n$-by-$n$ symmetric matrix $\matA$ is \emph{spectrally
equivalent }to another $n$-by-$n$ symmetric matrix $\matB$ if 
\begin{equation}
(1-n^{-1})\matB\preceq\matA\preceq(1+n^{-1})\matB.\label{eq:spec_eq}
\end{equation}
\end{defn}

\begin{defn}
Let $n\geq1$ be an integer, and $\X$ be a data domain. Two positive
definite kernels $k$ and $\tilde{k}$ are \textit{\textcolor{black}{$n$-spectrally
equivalent}} \emph{on domain $\X$} if for every $\matX$ with $n$
rows in $\X$, the kernel matrix $\tilde{\mat K}(\matX,\matX)$ is
spectrally equivalent to the kernel matrix $\mat K(\matX,\matX)$.
\end{defn}

The last definition uses two specific kernels, $k$ and $\tilde{k}$.
In GPR, we usually use a parameterized family of kernels $\{k_{\vtheta}\}_{\vtheta\in\Theta}$,
where $\vtheta$ represents the hyperparameters, and $\Theta$ is
a set of possible parameter values. We generally assume that $\Theta$
is bounded. Boundedness of $\Theta$ is necessary, since without it,
it is possible to drive the kernel matrix to identity, thereby making
it impossible to approximate it using a low rank matrix. We then approximate
each kernel $k_{\vtheta}$ by $\tilde{k}_{\vtheta}$, that is we use
the parameterized family $\{\tilde{k}_{\vtheta}\}_{\vtheta\in\Theta}$.
We say that the parameterized family $\{\tilde{k}_{\vtheta}\}_{\vtheta\in\Theta}$
approximates the parameterized family $\{k_{\vtheta}\}_{\vtheta\in\Theta}$
well if for every $\vtheta\in\Theta$ the kernel $\tilde{k}_{\vtheta}$
approximates $k_{\vtheta}$ well, as is captured by the following
definition.
\begin{defn}
Two parameterized families of positive definite kernels $\{k_{\vtheta}\}_{\vtheta\in\Theta}$
and $\{\tilde{k}_{\vtheta}\}_{\vtheta\in\Theta}$ are said to be \textit{\textcolor{black}{$n$-spectrally
equivalent}} on domain $\X$ if for every $\mat{\theta}\in\mat{\Theta}$,
$k_{\mat{\theta}}$ and $\tilde{k}_{\mat{\theta}}$ are $n$-spectrally
equivalent over $\X$.
\end{defn}

In light of Lemma~\ref{lem:kl-small}, if two parameterized families
$\{k_{\vtheta}\}_{\vtheta\in\Theta}$ and $\{\tilde{k}_{\vtheta}\}_{\vtheta\in\Theta}$
are \textit{\textcolor{black}{\emph{$n$-spectrally equivalent, then
for any parameters $\vtheta$ and any dataset consisting of $n$ data
points, the distributions on the response assumed by the two GP models
induced by these families are close in the sense that the KL-divergence
is close to $1$.}}}

\section{\label{sec:GL-Features}Gauss-Legendre Features}

In this section, we present our proposed method (Gauss-Legendre Features),
and show how it can be used to perform efficient Gaussian process
regression. Our method includes two important parameter vectors: $\matU$
and $\s$. In the next section we show how these parameters can be
set in order to obtain an approximation that is spectrally equivalent
to the true kernel.

\subsection{Feature Map}

We begin by describing the Gauss-Legendre feature map. The proposed
method builds feature maps for kernel families that can be written
in the following form:
\begin{equation}
k_{\vtheta}(\x,\x')=\sigma_{f}^{2}\int_{\R^{d}}\varphi(\x,\veta)\varphi(\x',\veta)^{*}p(\veta;\vtheta_{0})d\veta+\sigma_{n}^{2}\gamma(\x-\x')\label{eq:kernel-form}
\end{equation}
In the above, 
\[
\gamma(\z)\coloneqq\begin{cases}
1 & \z=0\\
0 & \z\neq0
\end{cases}\,,
\]
the function $\varphi:\X\times\R^{d}\to\C$ is such that for every
$\x\in\X$ the function $\varphi(\x,\cdot)$ is even-symmetric (i.e.,
for every $\veta\in\R^{d}$, $\varphi(\x,\veta)=\varphi(\x,-\veta)^{*}$),
$\vtheta=[\vtheta_{0},\sigma_{f}^{2},\sigma_{n}^{2}]$, and for every
$\vtheta_{0}$ the function $p(\cdot;\vtheta_{0})$ is an even probability
density on $\R^{d}$. Note that $\vtheta_{0}$ can be a vector.

Note that in Eq.~(\ref{eq:kernel-form}) we included a ridge term
$\sigma_{n}^{2}\gamma(\x-\x')$. Typically, the ridge term is omitted
from the kernel but appears in various equations involving the kernel
matrix due to Gaussian noise assumption in the GPR model. While we
could state our theory in the more traditional way of having the noise
term outside of the kernel, the definitions and theorems statements
will be somewhat more cumbersome. We found it more convenient to include
$\sigma_{n}^{2}$ as part of the vector of the parameter set $\vtheta$,
and include the ridge term in the kernel definition. The resulting
equations are the same.

There are quite a few kernel families that adhere to this structure.
For example, due to Bochener's theorem, a shift-invariant kernel with
an additional noise level term can be written in the form
\begin{equation}
k_{\vtheta}(\x,\x')=\sigma_{f}^{2}\int_{\R^{d}}e^{-i(\x-\x')^{\T}\veta}p(\veta;\vtheta_{0})d\veta+\sigma_{n}^{2}\gamma(\x-\x')\label{eq:base-kernel}
\end{equation}
So, we can use $\varphi(\x,\veta)=e^{-i\x^{\T}\veta}$ to cast shift
invariants kernels in the form of Eq.~(\ref{eq:kernel-form}). 

The underlying idea of Gauss-Legendre features is to first truncate
the integral Eq.~(\ref{eq:base-kernel}) to the domain ${\cal Q}_{\matU}=\prod_{k=1}^{d}[-U_{k},U_{k}]$,
for some $\matU=(U_{1},\dots,U_{d})^{\T}$, and then approximate the
truncated integral using a tensorized Gauss-Legendre quadrature. The
domain ${\cal Q}_{\matU}$ might depend on $\X$ and $\Theta$, but
not on the concrete dataset $\x_{1},\dots,\x_{n}$. Let $\text{(\ensuremath{\chi_{1}^{(m)},w_{1}^{(m)}),\dots,(\chi_{m}^{(m)},w_{m}^{(m)})}}$
denote the nodes and weights of the $m$-point Gauss-Legendre. Assume
we are given a list of quadrature size for each dimension: $\s=(s_{1},\dots,s_{d}$).
Let $s=\prod_{k=1}^{d}s_{k}$. The approximation then reads:
\begin{eqnarray}
k_{\vtheta}(\x,\x') & \approx & \sigma_{f}^{2}\int_{{\cal Q}_{U}}\varphi(\x,\veta)\varphi(\x',\veta)^{*}p(\veta;\vtheta_{0})d\veta+\sigma_{n}^{2}\gamma(\x-\x')\nonumber \\
 & \approx & \sigma_{f}^{2}\sum_{j_{1}=1}^{s_{1}}\cdots\sum_{j_{d}=1}^{s_{d}}w_{j_{1}\cdots j_{d}}p(\hat{\veta}_{j_{1}\cdots j_{d}};\vtheta_{0})\varphi(\x,\hat{\veta}_{j_{1}\dots j_{d}})\varphi(\x',\hat{\veta}_{j_{1}\cdots j_{d}})^{*}+\sigma_{n}^{2}\gamma(\x-\x')\nonumber \\
 & = & \sigma_{f}^{2}\sum_{j=1}^{s}h_{j}(\vtheta_{0})\varphi(\x,\veta_{j})\varphi(\x',\veta_{j})^{*}+\sigma_{n}^{2}\gamma(\x-\x')\label{eq:bijec-sum}
\end{eqnarray}
where 
\[
\hat{\veta}_{j_{1}\cdots j_{d}}\coloneqq\left[\begin{array}{c}
\eta_{j_{1}}^{(s_{1})}\\
\vdots\\
\eta_{j_{d}}^{(s_{d})}
\end{array}\right]=\left[\begin{array}{c}
U_{1}\cdot\chi_{j_{1}}^{(s_{1})}\\
\vdots\\
U_{d}\cdot\chi_{j_{d}}^{(s_{d})}
\end{array}\right]
\]
and 
\[
w_{j_{1}\cdots j_{d}}=\prod_{k=1}^{d}U_{k}w_{j_{k}}^{(s_{k})}\,.
\]
In Eq.~(\ref{eq:bijec-sum}), we assumed we have a bijective mapping
$a_{1},\dots,a_{s}$ between $\{1,\dots,s\}$ and $\{1,\dots,s_{1}\}\times\dots\times\{1,\dots,s_{d}\}$
and then defined:
\[
\veta_{j}\coloneqq\hat{\veta}_{a_{j}}\quad\quad h_{j}(\vtheta_{0})\coloneqq w_{a_{j}}p(\veta_{j};\vtheta_{0})\,.
\]
Finally, the parameterized family of approximate kernels is 
\begin{equation}
\tilde{k}_{\vtheta}(\x,\x')\coloneqq\sigma_{f}^{2}\sum_{j=1}^{s}h_{j}(\vtheta_{0})\varphi(\x,\veta_{j})\varphi(\x',\veta_{j})^{*}+\sigma_{n}^{2}\gamma(\x,\x')\label{eq:approx-kernel}
\end{equation}
Note that the conditions that $p(\cdot;\vtheta_{0})$ is even and
$\varphi(\x,\cdot)$ is even-symmetric, coupled with the fact that
the Gauss-Legendre quadrature is symmetric, ensures that $\tilde{k}_{\vtheta}(\x,\x')$
is always real.

Consider the first term in the right hand side of Eq.~(\ref{eq:approx-kernel}).
It is a bivariate function which can be written as a sum of $s$ separable
bivariate functions. Thus, we can informally view $s$ as the rank
of the decomposition, and if $s$ is small, then this is a low-rank
approximation. The parameterized family of approximate kernels $\{\tilde{k}_{\vtheta}\}_{\vtheta\in\Theta}$
is composed of kernels that can be written as a low-rank bivariate
function plus a ridge term. In the next subsection we show how to
utilize this low-rank structure in order efficiently perform Gaussian
process regression.

Of course, the crucial question is how do we choose $\matU=(U_{1},\dots,U_{d})$
and $\s=(s_{1},\dots,s_{d})$. We want to choose these parameters
such that $\{k_{\vtheta}\}_{\vtheta\in\Theta}$ and $\{\tilde{k}_{\vtheta}\}_{\vtheta\in\Theta}$
are $n$-spectrally equivalent over the domain $\X$, where $n$ is
the target dataset size (since GPR is nonparametric, the effective
rank of the kernel matrix goes to infinity when $n$ goes to infinity,
so it is impossible to approximate the kernel matrix well with a matrix
of fixed rank, i.e., with $s$ fixed, as $n$ goes to infinity). We
discuss this question in the next section. In the reminder of this
section, we discuss how to efficiently perform GPR using the approximate
kernel family $\{\tilde{k}_{\vtheta}\}_{\vtheta\in\Theta}$.

\subsection{\label{subsec:efficient-gpr}Efficient Gaussian Process Regression}

As a first and crucial step, we show how to write the kernel matrix
of $\tilde{k}_{\vtheta}$ as a low-rank matrix plus a ridge term.
Given a dataset $\x_{1},\dots,\x_{n}\in\X$, let $\matX$, as usual,
denote the matrix whose row $j$ is $\x_{j}^{\T}$. Define the matrix
\[
\matZ\in\C^{n\times s},\quad\matZ_{lj}\coloneqq\varphi(\x_{l},\veta_{j})\,.
\]
Note that $\matZ$ depends on $\matU$ and $\s$, but not on the hyperparameters
$\vtheta$. Next, define 
\[
\matW:\Theta\to\R_{+}^{s\times s},\quad\matW(\vtheta)\coloneqq\left[\begin{array}{ccc}
h_{1}(\vtheta_{0})\\
 & \ddots\\
 &  & h_{s}(\vtheta_{0})
\end{array}\right]\,.
\]
We now have
\[
\tilde{\matK}_{\vtheta}(\matX,\matX)=\sigma_{f}^{2}\matZ\matW(\vtheta)\matZ^{*}+\sigma_{n}^{2}\matI_{n}\,.
\]
Notice that dependence on $\vtheta$ is confined to the diagonal matrix
$\matW(\vtheta)$. This will be very helpful in deriving efficient
formulas for GPR.

We now discuss each of the various stages of GPR separately. For simplicity,
we assume that the GP prior has zero mean ($\mu=0$).

\paragraph{Training.}

Given $y_{1},\dots,y_{n}$, training usually amounts to computing
the vector
\[
\valpha\coloneqq\tilde{\matK}_{\vtheta}(\matX,\matX)^{-1}\y
\]
where $\y=[y_{1},\dots,y_{n}]^{\T}$. However, in our case, in order
to utilize the structure of $\tilde{\matK}_{\vtheta}(\matX,\matX)$,
we instead compute:
\[
\w\coloneqq\matZ^{*}\valpha=\matW(\vtheta)^{-1}(\sigma_{f}^{2}\mat Z^{*}\mat Z+\sigma_{n}^{2}\mat W(\vtheta)^{-1})^{-1}\mat Z^{*}\y\,.
\]
In the above, the second equality is a simple consequence of the Woodbury
matrix identity. Since $\matW(\vtheta)$ has positive diagonal, $\w$
can be computed using $O(ns^{2})$ operations, discounting the cost
of computing $\matZ$.

\paragraph{Prediction.}

Given a test set $\x_{1}^{(t)},\dots,\x_{t}^{(t)}$ which are distinct
from the training set, the predicted vector $\y^{(t)}=[y_{1}^{(t)},\dots,y_{t}^{(t)}]^{\T}$
is defined by $\y^{(t)}\coloneqq\tilde{\matK}_{\vtheta}(\matX^{(t)},\matX)\valpha$.
Let 
\[
\matZ^{(t)}\in\C^{t\times s},\quad\matZ_{lj}^{(t)}\coloneqq\varphi(\x_{l}^{(t)},\veta_{j})\,.
\]
Then have
\[
\tilde{\matK}_{\vtheta}(\matX^{(t)},\matX)=\sigma_{f}^{2}\matZ^{(t)}\matW(\vtheta)\matZ^{*}
\]
and
\begin{eqnarray*}
\y^{(t)} & = & \tilde{\matK}_{\vtheta}(\matX^{(t)},\matX)\valpha\\
 & = & \sigma_{f}^{2}\matZ^{(t)}\matW(\vtheta)\matZ^{*}\valpha\\
 & = & \sigma_{f}^{2}\matZ^{(t)}\matW(\vtheta)\w.
\end{eqnarray*}
Hence, once we have $\w$ (computed during training), we can compute
$\y^{(t)}$ using $O(ts)$ operations, discounting the cost of compute
$\matZ^{(t)}$.

\paragraph{Hyperparameter Learning.}

Hyperparameter learning amounts to finding the hyperparameters $\vtheta$
which maximize the log marginal likelihood. To do so, we need to be
to able to efficiently compute the log marginal likelihood, and its
gradient. It is well known that the likelihood is given by
\begin{equation}
{\cal L}(\mat{\theta})=-\frac{1}{2}\y^{\T}\tilde{\matK}_{\vtheta}(\matX,\matX)^{-1}\y-\frac{1}{2}\log\det\tilde{\matK}_{\vtheta}(\matX,\matX)-\frac{n}{2}\log2\pi\label{eq:L}
\end{equation}
and the derivatives are given by 
\begin{eqnarray}
\frac{\partial{\cal L}}{\partial\theta_{i}} & = & -\frac{1}{2}\Trace{\tilde{\matK}_{\vtheta}(\matX,\matX)^{-1}\frac{\partial\tilde{\matK}_{\vtheta}(\matX,\matX)}{\partial\theta_{i}}}+\frac{1}{2}\y^{\T}\tilde{\matK}_{\vtheta}(\matX,\matX)^{-1}\frac{\partial\tilde{\matK}_{\vtheta}(\matX,\matX)}{\partial\theta_{i}}\tilde{\matK}_{\vtheta}(\matX,\matX)^{-1}\y\label{eq:diff_L}
\end{eqnarray}
where $\theta_{i}$ represents an hyperparameter in $\vtheta$.
\begin{prop}
After an $O(ns^{2})$ preprocessing step of computing $\matZ^{*}\matZ$,
and discounting the cost of computing the partial derivatives of $p$
with respect to the hyperparamters, the log marginal likelihood ${\cal L}(\mat{\theta})$
and the gradient $\nabla{\cal L}(\mat{\theta})$ can be computed in
$O(ns+s^{3}+s|\vtheta|)$ arithmetic operations, where $|\vtheta|$
represents the number of hyperparameters in $\vtheta$. Furthermore,
the amount of memory storage required is $O(s^{2}).$
\end{prop}

\begin{proof}
First, let us consider the computation of the likelihood. For the
first term in Eq.~(\ref{eq:L}), note that 
\[
\valpha(\vtheta)=\tilde{\matK}_{\vtheta}(\matX,\matX)^{-1}\y=\sigma_{n}^{-2}(\y-\sigma_{f}^{2}\matZ\matW(\vtheta)\w(\vtheta))
\]
where $\w(\vtheta)=\matW(\vtheta)^{-1}(\sigma_{f}^{2}\mat Z^{*}\mat Z+\sigma_{n}^{2}\mat W(\vtheta)^{-1})^{-1}\mat Z^{*}\y$.
In the previous equations, we made the dependence of $\valpha$ and
$\w$ on $\vtheta$ explicit. Obviously, once $\matZ^{*}\matZ$ has
been computed (an $O(ns^{2})$ preprocessing step), we can compute
both $\w(\vtheta)$ and $\valpha(\vtheta)$ in $O(ns+s^{3})$. The
first term in Eq.~(\ref{eq:L}) is now equal to $-\y^{\T}\valpha(\vtheta)/2$.

For the second term in Eq.~(\ref{eq:L}), using the matrix determinant
lemma, we have 
\[
\log\det\tilde{\matK}_{\vtheta}(\matX,\matX)=n\log\sigma_{n}^{2}+s\log\sigma_{f}^{2}+\log\det\matW(\vtheta)+\log\det(\sigma_{f}^{-2}\matW(\vtheta)^{-1}+\sigma_{n}^{-2}\matZ^{*}\matZ)\,.
\]
Since $\matW(\vtheta)$ is diagonal, once $\matZ^{*}\matZ$ has been
computed, we can compute this term in $O(s^{3}$) operations.

Next, let us consider the computation of each derivative of the likelihood
according to Eq.~(\ref{eq:diff_L}). The crucial observations are:
\[
\frac{\partial\tilde{\matK}_{\vtheta}(\matX,\matX)}{\partial\sigma_{f}^{2}}=\matZ\matW(\vtheta)\matZ^{*},\quad\frac{\partial\tilde{\matK}_{\vtheta}(\matX,\matX)}{\partial\sigma_{n}^{2}}=\matI_{n},\quad\frac{\partial\tilde{\matK}_{\vtheta}(\matX,\matX)}{\partial\theta_{i}}=\sigma_{f}^{2}\matZ\frac{\partial\matW(\vtheta)}{\partial\theta_{i}}\matZ^{*}
\]
(using $\theta_{i}$ to denote an hyperparameter in $\vtheta_{0}$)
where using last equality amounts to computing partial derivative
$\partial p(\veta_{j};\vtheta_{0})/\partial\theta_{i}$ for $j=1,\dots,s$.
For the second term in Eq.~(\ref{eq:diff_L}) we have, 
\begin{eqnarray*}
\valpha(\vtheta)^{\T}\frac{\partial\tilde{\matK}_{\vtheta}(\matX,\matX)}{\partial\sigma_{f}^{2}}\valpha(\vtheta) & = & \w(\vtheta)^{\T}\matW(\vtheta)\w(\vtheta)\\
\valpha(\vtheta)^{\T}\frac{\partial\tilde{\matK}_{\vtheta}(\matX,\matX)}{\partial\sigma_{n}^{2}}\valpha(\vtheta) & = & \TNormS{\valpha(\vtheta)}\\
\valpha(\vtheta)^{\T}\frac{\partial\tilde{\matK}_{\vtheta}(\matX,\matX)}{\partial\theta_{i}}\valpha(\vtheta) & = & \sigma_{f}^{2}\w(\vtheta)^{\T}\frac{\partial\matW(\vtheta)}{\partial\theta_{i}}\w(\vtheta)
\end{eqnarray*}
so this term can be computed in $O(s)$ operations once we compute
$\w(\vtheta)$ and $\valpha(\vtheta)$ (which are computed during
the computation of the likelihood). For the first term in Eq.~(\ref{eq:diff_L}),
let 
\[
\mat F(\vtheta)\coloneqq\sigma_{f}^{2}(\sigma_{f}^{2}\mat{\mat Z}^{*}\mat{\mat Z}+\sigma_{n}^{2}\mat W(\vtheta)^{-1})^{-1}\mat{\mat Z}^{*}\mat{\mat Z}\,.
\]
Again, once $\matZ^{*}\matZ$ has been computed, $\mat F(\vtheta)$
can be computed in $O(s^{3})$ operations. Now, using the Woodbury
formula and cyclicality of the trace, we have
\begin{eqnarray}
\Trace{\tilde{\matK}_{\vtheta}(\matX,\matX)^{-1}\frac{\partial\tilde{\matK}_{\vtheta}(\matX,\matX)}{\partial\sigma_{f}^{2}}} & = & \sigma_{f}^{-2}\Trace{\mat F(\vtheta)}\label{eq:lik_trace}\\
\Trace{\tilde{\matK}_{\vtheta}(\matX,\matX)^{-1}\frac{\partial\tilde{\matK}_{\vtheta}(\matX,\matX)}{\partial\sigma_{n}^{2}}} & = & \sigma_{n}^{-2}(n-\Trace{\mat F(\vtheta)})\nonumber \\
\Trace{\tilde{\matK}_{\vtheta}(\matX,\matX)^{-1}\frac{\partial\tilde{\matK}_{\vtheta}(\matX,\matX)}{\partial\theta_{i}}} & = & \sigma_{n}^{-2}\sigma_{f}^{2}\Trace{\frac{\partial\matW(\vtheta)}{\partial\theta_{i}}\mat{\mat Z}^{*}\mat{\mat Z}}-\sigma_{n}^{-2}\sigma_{f}^{2}\Trace{\frac{\partial\matW(\vtheta)}{\partial\theta_{i}}\mat{\mat Z}^{*}\mat{\mat Z}\mat F(\vtheta)}\nonumber 
\end{eqnarray}
(the calculations leading to these formulas appear in Appendix~\ref{sec:Accelerating-Hyperparametes}).
Thus, once $\mat Z^{*}\mat Z$, $\mat F(\vtheta)$ and the diagonal
of $\mat{\mat Z}^{*}\mat{\mat Z}\mat F(\vtheta)$ have been computed,
all these computations can be done in $O(s|\vtheta|)$. Note that
the diagonal of $\mat{\mat Z}^{*}\mat{\mat Z}\mat F(\vtheta)$ can
be computed using $O(s^{2})$ operations once we have $\mat Z^{*}\mat Z$
and $\mat F(\vtheta)$.

In terms of memory storage, notice that once $\matZ^{\conj}\matZ$
and $\matZ^{\conj}\y$ have been computed there is no longer any need
for $\matZ$ (which requires $O(ns)$ words to store). However, in
order to compute $\matZ^{\conj}\matZ$ and $\matZ^{\conj}\y$ we do
not need to form all of $\matZ$ in memory, but rather can stream
over the training set transforming every training point using $\varphi$
and accumulating its contribution to $\matZ^{\conj}\matZ$ and $\matZ^{\conj}\y$.
Thus, the dominant storage cost is for holding $\matZ^{\conj}\matZ$
which is $O(s^{2})$.
\end{proof}
The computations can be performed more stably by utilizing various
matrix identities. We delegate the details to Appendix~\ref{sec:Accelerating-Hyperparametes}.

To have a computational advantage in the training and prediction
steps we need $s=o(n).$ However, since for most kernels computing
the gradient of the likelihood requires $O(n^{3}|\vtheta|)$, our
method has a computational advantage in the hyperparameter learning
phase even if $s=\Theta(n)$.

\subsection{\label{subsec:rff-problematic}Comparison to Other Methods}

Our proposed method is very much inspired by the Random Fourier Features
(RFF) method~\cite{rahimi2008random}. Although originally defined
only for shift-invariant kernels, the method can be easily generalized
for kernels of the form of Eq.~(\ref{eq:kernel-form}). We refer
to the generalization as \emph{Monte-Carlo Features }(MCF). RFF is
a special case of MCF. In particular, given some fixed parameters
$\vtheta$, a MCF approximate kernel is
\[
\tilde{k}_{\vtheta}(\x,\x')=\frac{\sigma_{f}^{2}}{s}\sum_{j=1}^{s}\varphi(\x,\veta_{j})\varphi(\x',\veta_{j})^{*}+\sigma_{n}^{2}\gamma(\x,\x')
\]
where $\veta_{1},\dots,\veta_{s}$ are sampled from the density function
$p(\veta;\vtheta_{0})$. As before, we include the ridge term $\sigma_{n}^{2}\gamma(\x,\x')$
in the kernel definition. Thus, the kernel matrix approximation is
\[
\tilde{\matK}_{\vtheta}^{\text{(MCF)}}(\matX,\matX)=\sigma_{f}^{2}\matZ(\vtheta)\matZ(\vtheta)^{\conj}+\sigma_{n}^{2}\matI_{n}
\]
where
\[
\matZ(\vtheta)\in\C^{n\times s},\quad\matZ(\vtheta)_{lj}\coloneqq\varphi(\x_{l},\veta_{j})
\]
Obviously, training and prediction can be efficiently executed by
utilizing the identity plus low-rank structure of $\tilde{\matK}_{\vtheta}^{\text{(MCF)}}$
much in the same way as we have done for Gauss-Legendre features,
and indeed this is the reason the method was developed~~\cite{rahimi2008random}.

The above developments were for a \emph{fixed} $\vtheta$. However,
it is less clear how to define a \emph{family }of approximations for
various $\vtheta$, and perform hyperparameter learning. A key issue
is that the kernel approximation should vary smoothly with $\vtheta$,
so obviously fresh $\veta_{1},\dots,\veta_{s}$ cannot be sampled
differently for every $\vtheta$. It is outside the scope of this
paper to consider how to use MCF to define parameterized families
of kernel approximation suitable for hyperparameter learning. Nevertheless,
since we wish to use RFF as a baseline for complexity comparisons
and numerical experiments, we show how it is possible to use the specific
case of RFF to form parameterized families of kernel approximation
and perform hyperparameter learning for a restricted family of kernels
that includes the Gaussian and Matèrn kernels.

Specifically, we will consider a restricted class of shift-invariant
whose kernel has the following specific form:
\begin{equation}
k_{\vtheta}(\x,\x')=\sigma_{f}^{2}k_{0}(\matL^{-1}(\x-\x'))+\sigma_{n}^{2}\gamma(\x-\x')\label{eq:crf-kernels}
\end{equation}
where $k_{0}(\cdot)$ is a positive definite function, and $\matL$
is diagonal with positive entries $L_{1},\dots,L_{d}$ which are part
of parameter vector $\vtheta$ (i.e., $\vtheta=[L_{1},\dots L_{d},\sigma_{f}^{2},\sigma_{n}^{2}]$).
Note that the number of variables in $\vtheta_{0}$ is equal to the
dimension $d$. Gaussian and Matèrn kernels are examples of such kernels.
In this case we can write $k_{\vtheta}$ in the form of Eq.~(\ref{eq:kernel-form}),
where $\varphi(\x,\veta)=e^{-i\x^{\T}\veta}$. We further assume that
$p(\veta;\matL)$ is such that sampling a random vector $\veta$ is
the same as sampling from the distribution defined by $p(\cdot;\matI_{d})$
and scaling the vector by $\matL^{-1}$. Thus, for such kernels we
can sample $\veta_{1},\dots,\veta_{s}$ once from $p(\cdot;\matI_{d}$)
and view any change in $\vtheta_{0}=\matL$ as corresponding change
in $\veta_{1},\dots,\veta_{s}$. Concretely, letting 
\[
\matW=\left[\begin{array}{ccc}
\veta_{1} & \dots & \veta_{s}\end{array}\right]\,,
\]
the feature matrix is

\[
\matZ(\matL)=\frac{1}{\sqrt{s}}\exp\left(-i\matX\matL^{-1}\matW\right)
\]
and the kernel matrix is 
\[
\tilde{\matK}_{\vtheta}^{\text{(RFF)}}(\matX,\matX)=\sigma_{f}^{2}\matZ(\matL)\matZ(\matL)^{\conj}+\sigma_{n}^{2}\matI_{n}\,.
\]
However, the crucial point is that now $\matZ(\matL)$ depends smoothly
on $\matL$, so we can compute gradients. Formulas quite similar to
the ones derived in the previous section can be derived (we omit most
details), with the main difference being in taking the derivative
of the kernel matrix with respect to the parameters in $\matL$. Here
we have
\[
\frac{\partial\tilde{\matK}_{\vtheta}^{\text{(RFF)}}(\matX,\matX)}{\partial L_{k}}=\sigma_{f}^{2}\frac{\partial(\matZ\matZ^{*})(\matL)}{\partial L_{k}}=\sigma_{f}^{2}\left(\frac{\partial\matZ(\matL)}{\partial L_{k}}\matZ(\matL)^{*}+\matZ\frac{\partial\matZ(\matL)^{\conj}}{\partial L_{k}}\right)
\]
\begin{eqnarray*}
\frac{\partial\matZ(\matL)}{\partial L_{k}} & = & -\frac{i}{\sqrt{s}}\left(\matX\frac{\partial\matL^{-1}}{\partial L_{k}}\matW\odot\exp\left(-i\cdot\matX\matL^{-1}\matW\right)\right)=-i\left(\matX\frac{\partial\matL^{-1}}{\partial L_{k}}\matW\odot\matZ(\matL)\right)\\
\frac{\partial\matZ(\matL)^{\conj}}{\partial L_{k}} & = & \frac{i}{\sqrt{s}}\left(\matW^{\T}\frac{\partial\matL^{-1}}{\partial L_{k}}\matX^{\T}\odot\exp\left(i\cdot\matX\matL^{-1}\matW\right)\right)=i\left(\matW^{\T}\frac{\partial\matL^{-1}}{\partial L_{k}}\matX^{\T}\odot\matZ(\matL)^{\conj}\right)=i\left(\matX\frac{\partial\matL^{-1}}{\partial L_{k}}\matW\odot\matZ(\matL)\right)^{*}
\end{eqnarray*}

In terms of complexity, the main difference between Gauss-Legendre
Features and RFF is that for the former the matrix $\matZ^{\conj}\matZ$
stays constant when $\vtheta$ varies, and so the product can be computed
once, while for the latter $\matZ(\matL)^{\conj}\matZ(\matL)$ varies,
and so changes every iteration. This adds an additional cost of $O(ns^{2})$
operations for every gradient computation. Furthermore, we need to
compute $\partial\matZ(\matL)/\partial L_{k}$ for $k=1,\dots.,d$
in each iteration, each costing $O(nsd)$ operations, for a total
of $O(nsd^{2})$ operations. Furthermore, since $\matZ(\matL)$ changes
in each iteration, and $\matZ(\matL)$ features in many of the equations,
we cannot compute the matrix $\matZ(\matL)^{\conj}\matZ(\matL)$ once
and reduce storage costs to $O(s^{2})$, and using $\matZ(\matL)$
implicitly many times will incur a large overhead. Thus, for RFF the
storage cost is $O(ns)$,

\begin{table}[H]
\caption{\label{tab:GLF_vs_CRF}Computational complexities (arithmetic operations)
comparison between Gauss-Legendre Features and Random Fourier Features
for kernels of the form of Eq.~(\ref{eq:crf-kernels}) (e.g., non-isotropic
Gaussian an Matèrn kernels). In the table, $n$ is the size of the
training set, $t$ is the size of the test set, $s$ is the approximation
rank, and $I$ is the number of gradient computations for hyperparameter
learning (e.g., number of gradient descent iterations).}

\centering{}%
\begin{tabular}{|>{\centering}p{2.5cm}|>{\centering}p{5cm}|c|}
\hline 
 & {\scriptsize{}Gauss-Legendre Features} & {\scriptsize{}Random Fourier Features}\tabularnewline
\hline 
\hline 
{\scriptsize{}Training} & $O(ns^{2})$ & $O(ns^{2})$\tabularnewline
\hline 
{\scriptsize{}Prediction} & $O(nt)$ & $O(nt)$\tabularnewline
\hline 
{\scriptsize{}Hyperparameter learning:} & $O(ns^{2}+I(ns+s^{3}+sd)$) & $O(I(ns^{2}+nsd^{2}))$\tabularnewline
\hline 
\end{tabular}
\end{table}

A similar issue will likely arise when using features based on Gaussian
quadrature, like was suggested in~\cite{dao2017gaussian} (that paper
does not discusses GPR hyperparameter learning). When $\vtheta$ changes,
the distribution that defines Gaussian quadrature changes. Unlike
Guass-Legendre features which uses fixed nodes, for Gaussian quadrature
features the quadrature nodes change with $\vtheta$, which prevents
the use of a fixed feature matrix $\matZ(\vtheta)$. Furthermore,
when performing hyperparameter learning with Gaussian quadrature features,
we need to not only compute the quadrature weights but also compute
their derivatives.

\section{\label{sec:parameters}Parameter Computation}

In order to complete the description of our method, we need to specify
how to choose $\matU$ and $\s$. First, we show how to set $\matU$
and $\s$ for a fixed $\vtheta\in\Theta$ and $n$ such that $k_{\vtheta}$
and $\tilde{k}_{\vtheta}$ are $n$-spectrally equivalent. We then
consider how to set a fixed $\matU$ and $\s$ such that the two families
$\{k_{\vtheta}\}_{\vtheta\in\Theta}$ and $\{\tilde{k}_{\vtheta}\}_{\vtheta\in\Theta}$
are $n$-spectrally equivalent.

\subsection{From Matrix Approximation to Integral Approximation}

For now (and until subsection \ref{subsec:domain}) let us assume
that $\vtheta$ is fixed. Our goal is to set $\matU$ and $\s$ such
that $k_{\vtheta}$ and $\tilde{k}_{\vtheta}$ are $n$-spectrally
equivalent, i.e. for every dataset $\x_{1},\dots,\x_{n}\in\X$ 
\[
(1-n^{-1})\matK_{\vtheta}(\matX,\matX)\preceq\tilde{\matK}_{\vtheta}(\matX,\matX)\preceq(1+n^{-1})\matK_{\vtheta}(\matX,\matX)
\]
In other words, we want to set $\matU$ and $\s$ such that for every
$\v\in\R^{n}$, 
\begin{equation}
(1-n^{-1})\v^{\T}\matK_{\vtheta}(\matX,\matX)\v\leq\v^{\T}\tilde{\matK}_{\vtheta}(\matX,\matX)\v\leq(1+n^{-1})\v^{\T}\matK_{\vtheta}(\matX,\matX)\v\label{eq:spectral}
\end{equation}
Henceforth, for conciseness, we drop $\matX$ from the expressions,
although the various expressions implicitly depend on $\matX$.

Let 
\[
\z(\veta)\coloneqq\left[\begin{array}{c}
\varphi(\x_{1},\veta)\\
\vdots\\
\varphi(\x_{n},\veta)
\end{array}\right]\,.
\]
Then, 
\[
\v^{\T}\matK_{\vtheta}\v=\sigma_{f}^{2}\int_{\R^{d}}|\z(\veta)^{*}\v|^{2}p(\veta;\vtheta_{0})d\veta+\sigma_{n}^{2}\TNormS{\v}
\]
and 
\[
\v^{\T}\tilde{\matK}_{\vtheta}\v=\sigma_{f}^{2}\sum_{j=1}^{s}h_{j}(\vtheta_{0})|\z(\veta_{j})^{*}\v|^{2}+\sigma_{n}^{2}\TNormS{\v}
\]
Since rescaling $\v$ rescales all the terms in the previous inequality,
we can assume without loss of generality that $\v^{\T}\matK_{\vtheta}\v=1$.
In that case, Eq.~(\ref{eq:spectral}) is equivalent to 
\begin{equation}
\left|\int_{\R^{d}}|\z(\veta)^{*}\v|^{2}p(\veta;\vtheta_{0})d\veta-\sum_{j=1}^{s}h_{j}(\vtheta_{0})|\z(\veta_{j})^{*}\v|^{2}\right|\leq\frac{1}{\sigma_{f}^{2}n}\label{eq:integral-approximation}
\end{equation}
Thus, the nodes $\veta_{1},\dots,\veta_{s}$ and weights $h_{1}(\vtheta_{0})\dots,h_{s}(\vtheta_{0})$
function as a quadrature approximation.

\subsection{Truncating the Integral}

As alluded earlier, we approach the quadrature approximation~Eq.~(\ref{eq:integral-approximation})
by first truncating the integral and then using a Gauss-Legendre quadrature
for the truncated integral. In other words, we write 
\begin{align}
\left|\int_{\R^{d}}|\z(\veta)^{*}\v|^{2}p(\veta;\vtheta_{0})d\veta-\sum_{j=1}^{s}h_{j}(\vtheta_{0})|\z(\veta_{j})^{*}\v|^{2}\right|\leq & \left|\int_{\R^{d}}|\z(\veta)^{*}\v|^{2}p(\veta;\vtheta_{0})d\veta-\int_{\Q_{\matU}}|\z(\veta)^{*}\v|^{2}p(\veta;\vtheta_{0})d\veta\right|+\nonumber \\
 & \quad\left|\int_{\Q_{\matU}}|\z(\veta)^{*}\v|^{2}p(\veta;\vtheta_{0})d\veta-\sum_{j=1}^{s}h_{j}(\vtheta_{0})|\z(\veta_{j})^{*}\v|^{2}\right|\label{eq:integral-split}
\end{align}
where $\Q_{\matU}=\prod_{k=1}^{d}[-U_{k},U_{k}]$. We set $\matU$
such that the first term is smaller than $\sigma_{f}^{-2}n^{-1}/2$,
and set each of the components in $\s$ to be large enough so that
the second term is also smaller than $\sigma_{f}^{-2}n^{-1}/2$.

Obviously, we want to set the components in $\matU$ to be as small
as possible, to limit the integration area. Having a smaller integration
area allows us to use smaller values in $\s$. The minimal values
in $\matU$ such that the first term is bounded by $\sigma_{f}^{-2}n^{-1}/2$
depends on how quickly $p(\veta;\vtheta_{0}$) decays as $\InfNorm{\veta}\to\infty$:
the faster the density decays, the smaller is the region where the
function value has significant contribution. Therefore, in our analysis,
we distinguish between four classes of decay of $p(\cdot;\vtheta_{0}$),
and analyze each on its own.

For a diagonal $\matL$, with positive entries $L_{1},\dots,L_{d}$
on the diagonal (i.e., $\matL=\diag{L_{1},\dots,L_{d}}$), let us
define:
\[
{\cal P}_{C,\matL}\coloneqq\left\{ p:\R^{d}\to\R_{+}\,\,\text{measurable}\,|\,\int_{\R^{d}}p(\mat{\eta})d\mat{\eta}=1,\,p(\mat{\eta})\leq C\cdot\prod_{k=1}^{d}\frac{1}{1+L_{k}^{2}\eta_{k}^{2}}\right\} 
\]
\[
{\cal P}_{C,\matL}^{(r)}\coloneqq\left\{ p:\R^{d}\to\R_{+}\,\,\text{measurable}\,|\,\int_{\R^{d}}p(\mat{\eta})d\mat{\eta}=1,\,p(\mat{\eta})\leq C\cdot\left(1+\TNormS{\matL\veta}\right)^{-r}\right\} 
\]

\[
{\cal E}_{C,\matL}^{(1)}\coloneqq\left\{ p:\R^{d}\to\R_{+}\,\,\text{measurable}\,|\,\int_{\R^{d}}p(\mat{\eta})d\mat{\eta}=1,\,p(\mat{\eta})\leq C\cdot e^{-\ONorm{\matL\veta}1}\right\} 
\]

\[
{\cal E}_{C,\matL}^{(2)}\coloneqq\left\{ p:\R^{d}\to\R_{+}\,\,\text{measurable}\,|\,\int_{\R^{d}}p(\mat{\eta})d\mat{\eta}=1,\,p(\mat{\eta})\leq C\cdot e^{-\TNormS{\matL\veta}}\right\} 
\]
The families ${\cal P}_{C,\matL}$ and ${\cal P}_{C,\matL}^{(r)}$
include densities that decay at a polynomial rate or faster, while
${\cal E}_{C,\matL}^{(1)}$ includes densities that decay at exponential
rate or faster, and ${\cal E}_{C,\matL}^{(2)}$ includes densities
that decay at a square exponential rate or faster. Table~\ref{tab:decay-examples}
shows four well known kernels, their corresponding densities, and
their decay class.

\begin{table}[t]
\caption{\label{tab:decay-examples}Example of decay classes for a few well
known kernel function. In the table below, $\protect\matL$ is a diagonal
matrix with non-negative diagonal entries.}

\centering{}%
\begin{tabular}{|>{\centering}p{1.5cm}|>{\centering}p{3.5cm}|c|c|}
\hline 
 & {\scriptsize{}$k_{\vtheta}(\x,\x')$} & {\scriptsize{}$p(\veta;\vtheta_{0})$} & {\scriptsize{}Decay Class}\tabularnewline
\hline 
\hline 
{\scriptsize{}Non-isotropic Gaussian} & {\scriptsize{}$\exp(-\TNormS{\matL^{-1}(\x-\x')}//2)$} & {\scriptsize{}$\begin{aligned}(2\pi)^{-d/2}\prod_{k=1}^{d}\ell_{k}\exp(-\TNormS{\mat L\veta})\\
L_{k}=\ell_{k}/\sqrt{2}
\end{aligned}
$} & {\scriptsize{}${\cal E}_{C,\matL}^{(2)},\quad C=(2\pi)^{-d/2}\prod_{k=1}^{d}\ell_{k}$}\tabularnewline
\hline 
{\scriptsize{}Non-isotropic Cauchy} & {\scriptsize{}$2^{d}\prod_{k=1}^{d}\frac{\ell_{k}}{\ell_{k}^{2}+(\x-\x')_{k}^{2}}$} & {\scriptsize{}$\begin{aligned}\exp(-\ONorm{\mat L\veta}1)\\
L_{k}=\ell_{k}
\end{aligned}
$} & {\scriptsize{}${\cal E}_{C,\matL}^{(1)},\quad C=1$}\tabularnewline
\hline 
{\scriptsize{}Non-isotropic Laplacian} & {\scriptsize{}$\exp(-\ONorm{\matL^{-1}(\x-\x')}1)$} & {\scriptsize{}$\begin{aligned}\pi^{-d}\prod_{k=1}^{d}\frac{\ell_{k}}{1+\ell_{k}^{2}\eta_{k}^{2}}\\
L_{k}=\ell_{k}
\end{aligned}
$} & {\scriptsize{}${\cal P}_{C,\mat L},\quad C=\pi^{-d}\prod_{k=1}^{d}\ell_{k}$}\tabularnewline
\hline 
{\scriptsize{}Matèrn} & {\scriptsize{}$\frac{2^{1-\nu}}{\Gamma(\nu)}(\sqrt{2\nu}\TNorm{\matL^{-1}(\x-\x')})^{\nu}\cdot$$K_{\nu}(\sqrt{2\nu}\TNorm{\matL^{-1}(\x-\x')})$} & {\scriptsize{}$\begin{aligned}\frac{\Gamma(\nu+d/2)}{\pi^{d/2}\Gamma(\nu)(2\nu)^{d/2}}\prod_{k=1}^{d}\ell_{k}\left(1+\TNormS{\matL\veta}\right)^{-(\nu+d/2)}\\
L_{k}=\ell_{k}/\sqrt{2\nu}
\end{aligned}
$} & {\scriptsize{}$\begin{aligned}{\cal P}_{C,\matL}^{(r)}\\
r=\nu+d/2\\
C=\frac{\Gamma(\nu+d/2)}{\Gamma(\nu)(2\pi\nu)^{d/2}}\prod_{k=1}^{d}\ell_{k}
\end{aligned}
$}\tabularnewline
\hline 
\end{tabular}
\end{table}

The following proposition specifics how to set $\matU$ based on the
decay class and the maximum value of $\varphi$.
\begin{prop}
\label{prop:setting-U}Suppose that, $M_{R}$ is such that for every
$\veta\in\R^{d}$ and every $\x\in\X$ we have $|\varphi(\x,\veta)|\leq M_{R}$.
Then, the following establishes a $\matU^{(\min)}=(U_{1}^{(\min)},\dots,U_{d}^{(\min)})$
such that if $\matU\geq\matU^{(\min)}$ (where we interpret the inequality
as entrywise) then we have 
\begin{equation}
\left|\int_{\R^{d}}|\z(\veta)^{*}\v|^{2}p(\veta;\vtheta_{0})d\veta-\int_{\Q_{\matU}}|\z(\veta)^{*}\v|^{2}p(\veta;\vtheta_{0})d\veta\right|\leq\frac{1}{2\sigma_{f}^{2}n}\label{eq:integral-split_1st_term}
\end{equation}
for every $\v$ such that $\v^{\T}\matK_{\vtheta}\v=1$.
\begin{enumerate}
\item If $p(\cdot;\vtheta_{0})\in{\cal P}_{C,\mat L}$, $U_{k}^{(\min)}\coloneqq\frac{1}{L_{k}}\cot\left(L_{k}\left(\frac{4CM_{R}^{2}\sigma_{f}^{2}n^{2}}{\sigma_{n}^{2}}\right)^{-1/d}\right)$.
\item If $p(\cdot;\vtheta_{0})\in{\cal P}_{C,\mat L}^{(r)}$ and $r>d/2$:
\begin{enumerate}
\item for $d=2$, set $U_{k}^{(\min)}\coloneqq\frac{1}{L_{k}}\sqrt{\sqrt[r-1]{\frac{\pi CM_{R}^{2}\sigma_{f}^{2}n^{2}}{(r-1)\sigma_{n}^{2}L_{1}L_{2}}}-1}$ .
\item for any $d\neq2$, let $x>0$ be the solution to the equation 
\begin{equation}
\frac{\pi^{d/2}CnM_{R}^{2}}{2^{d-2}\Gamma\left(\frac{d}{2}\right)(2r-d)\sigma_{n}^{2}\prod_{k=1}^{d}L_{k}}x^{d-2r}\left|\ensuremath{_{2}F_{1}}\left(r-d/2,r;r-d/2+1;-x^{-2}\right)\right|=\frac{1}{2\sigma_{f}^{2}n}\,.\label{eq:Matern_highD_eq}
\end{equation}
($_{2}F_{1}$ is the hypergeometric function). Then set $U_{k}^{(\min)}\coloneqq x/L_{k}$
. We have 
\[
U_{k}^{(\min)}\leq\frac{1}{L_{k}}\left(\frac{\pi^{d/2}CM_{R}^{2}\sigma_{f}^{2}n^{2}}{2^{d-2}\Gamma\left(\frac{d}{2}\right)(2r-d)\sigma_{n}^{2}\prod_{k=1}^{d}L_{k}}\right)^{1/(2r-d)}\,.
\]
\end{enumerate}
\item If $p(\cdot;\vtheta_{0})\in{\cal E}_{C,\matL}^{(1)}$, $U_{k}^{(\min)}\coloneqq\frac{1}{L_{k}}\ln\left(\frac{1}{L_{k}}\left(\frac{4CM_{R}^{2}\sigma_{f}^{2}n^{2}}{\sigma_{n}^{2}}\right)^{1/d}\right)$.
\item If $p(\cdot;\vtheta_{0})\in{\cal E}_{C,\matL}^{(2)}$, $U_{k}^{(\min)}\coloneqq\frac{1}{L_{k}}\sqrt{\ln\left(\frac{\sqrt{\pi}}{L_{k}}\left(\frac{2^{2-d}CM_{R}^{2}\sigma_{f}^{2}n^{2}}{\sigma_{n}^{2}}\right)^{1/d}\right)}$.
\end{enumerate}
\end{prop}

\begin{rem}
We recommend to set $\matU$ to $\matU^{(\min)}$ . For ${\cal P}_{C,\mat L}$,
${\cal E}_{C,\matL}^{(1)}$ and ${\cal E}_{C,\matL}^{(2)}$ we give
explicit formulas for $\matU^{(\min)}$. For ${\cal P}_{C,\mat L}^{(r)}$,
it is defined implicitly as the solution to a nonlinear equation.
We recommend finding the solution numerically using root-finding methods.
We also give an explicit upper bound for the value of $\matU^{(\min)}$,
which can be used if one wishes to avoid solving a non-linear equation,
however those upper bounds tend to be loose. Nevertheless, the upper
bound is used later to derive asymptotic bounds on $s$.
\end{rem}

\begin{proof}
First, since $\v^{\T}\matK_{\vtheta}\v=1$ , we have
\begin{eqnarray}
|\z(\veta)^{*}\v|^{2} & = & \left|\z(\veta)^{*}\matK_{\vtheta}^{-1/2}\matK_{\vtheta}^{1/2}\v\right|^{2}\nonumber \\
 & \leq & (\z(\veta)^{*}\matK_{\vtheta}^{-1}\z(\veta))\cdot(\v^{\T}\matK_{\vtheta}\v)\nonumber \\
 & = & \z(\veta)^{*}\matK_{\vtheta}^{-1}\z(\veta)\label{eq:uni_bound_f}\\
 & \leq & \sigma_{n}^{-2}\|\z(\veta)\|_{2}^{2}\nonumber \\
 & \leq & n\sigma_{n}^{-2}M_{R}^{2}\nonumber 
\end{eqnarray}
where the first inequality is due to the Cauchy-Schwartz inequality,
the second inequality follows from observing that the smallest eigenvalue
of $\matK_{\vtheta}$ is bigger than or equal to $\sigma_{n}^{2}$,
and the last inequality is due to the fact that every entry in $\z(\veta)$
has absolute value that is smaller or equal to $M_{R}$.

\paragraph*{The case of $p\in{\cal P}_{C,\protect\mat L}$:}

In this case,
\begin{eqnarray*}
\left|\int_{\R^{d}}|\z(\veta)^{*}\v|^{2}p(\veta;\vtheta_{0})d\veta-\int_{\Q_{U}}|\z(\veta)^{*}\v|^{2}p(\veta;\vtheta_{0})d\veta\right| & = & \left|\int_{\left|\mat{\eta}\right|\geq\mat U}|\z(\mat{\eta})^{\conj}\v|^{2}p(\mat{\eta})d\mat{\eta}\right|\\
 & \leq & \frac{2CnM_{R}^{2}}{\sigma_{n}^{2}}\left|\int_{U_{1}}^{\infty}\ldots\int_{U_{d}}^{\infty}\prod_{k=1}^{d}\frac{1}{1+L_{k}^{2}\eta_{k}^{2}}d\eta_{1}\cdots d\eta_{d}\right|\\
 & = & \frac{2CnM_{R}^{2}}{\sigma_{n}^{2}}\prod_{k=1}^{d}\left|\int_{U_{k}}^{\infty}\frac{1}{1+L_{k}^{2}\eta_{k}^{2}}d\eta_{k}\right|\\
 & = & \frac{2CnM_{R}^{2}}{\sigma_{n}^{2}}\prod_{k=1}^{d}\frac{1}{L_{k}}\left(\frac{\pi}{2}-\arctan(L_{k}U_{k})\right)\,.
\end{eqnarray*}
So, in order for Eq.~(\ref{eq:integral-split_1st_term}) to hold,
we set  
\[
U_{k}^{(\min)}=\frac{1}{L_{k}}\tan\left(\frac{\pi}{2}-L_{k}\left(\frac{4CM_{R}^{2}\sigma_{f}^{2}n^{2}}{\sigma_{n}^{2}}\right)^{-1/d}\right)=\frac{1}{L_{k}}\cot\left(L_{k}\left(\frac{4CM_{R}^{2}\sigma_{f}^{2}n^{2}}{\sigma_{n}^{2}}\right)^{-1/d}\right)
\]

\paragraph*{The case of $p\in{\cal P}_{C,\protect\mat L}^{(r)}$:}

In this case, {\scriptsize{}
\begin{eqnarray}
\left|\int_{\mathbb{R}^{d}}\left|\z(\mat{\eta})^{\conj}\v\right|^{2}p(\mat{\eta})d\mat{\eta}-\int_{{\cal Q}_{\mat U}}\left|\z(\mat{\eta})^{\conj}\v\right|^{2}p(\mat{\eta})d\mat{\eta}\right| & = & \left|\int_{\left|\mat{\eta}\right|\geq\mat U}|\z(\mat{\eta})^{\conj}\v|^{2}p(\mat{\eta})d\mat{\eta}\right|\label{eq:Matern_truncate}\\
 & \leq & \frac{2CnM_{R}^{2}}{\sigma_{n}^{2}}\left|\int_{U_{1}}^{\infty}\ldots\int_{U_{d}}^{\infty}\left(1+L_{1}^{2}\eta_{1}^{2}+\ldots+L_{d}^{2}\eta_{d}^{2}\right)^{-r}d\eta_{1}\cdots d\eta_{d}\right|\nonumber \\
 & = & \frac{2CnM_{R}^{2}}{\sigma_{n}^{2}\prod_{k=1}^{d}L_{k}}\left|\int_{L_{1}U_{1}}^{\infty}\ldots\int_{L_{d}U_{d}}^{\infty}\left(1+\eta_{1}^{2}+\ldots+\eta_{d}^{2}\right)^{-r}d\eta_{1}\cdots d\eta_{d}\right|\nonumber \\
 & \leq & \frac{2CnM_{R}^{2}}{\sigma_{n}^{2}\prod_{k=1}^{d}L_{k}}\left|\int_{0}^{\pi/2}\int_{m}^{\infty}\int_{0}^{\pi/2}\ldots\int_{0}^{\pi/2}\left(\frac{t^{d-1}}{\left(1+t^{2}\right)^{r}}\prod_{j=1}^{d-2}\sin^{j}\phi_{d-1-j}\right)d\phi_{1}\cdots d\phi_{d-2}dtd\theta\right|\nonumber \\
 & = & \frac{2CnM_{R}^{2}}{\sigma_{n}^{2}\prod_{k=1}^{d}L_{k}}\cdot\frac{\pi}{2}\left|\int_{m}^{\infty}\frac{t^{d-1}}{\left(1+t^{2}\right)^{r}}\prod_{j=1}^{d-2}\left(\int_{0}^{\pi/2}\sin^{j}\phi_{d-1-j}d\phi_{d-1-j}\right)dt\right|\nonumber \\
 & = & \frac{\pi^{d/2}CnM_{R}^{2}}{2^{d-2}\Gamma\left(\frac{d}{2}\right)\sigma_{n}^{2}\prod_{k=1}^{d}L_{k}}\cdot\left|\int_{m}^{\infty}\frac{t^{d-1}}{\left(1+t^{2}\right)^{r}}dt\right|\nonumber \\
 & = & \frac{\pi^{d/2}CnM_{R}^{2}}{2^{d-1}\Gamma\left(\frac{d}{2}\right)\sigma_{n}^{2}\prod_{k=1}^{d}L_{k}}\cdot\left|\int_{m^{2}}^{\infty}\frac{t^{d/2-1}}{\left(1+t\right)^{r}}dt\right|\nonumber 
\end{eqnarray}
}where $m=\min_{1\leq k\leq d}\{L_{k}U_{k}\}$, and in the second
inequality we use $d$-dimensional ($d\geq2$) spherical coordinates
\cite{blumenson1960derivation}: 
\begin{eqnarray*}
\eta_{1} & = & t\cos\phi_{1}\\
2\leq k\leq d-2:\,\eta_{k} & = & t\cos\phi_{k}\prod_{j=1}^{k-1}\sin\phi_{j}\\
\eta_{d-1} & = & t\sin\theta\prod_{j=1}^{d-2}\sin\phi_{j}\\
\eta_{d} & = & t\cos\theta\prod_{j=1}^{d-2}\sin\phi_{j}
\end{eqnarray*}
where $\left\Vert \mat{\eta}\right\Vert =t$, $0\leq\phi_{j}\leq\pi/2,\,0\leq\theta<\pi/2,\,m\leq t<\infty$,
and the Jacobian is
\[
J=t^{d-1}\prod_{j=1}^{d-2}\sin^{j}\phi_{d-1-j}\,.
\]
Also, in the fourth equality we use the following property of the
beta function:
\[
\frac{\Gamma(x)\Gamma(y)}{\Gamma(x+y)}=\text{B}\left(x,y\right)=2\int_{0}^{\pi/2}\left(\sin\phi\right)^{2x-1}\left(\cos\phi\right)^{2y-1}d\phi
\]
with $y=1/2$ and $x=(j+1)/2$, for any $j=1,\ldots,d-2$, that is
\[
\int_{0}^{\pi/2}\sin^{j}\phi_{d-1-j}d\phi_{d-1-j}=\frac{\Gamma\left(\frac{1}{2}\right)\Gamma\left(\frac{j+1}{2}\right)}{2\Gamma\left(\frac{j+2}{2}\right)}
\]
which implies that
\[
\prod_{j=1}^{d-2}\left(\int_{0}^{\pi/2}\sin^{j}\phi_{d-1-j}d\phi_{d-1-j}\right)=\frac{\Gamma\left(\frac{1}{2}\right)^{d-2}}{2^{d-2}\Gamma\left(\frac{d}{2}\right)}=\frac{\pi^{d/2-1}}{2^{d-2}\Gamma\left(\frac{d}{2}\right)}\,.
\]
Now, we write the last integral of (\ref{eq:Matern_truncate}) in
terms of incomplete beta function, as follows:
\begin{eqnarray*}
\int_{m^{2}}^{\infty}\frac{t^{d/2-1}}{\left(1+t\right)^{r}}dt & =\\
\left[t=\frac{\tilde{t}}{1-\tilde{t}}\right] & = & \int_{m^{2}/(1+m^{2})}^{1}\tilde{t}^{d/2-1}\cdot(1-\tilde{t})^{r-d/2-1}d\tilde{t}\\
\left[\tilde{t}=\frac{1}{1-u}\right] & = & (-1)^{1+d/2-r}\int_{-m^{-2}}^{0}u^{r-d/2-1}(1-u)^{-r}du\\
 & = & (-1)^{d/2-r}\int_{0}^{-m^{-2}}u^{r-d/2-1}(1-u)^{-r}du\\
 & = & \frac{2m^{d-2r}}{2r-d}\,\ensuremath{_{2}F_{1}}\left(r-d/2,r;r-d/2+1;-m^{-2}\right)
\end{eqnarray*}
The expression in the last equality is the analytic continuation
of the beta function $\text{B}_{-m^{-2}}\left(r-d/2,1-r\right)$ (\cite[Sections 8.17, 15.4]{olver2010nist})
, i.e.
\[
\text{B}_{-m^{-2}}\left(r-d/2,1-r\right)=\frac{2m^{d-2r}}{2r-d}\,\ensuremath{_{2}F_{1}}\left(r-d/2,r;r-d/2+1;-m^{-2}\right)\,.
\]
Therefore, we obtain 
\begin{align*}
\left|\int_{\mathbb{R}^{d}}|\z(\mat{\eta})\valpha|^{2}p(\mat{\eta})d\mat{\eta}-\int_{{\cal Q}_{\mat U}}\left|\z(\mat{\eta})^{\conj}\valpha\right|^{2}p(\mat{\eta})d\mat{\eta}\right| & \leq\frac{\pi^{d/2}CnM_{R}^{2}}{2^{d-2}\Gamma\left(\frac{d}{2}\right)(2r-d)\sigma_{n}^{2}\prod_{k=1}^{d}L_{k}}m^{d-2r}\\
 & \quad\cdot\left|\ensuremath{_{2}F_{1}}\left(r-d/2,r;r-d/2+1;-m^{-2}\right)\right|\,.
\end{align*}
In order for Eq.~(\ref{eq:integral-split_1st_term}) to hold, let
$x>0$ be the solution of the equation
\[
\frac{\pi^{d/2}CnM_{R}^{2}}{2^{d-2}\Gamma\left(\frac{d}{2}\right)(2r-d)\sigma_{n}^{2}\prod_{k=1}^{d}L_{k}}x^{d-2r}\left|\ensuremath{_{2}F_{1}}\left(r-d/2,r;r-d/2+1;-x^{-2}\right)\right|=\frac{1}{2\sigma_{f}^{2}n}
\]
and then set $U_{k}^{(\min)}=x/L_{k}$.

Note that if we replace the expression after the second equality in
(\ref{eq:Matern_truncate}) with the upper bound 
\[
\frac{2CnM_{R}^{2}}{\sigma_{n}^{2}\prod_{k=1}^{d}L_{k}}\left|\int_{U_{1}}^{\infty}\ldots\int_{U_{d}}^{\infty}\left(\eta_{1}^{2}+\ldots+\eta_{d}^{2}\right)^{-r}d\eta_{1}\cdots d\eta_{d}\right|
\]
then, in order for Eq.~(\ref{eq:integral-split_1st_term})to hold,
we can set 
\[
U_{k}^{(\min)}=\frac{1}{L_{k}}\left(\frac{\pi^{d/2}CM_{R}^{2}\sigma_{f}^{2}n^{2}}{2^{d-2}(2r-d)\Gamma\left(\frac{d}{2}\right)\sigma_{n}^{2}\prod_{k=1}^{d}L_{k}}\right)^{1/(2r-d)}\,.
\]

Note also that for $d=2$, we have the simplest case of spherical
coordinates, which yields 
\[
U_{k}^{(\min)}=\frac{1}{L_{k}}\sqrt{\sqrt[r-1]{\frac{\pi CM_{R}^{2}\sigma_{f}^{2}n^{2}}{(r-1)\sigma_{n}^{2}L_{1}L_{2}}}-1}\,.
\]

Also, for $d=1$ we do not need spherical coordinates. In that case,
we have 
\begin{eqnarray*}
\left|\int_{\mathbb{R}^{d}}|\z(\mat{\eta})\valpha|^{2}p(\mat{\eta})d\mat{\eta}-\int_{{\cal Q}_{\mat U}}\left|\z(\mat{\eta})^{\conj}\valpha\right|^{2}p(\mat{\eta})d\mat{\eta}\right| & = & \frac{2nCM_{R}^{2}}{\sigma_{n}^{2}}\left|\int_{U}^{\infty}\frac{1}{\left(1+L_{1}^{2}\eta^{2}\right)^{r}}d\eta\right|\\
 & = & \frac{nCM_{R}^{2}}{\sigma_{n}^{2}L_{1}}\left|\int_{L_{1}^{2}U^{2}}^{\infty}\frac{s^{-1/2}}{\left(1+s\right)^{r}}ds\right|\\
 & = & \frac{2nCM_{R}^{2}L_{1}^{1-2r}U^{1-2r}}{(2r-1)\sigma_{n}^{2}L_{1}}\left|\ensuremath{_{2}F_{1}}\left(r-1/2,r;r+1/2;-L_{1}^{-2}U^{-2}\right)\right|\,.
\end{eqnarray*}
Now, $U^{(\min)}$ that equates the last expression with $1/2\sigma_{f}^{2}n$
is obtained by solving the same equation obtained for $d\geq3$, only
with $d=1$.

Note that if we bound the first integral similarly to the bound in
the case $d\geq3$, we obtain the same formula for $U^{(\min)}$ only
with $d=1$.The case of \textbf{$p\in{\cal E}_{C,\mat L}^{(1)}$:}

In this case,
\begin{eqnarray*}
\left|\int_{\mathbb{R}^{d}}|\z(\mat{\eta})^{*}\v|^{2}p(\mat{\eta})d\mat{\eta}-\int_{{\cal Q}_{U}}\left|\z(\mat{\eta})^{\conj}\v\right|^{2}p(\mat{\eta})d\mat{\eta}\right| & = & \left|\int_{\left|\mat{\eta}\right|\geq\mat U}|\z(\mat{\eta})^{\conj}\v|^{2}p(\mat{\eta})d\mat{\eta}\right|\\
 & \leq & \frac{2CnM_{R}^{2}}{\sigma_{n}^{2}}\left|\int_{U_{1}}^{\infty}e^{-L_{1}\eta_{1}}d\eta_{1}\cdots\int_{U_{d}}^{\infty}e^{-L_{d}\eta_{d}}d\eta_{d}\right|\\
 & = & \frac{2CnM_{R}^{2}}{\sigma_{n}^{2}}\prod_{k=1}^{d}\left|\int_{U_{k}}^{\infty}e^{-L_{k}\eta_{k}}d\eta_{k}\right|\\
 & = & \frac{2CnM_{R}^{2}}{\sigma_{n}^{2}}\prod_{k=1}^{d}\frac{e^{-L_{k}U_{k}}}{L_{k}}
\end{eqnarray*}
So, in order for Eq.~(\ref{eq:integral-split_1st_term}) to hold
we set 
\[
U_{k}^{(\min)}=\frac{1}{L_{k}}\ln\left(\frac{1}{L_{k}}\left(\frac{4CM_{R}^{2}\sigma_{f}^{2}n^{2}}{\sigma_{n}^{2}}\right)^{1/d}\right)\,.
\]

\paragraph*{The case of $p\in{\cal E}_{C,\protect\mat L}^{(2)}$:}

In this case,

\begin{eqnarray*}
\left|\int_{\mathbb{R}^{d}}|\z(\mat{\eta})^{*}\v|^{2}p(\mat{\eta})d\mat{\eta}-\int_{{\cal Q}_{U}}\left|\z(\mat{\eta})^{\conj}\v\right|^{2}p(\mat{\eta})d\mat{\eta}\right| & = & \left|\int_{\left|\mat{\eta}\right|\geq\mat U}|\z(\mat{\eta})^{\conj}\v|^{2}p(\mat{\eta})d\mat{\eta}\right|\\
 & \leq & \frac{2CnM_{R}^{2}}{\sigma_{n}^{2}}\left|\int_{U_{1}}^{\infty}e^{-L_{1}^{2}\eta_{1}^{2}}d\eta_{1}\cdots\int_{U_{d}}^{\infty}e^{-L_{d}^{2}\eta_{d}^{2}}d\eta_{d}\right|\\
 & = & \frac{\pi^{d/2}CnM_{R}^{2}}{2^{d-1}\sigma_{n}^{2}}\prod_{k=1}^{d}\frac{\erfc{L_{k}U_{k}}}{L_{k}}\\
 & \leq & \frac{\pi^{d/2}CnM_{R}^{2}}{2^{d-1}\sigma_{n}^{2}}\prod_{k=1}^{d}\frac{e^{-L_{k}^{2}U_{k}^{2}}}{L_{k}}
\end{eqnarray*}
where we used the bound $\erfc x\leq e^{-x^{2}}$ for non-negative
$x$, which follows from \cite{craig1991new}. In this paper it was
shown that for any non-negative $x$
\[
\erfc x=\frac{2}{\pi}\int_{0}^{\frac{\pi}{2}}e^{-\frac{x^{2}}{\sin^{2}\theta}}d\theta
\]
and therefore $\erfc x=\frac{2}{\pi}\int_{0}^{\frac{\pi}{2}}e^{-\frac{x^{2}}{\sin^{2}\theta}}d\theta\leq\frac{2}{\pi}\int_{0}^{\frac{\pi}{2}}e^{-x^{2}}d\theta=e^{-x^{2}}$.
So, in order for Eq.~(\ref{eq:integral-split_1st_term}) to hold
we set
\[
U_{k}^{(\min)}=\frac{1}{L_{k}}\sqrt{\ln\left(\frac{\sqrt{\pi}}{L_{k}}\left(\frac{2^{2-d}CM_{R}^{2}\sigma_{f}^{2}n^{2}}{\sigma_{n}^{2}}\right)^{1/d}\right)}\,.
\]
\end{proof}

\subsection{Approximating the Truncated Integral}

The nodes $\veta_{1},\dots,\veta_{s}$ and their weights $h_{1}(\vtheta_{0}),\dots,h_{s}(\vtheta_{0})$
are simply rescaled multivariate Gauss-Legendre quadrature nodes and
weights\footnote{Gauss-Legendre quadrature is defined for one dimensional integrals.
In multivariate Gauss-Legendre quadrature we refer to the quadrature
obtained by tensorizing the one-dimensional quadrature.}, and so $\sum_{j=1}^{s}h_{j}(\vtheta_{0})|\z(\veta_{j})^{*}\v|^{2}$
is a quadrature approximation of $\int_{\Q_{\mat U}}|\z(\veta)^{*}\v|^{2}p(\veta;\vtheta_{0})d\veta$.
In this section we derive a lower bound on $\s$ that guarantees that
\[
\left|\int_{\Q_{\mat U}}|\z(\veta)^{*}\v|^{2}p(\veta;\vtheta_{0})d\veta-\sum_{j=1}^{s}h_{j}(\vtheta_{0})|\z(\veta_{j})^{*}\v|^{2}\right|\leq\frac{1}{2\sigma_{f}^{2}n}
\]
The bound on $\s$ depends on $\mat U$. Together with Proposition~\ref{prop:setting-U},
we completely specify how to build the quadrature approximation so
that Eq.~(\ref{eq:integral-approximation}) holds.

\subsubsection{Decay of Chebyshev Coefficients for Multivariate Functions}

Our analysis relies on generalizations of existing decay bounds for
Chebyshev expansions of analytic functions in one dimension to multivariate
functions. In this subsection we introduce these results.

Classical decay bounds for Chebyshev expansions of analytic functions
in one dimension are based on bounding the function values on the
Bernstein ellipse (see \cite[Section 1.4]{mason2002chebyshev} for
further details). For multivariate functions, a polyellipse is used
instead.
\begin{defn}
A \emph{Bernstein ellipse} is an open region in the complex plane
which bounded by an ellipse with foci $\pm1$. A \emph{Bernstein polyellipse
in $d$-dimensions }is a cartesian product of $d$ Bernstein ellipses.

Let $\rho(U,\beta)\coloneqq\beta/(2U)+\sqrt{\beta^{2}/(4U^{2})+1}$.
Given $\mat U=(U_{1},\dots,U_{d})>0$ and a singularity point $\vbeta=(\beta_{1},\dots,\beta_{d})>0$,
denote\textcolor{green}{. }
\[
E_{\mat U,\vbeta}\coloneqq\left\{ \z\in\mathbb{C}^{d}:\,\left|z_{k}+\sqrt{z_{k}^{2}-U_{k}^{2}}\right|<U_{k}\rho(U_{k},\beta_{k})\quad\quad\forall k=1,\ldots,d\right\} \,.
\]
Note that $E_{1,\vbeta\oslash\mat U}$, where $\vbeta\oslash\mat U$
denotes entrywise division between $\vbeta$ and $\mat U$, is a Bernstein
polyellipse, and that $E_{\mat U,\vbeta}=\mat U\odot E_{1,\vbeta\oslash\mat U}$.
So, $E_{\mat U,\vbeta}$ is a polyellipse with foci at $\pm U_{k}$.

For a multivariate analytic function $f$ on $\left[-1,1\right]^{d}$,
the multivariate tensorised Chebyshev expansion is given by
\[
f(\x)=\sum_{j_{1},\ldots,j_{d}=0}^{\infty}a_{j_{1}\ldots j_{d}}T_{j_{1}}(x_{1})\cdots T_{j_{d}}(x_{d})
\]
where the coefficients are given by
\[
a_{j_{1}\ldots j_{d}}=\frac{2^{d-m}}{\pi^{d}}\int_{{\cal Q}_{1}}\frac{f\left(x_{1},\ldots,x_{d}\right)T_{j_{1}}(x_{1})\cdots T_{j_{d}}(x_{d})}{\sqrt{1-x_{1}^{2}}\cdots\sqrt{1-x_{d}^{2}}}dx_{1}\cdots dx_{d}
\]
where $m:=\#\left\{ j_{k}:\,j_{k}=0\right\} $.

The following is a generalization of classical results for one dimension~\cite[Theorem 8.1, Theorem 8.2]{ftrefethen2013approximation}:
\end{defn}

\begin{thm}
\label{thm:trefethen_coeff_highD} Let $f$ be an analytic function
on $\left[-1,1\right]^{d}$ and analytically continuable to $E_{1,\vbeta\oslash\mat U}$
where it satisfies $|f\left(x_{1},\ldots,x_{d}\right)|\leq M$ for
some $M>0$. Then, for all $j_{1},\dots,j_{d}$,
\[
|a_{j_{1}\ldots j_{d}}|\leq\frac{2^{d-m}M}{\rho_{1}^{j_{1}}\cdots\rho_{d}^{j_{d}}}
\]
\end{thm}

Although Theorem \ref{thm:trefethen_coeff_highD} has essentially
been proven in \cite{wang2020analysis}, for completeness we include
in Appendix \ref{sec:multi_cheb} our proof of Theorem~ \ref{thm:trefethen_coeff_highD}
which is based on a different technique.

\subsubsection{Bounding the Integration Error}

We have the following result:
\begin{thm}
\label{thm: quad-thm}Given $\matU=(U_{1},\ldots,U_{d})$ such that
$U_{1},\ldots,U_{d}>0$, and $\vbeta$ such that $\beta_{1},\ldots,\beta_{d}>0$,
let $E_{\mat U,\vbeta}$ be the polyellipse such that in dimension
$j$ the foci is $\pm U_{j}$ and passes through $i\beta_{j}/2$,
and let $\vrho=(\rho_{1},\dots,\rho_{d})$ with $\rho_{j}\coloneqq\rho(U_{j},\beta_{j})$
(for $j=1,\dots,d$) denote the sum of the semi-axes in each dimension.
Assume that either $p(\cdot;\vtheta_{0})\in{\cal P}_{C,\mat L}$ or
$p(\cdot;\vtheta_{0})\in{\cal P}_{C,\mat L}^{(r)}$ where $r>d/2$,
or $p\in{\cal E}_{C,\mat L}^{(1)}$ or $p\in{\cal E}_{C,\mat L}^{(2)}$.
Furthermore, assume that if $\z(\veta)=\a(\veta)+i\b(\veta)$, then
for each $1\leq j\leq n$ the functions $\a_{j},\b_{j}$ are analytic
on $\mathbb{R}^{d}$. Finally, assume that $p(\cdot;\vtheta_{0})$
has an analytic continuation $\hat{p}(\cdot;\vtheta_{0})$ to $E_{\matU,\vbeta}$.
Let $\hat{\z}(\cdot)$ denote the analytic continuation of $\z(\cdot)$.
Denote
\begin{align*}
M_{R} & \coloneqq\sup_{\veta\in\R^{d}}\InfNorm{\z(\veta)}\\
M_{\mat U,\vbeta} & \coloneqq\sup_{\veta\in E_{\mat U,\vbeta}}\InfNorm{\hat{\z}(\veta)}\\
C_{\mat U,\vbeta} & \coloneqq\sup_{\veta\in E_{\matU,\vbeta}}|\hat{p}(\veta;\vtheta_{0})|\,.
\end{align*}
Then for
\[
s_{k}\geq\frac{\frac{1}{d}\ln\left(2^{2d+2}M_{\matU,\vbeta}^{2}C_{\mat U,\vbeta}\sigma_{n}^{-2}\sigma_{f}^{2}n^{2}\right)+\ln U_{k}-\ln(\rho_{k}-1)}{2\ln\rho_{k}}+1,\,k=1,\dots,d,
\]
we have
\[
\left|\int_{\Q_{\mat U}}|\z(\veta)^{*}\v|^{2}p(\veta;\vtheta_{0})d\veta-\sum_{j=1}^{s}h_{j}(\vtheta_{0})|\z(\veta_{j})^{*}\v|^{2}\right|\leq\frac{1}{2\sigma_{f}^{2}n}\,.
\]
Note that in this case 
\[
s=\prod_{k=1}^{d}s_{k}=O\left((2d)^{-d}\prod_{k=1}^{d}\frac{\ln\left(2^{2d+2}M_{\matU,\vbeta}^{2}C_{\mat U,\vbeta}U_{k}^{d}\sigma_{n}^{-2}\sigma_{f}^{2}n^{2}\right)-d\ln(\rho_{k}-1)}{\ln\rho_{k}}\right)
\]
\end{thm}

\begin{proof}
For conciseness, we drop $\vtheta_{0}$ from $p$ throughout the proof.
For convenience, we use the following form of the quadrature rule
\[
\sum_{j=1}^{s}h_{j}(\vtheta_{0})|\z(\veta_{j})^{*}\v|^{2}=\sum_{j_{1}=1}^{s_{1}}\cdots\sum_{j_{d}=1}^{s_{d}}w_{j_{1}\cdots j_{d}}p(\veta_{j_{1}\cdots j_{d}})|\z(\veta_{j_{1}\cdots j_{d}})^{*}\v|^{2}
\]
as presented in Section \ref{sec:GL-Features}.

Denote $f_{\v}(\mat{\eta})=|\z(\mat{\eta})^{\conj}\v|^{2}$. Also
denote $\tilde{p}(\mat{\chi})\coloneqq p(U_{1}\chi_{1},\dots,U_{d}\chi_{d})$
and $\tilde{f_{\v}}(\mat{\chi})\coloneqq f_{\v}(U_{1}\chi_{1},\dots,U_{d}\chi_{d})$.
Note that 
\begin{align*}
\left|\int_{{\cal Q}_{\mat U}}f_{\v}(\mat{\eta})p(\mat{\eta})d\mat{\eta}-\sum_{j_{1}=1}^{s_{1}}\cdots\sum_{j_{d}=1}^{s_{d}}w_{j_{1}\ldots j_{d}}p(\veta_{j_{1}\cdots j_{d}})f_{\v}(\mat{\eta}_{j_{1}\ldots j_{d}})\right| & =\\
\left(\prod_{k=1}^{d}U_{k}\right)\left|\int_{\left[-1,1\right]^{d}}\tilde{f_{\v}}(\mat{\chi})\tilde{p}(\mat{\chi})d\mat{\chi}-\sum_{j_{1}=1}^{s_{1}}\cdots\sum_{j_{d}=1}^{s_{d}}\tilde{w}{}_{j_{1}\ldots j_{d}}\tilde{p}(\mat{\chi}_{j_{1}\ldots j_{d}})\tilde{f_{\v}}(\mat{\chi}_{j_{1}\ldots j_{d}})\right|
\end{align*}
where $\tilde{w}{}_{j_{1}\ldots j_{d}}=\prod_{k=1}^{d}w_{j_{k}}^{(s_{k})}$.
The sum in the right-hand side is a quadrature approximation of $\tilde{f_{\v}}(\mat{\chi})\tilde{p}(\mat{\chi})$,
which we analyze.

To that end, we first bound the analytic continuation of $\tilde{f_{\v}}(\mat{\chi})\tilde{p}(\mat{\chi})$
on $E_{1,\vbeta\oslash\mat U}$ (where $\vbeta\oslash\matU$ denotes
entrywise division). For every $\mat{\eta}\in E_{\mat U,\vbeta}$
we have (similar to the derivation of Eq.~(\ref{eq:uni_bound_f})):
\[
|\hat{f_{\v}}(\mat{\eta})|=\left||\hat{\z}(\mat{\eta})^{\conj}\v|^{2}\right|\leq\sigma_{n}^{-2}\|\hat{\z}(\mat{\eta})\|_{2}^{2}\leq n\sigma_{n}^{-2}M_{\mat U,\vbeta}^{2}\,.
\]

Thus, $|\hat{f_{\v}}(\mat{\eta})\hat{p}(\mat{\eta})|\leq n\sigma_{n}^{-2}M_{\mat U,\vbeta}^{2}C_{\mat U,\vbeta}$
and $|\hat{\tilde{f_{\v}}}(\mat{\chi})\hat{\tilde{p}}(\mat{\chi})|\leq n\sigma_{n}^{-2}M_{\mat U,\vbeta}^{2}C_{\mat U,\vbeta}$.
We can now apply quadrature approximation bounds on $\tilde{f_{\v}}(\mat{\chi})\tilde{p}(\mat{\chi})$
to bound the error
\[
e_{s}\coloneqq\left|\int_{\left[-1,1\right]^{d}}\tilde{f_{\v}}(\mat{\chi})\tilde{p}(\mat{\chi})d\mat{\chi}-\sum_{j_{1}=1}^{s_{1}}\cdots\sum_{j_{d}=1}^{s_{d}}\tilde{w}{}_{j_{1}\ldots j_{d}}\tilde{f_{\v}}(\mat{\chi}_{j_{1}\ldots j_{d}})\tilde{p}(\mat{\chi}_{j_{1}\ldots j_{d}})\right|.
\]

Let 
\[
\tilde{f_{\v}}(\mat{\chi})\tilde{p}(\mat{\chi})=\sum_{k_{1},\ldots,k_{d}=0}^{\infty}a_{k_{1}\ldots k_{d}}T_{k_{1}}(\chi_{1})\cdots T_{k_{d}}(\chi_{d})
\]
be the multivariate Chebyshev expansion of $\tilde{f_{\v}}(\mat{\chi})\tilde{p}(\mat{\chi})$,
and let $P_{2s-1}$ be the truncated expansion: 
\[
P_{2s-1}(\mat{\chi})\coloneqq\sum_{k_{1}=0}^{2s_{1}-1}\dots\sum_{k_{d}=0}^{2s_{d}-1}a_{k_{1}\ldots k_{d}}T_{k_{1}}(\chi_{1})\cdots T_{k_{d}}(\chi_{d})\,.
\]
Similarly to the strategy employed in \cite{ubaru2017fast} and \cite[Theorem 19.3]{ftrefethen2013approximation},
we have
\begin{eqnarray*}
e_{s} & = & \left|\int_{\left[-1,1\right]^{d}}\tilde{f_{\v}}(\mat{\chi})\tilde{p}(\mat{\chi})d\mat{\chi}-\int_{\left[-1,1\right]^{d}}P_{2s-1}(\mat{\chi})d\mat{\chi}+\right.\\
 &  & \quad\left.\sum_{j_{1}=1}^{s_{1}}\cdots\sum_{j_{d}=1}^{s_{d}}\left(P_{2s-1}(\mat{\chi}_{j_{1}\ldots j_{d}})\tilde{w}{}_{j_{1}\ldots j_{d}}\right)-\sum_{j_{1}=1}^{s_{1}}\cdots\sum_{j_{d}=1}^{s_{d}}\left(\tilde{f_{\v}}(\mat{\chi}_{j_{1}\ldots j_{d}})\tilde{p}(\mat{\chi}_{j_{1}\ldots j_{d}})\tilde{w}{}_{j_{1}\ldots j_{d}}\right)\right|\\
 & = & \left|\int_{\left[-1,1\right]^{d}}\left(\tilde{f_{\v}}(\mat{\chi})\tilde{p}(\mat{\chi})-P_{2s-1}(\mat{\chi})\right)d\mat{\chi}-\sum_{j_{1}=1}^{s_{1}}\cdots\sum_{j_{d}=1}^{s_{d}}\left(\tilde{f_{\v}}(\mat{\chi}_{j_{1}\ldots j_{d}})\tilde{p}(\mat{\chi}_{j_{1}\ldots j_{d}})-P_{2s-1}(\mat{\chi}_{j_{1}\ldots j_{d}})\right)\tilde{w}{}_{j_{1}\ldots j_{d}}\right|\\
 & = & \left|\int_{-1}^{1}\ldots\int_{-1}^{1}\sum_{k_{1}=2s_{1}}^{\infty}\dots\sum_{k_{d}=2s_{d}}^{\infty}a_{k_{1}\ldots k_{d}}T_{k_{1}}(\chi_{1})\cdots T_{k_{d}}(\chi_{d})d\chi_{1}\cdots d\chi_{d}-\right.\\
 &  & \quad\left.\sum_{j_{1}=1}^{s_{1}}\cdots\sum_{j_{d}=1}^{s_{d}}\left(\sum_{k_{1}=2s_{1}}^{\infty}\dots\sum_{k_{d}=2s_{d}}^{\infty}a_{k_{1}\ldots k_{d}}T_{k_{1}}(\chi_{j_{1}}^{(s_{1})})\cdots T_{k_{d}}(\chi_{j_{d}}^{(s_{d})})\right)\tilde{w}{}_{j_{1}\ldots j_{d}}\right|\\
 & \leq & \sum_{k_{1}=2s_{1}}^{\infty}\dots\sum_{k_{d}=2s_{d}}^{\infty}\left|a_{k_{1}\ldots k_{d}}\right|\left[\prod_{m=1}^{d}\int_{-1}^{1}\left|T_{k_{m}}(\chi_{m})\right|d\chi_{m}+\sum_{j_{1}=1}^{s_{1}}\cdots\sum_{j_{d}=1}^{s_{d}}w_{j_{1}}^{(s_{1})}\cdots w_{j_{d}}^{(s_{d})}\left|T_{k_{1}}(\chi_{j_{1}}^{(s_{1})})\right|\cdots\left|T_{k_{d}}(\chi_{j_{d}}^{(s_{d})})\right|\right]\\
 & \leq & \sum_{k_{1}=2s_{1}}^{\infty}\dots\sum_{k_{d}=2s_{d}}^{\infty}\left|a_{k_{1}\ldots k_{d}}\right|\left[2^{d}+\left(\sum_{j_{1}=1}^{s_{1}}w_{j_{1}}^{(s_{1})}\right)\cdots\left(\sum_{j_{d}=1}^{s_{d}}w_{j_{d}}^{(s_{d})}\right)\right]\\
 & \leq & \sum_{k_{1}=2s_{1}}^{\infty}\dots\sum_{k_{d}=2s_{d}}^{\infty}\left|a_{k_{1}\ldots k_{d}}\right|\left[2^{d}+2^{d}\right]\\
 & \leq & \frac{2^{2d+1}nM_{\mat U,\vbeta}^{2}C_{\mat U,\vbeta}}{\sigma_{n}^{2}}\sum_{k_{1}=2s_{1}}^{\infty}\dots\sum_{k_{d}=2s_{d}}^{\infty}\frac{1}{\rho_{1}^{k_{1}}\cdots\rho_{d}^{k_{d}}}\\
 & = & \frac{2^{2d+1}nM_{\mat U,\vbeta}^{2}C_{\mat U,\vbeta}}{\sigma_{n}^{2}}\cdot\prod_{k=1}^{d}\frac{1}{\rho_{k}^{2s_{k}-1}(\rho_{k}-1)}
\end{eqnarray*}
 where we use the bound 
\[
\tilde{f_{\v}}(\mat{\chi})\tilde{p}(\mat{\chi})\leq\frac{nM_{\matU,\vbeta}^{2}C_{\mat U,\vbeta}}{\sigma_{n}^{2}}\,.
\]
In the first equality, we use the following equality
\[
\int_{\left[-1,1\right]^{d}}P_{2s-1}(\mat{\chi})d\mat{\chi}=\sum_{j_{1}=1}^{s_{1}}\cdots\sum_{j_{d}=1}^{s_{d}}\tilde{w}{}_{j_{1}\ldots j_{d}}P_{2s-1}(\mat{\chi}_{j_{1}\ldots j_{d}})
\]
which follows from the exactness of the Gauss-Legendre quadrature
in one dimension:
\begin{eqnarray*}
\int_{\left[-1,1\right]^{d}}P_{2s-1}(\mat{\chi})d\mat{\chi} & = & \int_{\left[-1,1\right]^{d}}\sum_{k_{1}=0}^{2s_{1}-1}\dots\sum_{k_{d}=0}^{2s_{d}-1}a_{k_{1}\ldots k_{d}}T_{k_{1}}(\chi_{1})\cdots T_{k_{d}}(\chi_{d})d\chi_{1}\cdots d\chi_{d}\\
 & = & \sum_{k_{1}=0}^{2s_{1}-1}\dots\sum_{k_{d}=0}^{2s_{d}-1}a_{k_{1}\ldots k_{d}}\int_{\left[-1,1\right]^{d}}T_{k_{1}}(\chi_{1})\cdots T_{k_{d}}(\chi_{d})d\chi_{1}\cdots d\chi_{d}\\
 & = & \sum_{k_{1}=0}^{2s_{1}-1}\dots\sum_{k_{d}=0}^{2s_{d}-1}a_{k_{1}\ldots k_{d}}\left(\int_{-1}^{1}T_{k_{1}}(\chi_{1})d\chi_{1}\right)\cdots\left(\int_{-1}^{1}T_{k_{d}}(\chi_{d})d\chi_{d}\right)\\
 & = & \sum_{k_{1}=0}^{2s_{1}-1}\dots\sum_{k_{d}=0}^{2s_{d}-1}a_{k_{1}\ldots k_{d}}\left(\sum_{j_{1}=1}^{s_{1}}w_{j_{1}}^{(s_{1})}T_{k_{1}}(\chi_{1,j_{1}})\right)\cdots\left(\sum_{j_{d}=1}^{s_{d}}w_{j_{d}}^{(s_{d})}T_{k_{d}}(\chi_{d,j_{d}})\right)\\
 & = & \sum_{j_{1}=1}^{s_{1}}\dots\sum_{j_{d}=1}^{s_{d}}w_{j_{1}}^{(s_{1})}\cdots w_{j_{d}}^{(s_{d})}\left(\sum_{k_{1}=0}^{2s_{1}-1}\dots\sum_{k_{d}=0}^{2s_{d}-1}a_{k_{1}\ldots k_{d}}T_{k_{1}}(\chi_{1,j_{1}})\cdots T_{k_{d}}(\chi_{d,j_{d}})\right)\\
 & = & \sum_{j_{1}=1}^{s_{1}}\dots\sum_{j_{d}=1}^{s_{d}}w_{j_{1}}^{(s_{1})}\cdots w_{j_{d}}^{(s_{d})}P_{2s-1}(\mat{\chi}_{j_{1}\ldots j_{d}})\,.
\end{eqnarray*}
Hence,
\[
\left|\int_{{\cal Q}_{\mat U}}\left|\z(\mat{\eta})^{\conj}\v\right|^{2}p(\mat{\eta})d\mat{\eta}-\sum_{j_{1}=1}^{s_{1}}\cdots\sum_{j_{d}=1}^{s_{d}}w_{j_{1}\cdots j_{d}}|\z(\mat{\eta}_{j})^{\conj}\v|^{2}\right|=e_{s}\prod_{k=1}^{d}U_{k}\leq\frac{2^{2d+1}nM_{\matU,\vbeta}^{2}C_{\mat U,\vbeta}}{\sigma_{n}^{2}}\cdot\prod_{k=1}^{d}\frac{U_{k}}{\rho_{k}^{2s_{k}-1}(\rho_{k}-1)}\,.
\]
Finally, bounding
\[
\left(\frac{2^{2d+1}nM_{\matU,\vbeta}^{2}C_{\matU,\vbeta}}{\sigma_{n}^{2}}\right)^{1/d}\cdot\frac{U_{k}}{\rho_{k}^{2s_{k}-1}(\rho_{k}-1)}\leq\frac{1}{\left(2\sigma_{f}^{2}n\right)^{1/d}}
\]
for each $k=1,\dots,d$ , gives the bound from the theorem and the
statement now follows immediately.
\end{proof}
The last theorem allows us to compute the required $\s$ based on
$\matU$ and the singularities in $p(\cdot;\vtheta_{0})$. In Section~\ref{sec:examples},
we show concrete examples for using Theorem~\ref{thm: quad-thm}
to bound $\s$ and $\matU$. The main step is finding the polyellipse
parameters. If the analytic extension of $p(\cdot;\vtheta_{0})$ has
its first singularity at the pure imaginary value $\x_{0}=\pm i\mat{\beta}$
for some $\mat{\beta}=(\beta_{1},\dots,\beta_{d})^{\T}$ such that
$0<\beta_{k}<U_{k}$, then $\tilde{p}(\cdot;\vtheta_{0})$ has its
first singularity at $\tilde{\x}_{0}=\pm i\vbeta\oslash\matU$. Thus,
we can choose the ellipses parameters to be $\rho_{k}=\frac{\beta_{k}}{2U_{k}}+\sqrt{\frac{\beta_{k}^{2}}{4U_{k}^{2}}+1}$.
Otherwise, we can choose $\beta_{k}=2U_{k}$, i.e., $\rho_{k}=1+\sqrt{2}$.

\subsection{\label{subsec:domain}Handling a Parameter Domain $\Theta$}

Given a hyperparameter domain $\Theta$, we want to set the parameters
$\matU=(U_{1},\dots,U_{d})$ and $\s=(s_{1},\dots,s_{d})$ such that
Eq.~(\ref{eq:integral-approximation}) hold for every $\vtheta\in\Theta$.
This way, the parameterized family of positive definite kernel approximations
$\{\tilde{k}_{\vtheta}\}_{\vtheta\in\Theta}$ given by Eq. (\ref{eq:approx-kernel})
with these $\matU$ and $\s$ is $n$-spectrally equivalent to $\{k_{\vtheta}\}_{\vtheta\in\Theta}$
on the data domain ${\cal X}$. To do that, we need to find the worst-case
(over $\vtheta\in\Theta$) parameters $\matU$ and $\s$. The following
gives a general end-to-end statement.
\begin{thm}
\label{thm:set_parameter}Let 
\[
k_{\vtheta}(\x,\x')=\sigma_{f}^{2}\int_{\R^{d}}\varphi(\x,\veta)\varphi(\x',\veta)^{*}p(\veta;\vtheta_{0})d\veta+\sigma_{n}^{2}\gamma(\x,\x')
\]
be a parameterized family of kernels where $\vtheta\in\Theta$. Suppose
that:
\begin{enumerate}
\item $|\varphi(\x,\veta)|\leq M_{R}$ for every $\x\in{\cal X}$ and $\veta\in\mathbb{R}^{d}$.
\item We set $\matU\geq\sup_{\vtheta\in\Theta}\matU^{(\min)}(\vtheta)$
where $\matU^{(\min)}(\vtheta)$ is the value set by Proposition~\ref{prop:setting-U}
using parameters $\vtheta$.
\item For every $n$ points $\x_{1},\dots,\x_{n}\in\X$ we set $\z(\veta)=[\varphi(\x_{1},\veta),\dots,\varphi(\x_{n},\veta)]^{\T}$.
If we write $\z(\veta)=\a(\veta)+i\b(\veta)$, then for each $1\leq j\leq n$
the functions $\a_{j},\b_{j}$ are analytic on $\mathbb{R}^{d}$.
\item $\vbeta$ is such that $p(\cdot;\vtheta_{0})$ has an analytic continuation
$\hat{p}(\cdot;\vtheta_{0})$ to $E_{\mat U,\vbeta}$ for all $\vtheta\in\Theta$.
Let $\hat{\z}(\cdot)$ denote the analytic continuation of $\z(\cdot)$.
\end{enumerate}
Let

\begin{align*}
M_{\matU,\vbeta} & \coloneqq\sup_{\veta\in E_{\matU,\vbeta}}\InfNorm{\hat{\z}(\veta)}\\
C_{\matU,\vbeta} & \coloneqq\sup_{\vtheta\in\Theta,\,\veta\in E_{\matU,\vbeta}}\left|\hat{p}(\veta;\vtheta_{0})\right|\,.
\end{align*}
Then for
\[
s_{k}\geq\frac{\frac{1}{d}\ln\left(2^{2d+2}M_{\matU,\vbeta}^{2}C_{\matU,\vbeta}\sigma_{n}^{-2}\sigma_{f}^{2}n^{2}\right)+\ln U_{k}-\ln(\rho_{k}-1)}{2\ln\rho_{k}}+1,\,k=1,\dots,d,
\]
the parameterized family of kernel approximations given by Eq. (\ref{eq:approx-kernel})
is $n$-spectrally equivalent to $\{k_{\vtheta}\}_{\vtheta\in\Theta}$
on the data domain ${\cal X}$. Furthermore, if we set $\matU=\matU^{(\min)}$
and $\s$ according the last lower bound we have 
\[
s=\prod_{k=1}^{d}s_{k}=O\left((2d)^{-d}\prod_{k=1}^{d}\frac{\ln\left(2^{2d+2}M_{\matU,\vbeta}^{2}C_{\matU,\vbeta}U_{k}^{d}\sigma_{n}^{-2}\sigma_{f}^{2}n^{2}\right)-d\ln(\rho_{k}-1)}{\ln\rho_{k}}\right)
\]

\end{thm}

To use this theorem, one needs to bound the decay of the density functions
$p(\cdot;\vtheta_{0})$ over $\Theta$ and calculate an upper bound
on $C_{\matU,\vbeta}$. In general this might be hard, but luckily
in most kernels display a monotonicity in their hyperparameters that
helps identify the worst case for $\vtheta$ over $\Theta$. For example,
for the one dimensional Gaussian kernel $\exp(-\TNormS{\ell^{-1}(x-x')}//2)$,
the various parameters in the theorems monotonically increase as $\ell\to0$.
In the next section we given concrete examples for using Theorem~\ref{thm:set_parameter}
for the kernels listed in Table~\ref{tab:decay-examples}.

\section{\label{sec:examples}Examples of Feature Maps for Kernels}

In this section we show how to apply the theory presented in the previous
section to design $n$-spectrally equivalent kernel approximations
for a few widely used kernel functions. Throughout this section, we
assume that $n$ is fixed, the data domain is $\X\subseteq\R^{d}$,
and the hyperparameter domain is $\Theta$. Furthermore, we assume
that we have a bounding box on the domain, i.e. $\X\subseteq\prod_{k=1}^{d}[-R_{i}/2,R_{i}/2]$
(obviously, such a bounding box can be easily computed from the input
data). Let $\mat R=[R_{1},\dots,R_{d}]$.

\subsection{Gaussian Kernel}

Recall that the Gaussian kernel is
\[
k_{\vtheta}(\x,\x')=\sigma_{f}^{2}\exp\left(-\TNormS{\x-\x'}/2\ell^{2}\right)+\sigma_{n}^{2}\gamma(\x-\x')
\]
($\vtheta=[\ell,\sigma_{f}^{2},\sigma_{n}^{2}]$) where we added a
scaling factor $\sigma_{f}^{2}$ and included the ridge term $\sigma_{n}^{2}\gamma(\x-\x')$
in the kernel definition. Note that for conciseness, we consider the
isotropic version; the formulas can be modified for the anisotropic
case. As discussed in Section~\ref{sec:GL-Features}, by setting
$\varphi(\x,\veta)=e^{-i\x^{\T}\veta}$ , this kernel matches the
form of Eq.~(\ref{eq:kernel-form}). We assume that the hyperparameters
are bounded as follows:
\[
\Theta=\{[\ell,\sigma_{n}^{2},\sigma_{f}^{2}]\,:\,\ell\geq\ell_{0},\sigma_{n}^{2}\geq\sigma_{n0}^{2},\sigma_{f}^{2}\leq\sigma_{f0}^{2}\}
\]
where $\sigma_{n0}^{2}>0$ (i.e. we have a ridge term; our method
is not able to approximate the Gaussian kernel in the absence of a
ridge term).

The density is given by
\[
p(\veta;\ell)=\ell^{d}(2\pi)^{-d/2}\exp\left(-\TNormS{\veta}\ell^{2}/2\right)\,.
\]
Therefore, for $\vtheta\in\Theta$ , $p(\cdot;\ell)\in{\cal E}_{C,\matL}^{(2)}$
with $C=\ell_{0}^{d}(2\pi)^{-d/2}$ and $\matL=\frac{\ell_{0}}{\sqrt{2}}\matI_{d}$.
We also have $|\varphi(\x,\veta)|\leq1$ for all $\x$ and $\veta$,
so we set $M_{R}=1$. So, based on Proposition~\ref{prop:setting-U}
we set
\[
U_{k}=\frac{1}{\ell_{0}}\sqrt{2\ln\left(\left(\frac{2^{2-d}\sigma_{f0}^{2}n^{2}}{\sigma_{n0}^{2}}\right)^{1/d}\right)}
\]

In addition, $p(\cdot;\ell)$ is analytic on $\mathbb{R}^{d}$, and
in particular it is analytically continuable to the polyellipse $E_{\mat U,\vbeta}$
with $\beta_{k}=2U_{k},\,\rho_{k}=1+\sqrt{2},$ as described in Theorem
\ref{thm: quad-thm}. Now we bound $p(\cdot;\ell_{0})$ on the polyellipse
as follows, for $x_{k}\in[-\sqrt{2}U_{k},\sqrt{2}U_{k}],\,y_{k}\in[-U_{k},U_{k}]$:
\begin{eqnarray*}
\left|p(\x+i\y;\ell_{0})\right| & = & \left|\ell_{0}^{d}(2\pi)^{-d/2}\exp\left(-\ell_{0}^{2}/2\sum_{k=1}^{d}(x_{k}+iy_{k})^{2}\right)\right|\\
 & = & \left|\ell_{0}^{d}(2\pi)^{-d/2}\exp\left(-\ell_{0}^{2}/2\sum_{k=1}^{d}x_{k}^{2}-y_{k}^{2}+2ix_{k}y_{k}\right)\right|\\
 & = & \ell_{0}^{d}(2\pi)^{-d/2}\exp\left(\ell_{0}^{2}/2\sum_{k=1}^{d}y_{k}^{2}-x_{k}^{2}\right)\\
 & \leq & \ell_{0}^{d}(2\pi)^{-d/2}\exp\left(\ell_{0}^{2}/2\sum_{k=1}^{d}y_{k}^{2}\right)\\
 & \leq & \ell_{0}^{d}(2\pi)^{-d/2}\exp\left(\ell_{0}^{2}/2\sum_{k=1}^{d}U_{k}^{2}\right)\\
 & = & \ell_{0}^{d}(2\pi)^{-d/2}\exp\left(\ell_{0}^{2}\TNormS{\matU}/2\right)
\end{eqnarray*}
Hence, $C_{\matU,\vbeta}=\ell_{0}^{d}(2\pi)^{-d/2}\exp\left(\ell_{0}^{2}\TNormS{\matU}/2\right)$.

Recall that $\mat z:\R^{d}\to\C^{n}$ is given by $\z(\veta)_{j}=e^{-i\veta^{\T}\x_{j}}$,
it is easy to verify that for every $\v\in\mathbb{R}^{n}$ the function
$\left|\z(\cdot)^{\conj}\v\right|^{2}$ is an analytic function on
$\mathbb{R}^{d}$. In particular, it is analytically continuable to
the polyellipse $E_{\mat U,\vbeta}$ with $\beta_{k}=2U_{k},\,\rho_{k}=1+\sqrt{2}$
as required by Theorem \ref{thm: quad-thm}. Now we bound $\left|\z(\veta)_{j}\right|$
for each $j$ on its corresponding ellipse. For any $\x,\y\in\R^{d}:\,\left|\z(\x+i\y)_{j}\right|=\left|\z(i\y)_{j}\right|$,
so we need to bound $\left|\z(i\y)_{j}\right|$ for $\y\in\prod_{k=1}^{d}[-U_{k},U_{k}]$:
\begin{eqnarray*}
\left|\z(i\y)_{j}\right| & = & \left|e^{-i\cdot i\y^{\T}\x_{j}}\right|=e^{\y^{\T}\x_{j}}\leq e^{\TNorm{\y}\TNorm{\x_{j}}}\leq e^{\TNorm{\mat U}\TNorm{\mat R}/2}\,.
\end{eqnarray*}
Hence, $|\z(\veta)_{j}|\leq e^{\TNorm{\matU}\TNorm{\mat R}/2}\eqqcolon M_{\matU,\vbeta}$.
Now, one can apply Theorem ~\ref{thm:set_parameter} with the these
parameters and obtain that for 
\[
s_{k}\geq\frac{\frac{1}{d}\ln\left(2^{2d+2}\pi^{-d/2}\sigma_{n0}^{-2}\sigma_{f0}^{2}n^{2}\right)+\frac{\ell_{0}^{2}}{2d}\TNormS{\matU}+\frac{1}{d}\TNorm{\matU}\TNorm{\mat R}+\frac{1}{2}\ln\ln\left(\left(\frac{2^{2-d}\sigma_{f0}^{2}n^{2}}{\sigma_{n0}^{2}}\right)^{1/d}\right)-\ln(\sqrt{2})}{2\ln(1+\sqrt{2})}+1
\]
we have the desired bound.

Since $\TNorm{\matU}=O(\sqrt{\ln n})$ (in particular, $\TNormS{\matU}=O(\ln n)$)
and assuming the bounding box $\matR$ is fixed then $s_{k}=O(\ln n)$
suffice and $s=\prod_{k=1}^{d}s_{k}=O((\ln n)^{d})$ suffices. 

\subsection{Matèrn Kernel}

Recall that the Matèrn kernel is
\[
k_{\vtheta}(\x,\x')=\sigma_{f}^{2}\frac{2^{1-\nu}}{\Gamma(\nu)}\left(\frac{\sqrt{2\nu}}{\ell}\TNorm{\x-\x'}\right)^{\nu}K_{\nu}\left(\frac{\sqrt{2\nu}}{\ell}\TNorm{\x-\x'}\right)+\sigma_{n}^{2}\gamma(\x-\x')
\]
($\vtheta=[\ell,\sigma_{f}^{2},\sigma_{n}^{2}]$) where we added a
scaling factor $\sigma_{f}^{2}$ and an included the ridge term $\sigma_{n}^{2}\gamma(\x-\x')$
in the kernel definition. Note that for conciseness, we consider the
isotropic version for a fixed $\nu$; the formulas can be modified
for the anisotropic case. As in the case of the Gaussian kernel, by
setting $\varphi(\x,\veta)=e^{-i\x^{\T}\veta}$ , this kernel matches
the form of Eq.~(\ref{eq:kernel-form}). We assume that the hyperparameters
are bounded as follows:
\[
\Theta=\{[\ell,\sigma_{n}^{2},\sigma_{f}^{2}]\,:\,\ell\geq\ell_{0},\sigma_{n}^{2}\geq\sigma_{n0}^{2},\sigma_{f}^{2}\leq\sigma_{f0}^{2}\}
\]
where $\sigma_{n0}^{2}>0$ (i.e. we have a ridge term; our method
is not able to approximate the Matèrn kernel in the absence of a ridge
term).

The density is given by
\[
p(\veta;\ell)=\frac{\Gamma(\nu+d/2)\ell^{d}}{\Gamma(\nu)\left(2\nu\pi\right)^{d/2}}\left(1+\frac{\ell^{2}}{2\nu}\TNormS{\veta}\right)^{-\left(\nu+d/2\right)}\,.
\]
Therefore $p\in{\cal P}_{C,\matL}^{(r)}$ where we consider $C=\frac{\Gamma(\nu+d/2)\ell_{0}^{d}}{\Gamma(\nu)\left(2\nu\pi\right)^{d/2}}$,
$\matL=\frac{\ell_{0}}{\sqrt{2\nu}}\matI_{d}$ and $r=\nu+d/2>d/2$,
i.e., $2r-d=2\nu$. So we set $U$ to be the numerical solution of
Eq.~(\ref{eq:Matern_highD_eq}) for $d=1$, and we set $U_{k}$ to
be the numerical solution of Eq.~(\ref{eq:Matern_highD_eq}) for
$d\geq2$. In addition, $p(\cdot;\ell)$ is analytic on $\mathbb{R}^{d}$
and it is analytically continuable to the polyellipse $E_{\mat U,\vbeta}$
with $\beta_{k}=\frac{\sqrt{2\nu}}{\ell_{0}\sqrt{d}},\,\rho_{k}=\frac{\sqrt{2\nu}}{2\ell_{0}\sqrt{d}U_{k}}+\sqrt{\frac{2\nu}{4\ell_{0}^{2}dU_{k}^{2}}+1}$,
as required by Theorem \ref{thm: quad-thm}. Now we bound $p(\cdot;\ell_{0})$
on the polyellipse as follows, for $x_{k}\in\left[-\sqrt{\frac{\nu}{2\ell_{0}^{2}d}+U_{k}^{2}},\sqrt{\frac{\nu}{2\ell_{0}^{2}d}+U_{k}^{2}}\right],\,y_{k}\in\left[-\frac{\sqrt{2\nu}}{2\ell_{0}\sqrt{d}},\frac{\sqrt{2\nu}}{2\ell_{0}\sqrt{d}}\right]$:

\begin{eqnarray*}
\left|p(\x+i\y;\ell_{0})\right| & = & \frac{\Gamma(\nu+d/2)\ell_{0}^{d}}{\Gamma(\nu)\left(2\nu\pi\right)^{d/2}}\left|1+\frac{\ell_{0}^{2}}{2\nu}\sum_{k=1}^{d}(x_{k}+iy_{k})^{2}\right|^{-\left(\nu+d/2\right)}\\
 & \leq & \frac{\Gamma(\nu+d/2)\ell_{0}^{d}}{\Gamma(\nu)\left(2\nu\pi\right)^{d/2}}\left(1-\frac{\ell_{0}^{2}}{2\nu}\sum_{k=1}^{d}y_{k}^{2}\right)^{-\left(\nu+d/2\right)}\\
 & \leq & \frac{\Gamma(\nu+d/2)\ell_{0}^{d}}{\Gamma(\nu)\left(2\nu\pi\right)^{d/2}}\left(1-\frac{\ell_{0}^{2}}{2\nu}\sum_{k=1}^{d}\frac{\nu}{2\ell_{0}^{2}d}\right)^{-\left(\nu+d/2\right)}\\
 & = & \frac{\Gamma(\nu+d/2)\ell_{0}^{d}}{\Gamma(\nu)\left(2\nu\pi\right)^{d/2}}\left(\frac{3}{4}\right)^{-\left(\nu+d/2\right)}
\end{eqnarray*}
where the maximum value is obtained at the nearest points to the poles:
$y_{k}=\pm\frac{\sqrt{2\nu}}{2\ell_{0}d}$. Hence, $C_{\matU,\vbeta}\coloneqq\frac{\Gamma(\nu+d/2)\ell_{0}^{d}}{\Gamma(\nu)\left(2\nu\pi\right)^{d/2}}\left(\frac{3}{4}\right)^{-\left(\nu+d/2\right)}$.

Recall that $\mat z:\R^{d}\to\C^{n}$ is given by $\mat z(\veta)_{j}=e^{-i\veta^{\T}\x_{j}}$.
Now we bound $\left|\mat z(\veta)_{j}\right|$ for each $j$ on its
corresponding ellipse. For any $\x,\y\in\R^{d}:\,\left|\mat z(\x+i\y)_{j}\right|=\left|\mat z(i\y)_{j}\right|$,
so we need to bound $\left|\mat z(i\y)_{j}\right|$ for $\y\in\prod_{k=1}^{d}\left[-\beta_{k}/2,\beta_{k}/2\right]$:
\begin{eqnarray*}
\left|\mat z(i\y)_{j}\right| & = & \left|e^{-i\cdot i\y^{\T}\x_{j}}\right|=e^{\y^{\T}\x_{j}}\leq e^{\TNorm{\y}\TNorm{\x_{j}}}\leq e^{\TNorm{\mat{\vbeta}}\TNorm{\mat R}/4}
\end{eqnarray*}
Hence, $|\mat z(\veta)_{j}|\leq e^{\TNorm{\mat{\vbeta}}\TNorm{\mat R}/4}\eqqcolon M_{\matU,\vbeta}$.
Now, for the asymptotics, consider the upper bound for $\matU$ in
Theorem ~\ref{thm:set_parameter}. If one denotes $\gamma=\ln\left(2^{2d+2}\pi^{-d/2}\frac{\Gamma(\nu+d/2)}{\Gamma(\nu)}\left(\frac{3}{4}\right)^{-\left(\nu+d/2\right)}\sigma_{n0}^{-2}\sigma_{f0}^{2}n^{2}\right)$
and $\delta=\ln\left(\frac{\sigma_{n0}^{-2}\sigma_{f0}^{2}n^{2}}{2^{d-1}\nu\text{B}\left(\nu,\frac{d}{2}\right)}\right)$,
then Theorem ~\ref{thm:set_parameter} can be applied with the these
parameters and obtain that (for $d\geq3$){\footnotesize{}
\[
s_{k}\geq\frac{\frac{1}{2d}\TNorm{\mat{\vbeta}}\TNorm{\mat R}+\frac{\gamma}{d}+\frac{\delta}{2\nu}-\ln\left(\beta_{k}/(2U_{k})-1+\sqrt{\beta_{k}^{2}/(4U_{k}^{2})+1}\right)}{2\ln\left(\beta_{k}/(2U_{k})+\sqrt{\beta_{k}^{2}/(4U_{k}^{2})+1}\right)}+1
\]
}we have the desired bound.

Assuming the bounding box $\matR$ is fixed, $s_{k}=O(\ln n)$ suffice
and $s=\prod_{k=1}^{d}s_{k}=O((\ln n)^{d})$ suffices.

\subsection{\label{subsec:semigroup}Semigroup Kernels}

The previous two examples were of shift-invariant kernels, and the
feature mapping $\varphi$ was based on Bochner's theorem. In this
section, we demonstrate the application of our theory to a different
type of kernels: semigroup kernels~\cite{YangEtAl14}. These type
of kernels require a slight modification of our setup, which we briefly
describe below, but adjusting theory itself is technical and we omit
it.

Semigroup kernels are well-suited for non-negative data, i.e. $\X\subseteq\R_{+}^{d}$,
and require that the kernel value at $\x$ and $\x'$ depends only
on the sum $\x+\x'$: $k(\x,\x')=k_{0}(\x+\x'$). One example of such
kernel is the reciprocal semigroup kernel: 
\[
k_{\vtheta}(\x,\x')=\sigma_{f}^{2}\prod_{k=1}^{d}\frac{\lambda}{x_{k}+x_{k}'+\lambda}+\sigma_{n}^{2}\gamma(\x-\x')
\]
($\vtheta=[\lambda,\sigma_{f}^{2},\sigma_{n}^{2}]$) where we add
a scaling factor $\sigma_{f}^{2}$ and an included the ridge term
$\sigma_{n}^{2}\gamma(\x-\x')$ in the kernel definition. It can be
shown that every semigroup kernel can be written in the following
integral form \cite{berg1984harmonic}, which is analogous to~Eq.~(\ref{eq:base-kernel}):
\[
k_{\vtheta}(\x,\x')=\sigma_{f}^{2}\int_{\R_{+}^{d}}e^{-\veta^{\T}(\x+\x')}p(\veta;\lambda)d\veta+\sigma_{n}^{2}\gamma(\x-\x')
\]
where $p(\cdot;\lambda)$ is a probability density function which
is supported only on $\R_{+}^{d}$. For the reciprocal semigroup kernel
we have $p(\veta;\lambda)=\lambda e^{-\lambda\ONorm{\veta}1}=\lambda e^{-\lambda\sum_{k=1}^{d}\eta_{k}}$,
so $p(\cdot;\lambda)\in{\cal E}_{C,\matL}^{(1)}$with $C=\lambda,\,\mat L=\lambda\mat I_{d}$,
which is analytic on $\mathbb{R}_{+}^{d}$. Thus, we see that semigroup
kernels can be represented as
\[
k_{\vtheta}(\x,\x')=\sigma_{f}^{2}\int_{\R_{+}^{d}}\varphi(\x,\veta)\varphi(\x',\veta)^{*}p(\veta;\vtheta_{0})d\veta+\sigma_{n}^{2}\gamma(\x-\x')
\]
which is almost the same as Eq.~(\ref{eq:kernel-form}), except the
integration area is $\R_{+}^{d}$ instead of $\R^{d}$. For semigroup
kernels $\varphi(\x,\veta)=e^{-\veta^{\T}\x}$.

The construction of Gauss-Legendre features is quite similar to the
integration area is $\R^{d}$, except that we replace the assumption
that $\X\subseteq\prod_{k=1}^{d}[-R_{k}/2,R_{k}/2]$ with $\X\subseteq\prod_{k=1}^{d}[0,R_{k}]$,
the truncated integration area ${\cal Q}_{\matU}$ with ${\cal H}_{\matU}\coloneqq\prod_{k=1}^{d}[0,U_{k}]$,
and the integration nodes and weights are obtained by linearly transforming
${\cal H}_{\matU}$ (instead of $\mathcal{Q}_{\matU}$) to $[-1,1]^{d}$
with the transformation $\eta_{k}=U_{k}\cdot\frac{\chi_{k}+1}{2}$
for $k=1,\dots,d$. We omit the details of the construction, since
they mostly repeat the construction described in Section~\ref{sec:parameters}.

Now consider the reciprocal semigroup kernel. We assume that the hyperparameters
are bounded as follows:
\[
\Theta=\{[\lambda,\sigma_{n}^{2},\sigma_{f}^{2}]\,:\,\lambda\geq\lambda_{0},\sigma_{n}^{2}\geq\sigma_{n0}^{2},\sigma_{f}^{2}\leq\sigma_{f0}^{2}\}
\]
where $\sigma_{n0}^{2}>0$ (i.e., we have a ridge term). We set
\[
U_{k}=\frac{1}{\lambda_{0}}\ln\left(\left(\frac{2\lambda_{0}^{1-d}\sigma_{f0}^{2}n^{2}}{\sigma_{n0}^{2}}\right)^{1/d}\right)
\]

In addition, since $p(\cdot;\lambda)$ is analytic on $\mathbb{R}_{+}^{d}$,
and in particular it is analytically continuable to the polyellipse
$E_{\mat U,\vbeta}$ with $\beta_{k}=2U_{k},\,\rho_{k}=1+\sqrt{2}$.
Now we bound the analytic continuation of $p(\cdot;\lambda_{0})$
(which we also denote by $p(\cdot;\lambda_{0})$) on the polyellipse
as follows. For any $\x\in\mathbb{R}^{d},\,\y\in\mathbb{R}^{d}$,
$\left|p(\x+i\y;\lambda_{0})\right|=\left|\lambda_{0}e^{-\lambda_{0}\sum_{k=1}^{d}x_{k}+iy_{k}}\right|=\left|p(\x;\lambda_{0})\right|$
, so we need to bound $\left|p(\x;\lambda_{0})\right|$ for $x_{k}\in[U_{k}\frac{1-\sqrt{2}}{2},U_{k}\frac{1+\sqrt{2}}{2}],\,y_{k}\in[0,U_{k}]$:
\begin{eqnarray*}
\left|p(\x;\lambda_{0})\right| & = & \lambda_{0}e^{-\lambda_{0}\sum_{k=1}^{d}x_{k}}\leq\lambda_{0}e^{\lambda_{0}\sum_{k=1}^{d}U_{k}\frac{\sqrt{2}-1}{2}}=\lambda_{0}e^{\lambda_{0}\frac{\sqrt{2}-1}{2}\ONorm{\matU}1}
\end{eqnarray*}
Hence, $C_{\matU,\vbeta}\coloneqq\lambda_{0}e^{\lambda_{0}\frac{\sqrt{2}-1}{2}\ONorm{\matU}1}$.

For semigroup kernels we use $\mat z:\mathbb{R}_{+}^{d}\to\R^{n}$
defined by $\mat z(\veta)_{j}=e^{-\veta^{\T}\x_{j}}$ as the feature
map. It can be seen that for every $\v\in\mathbb{R}^{n}$ the function
$\left|\mat z(\cdot)^{\conj}\v\right|^{2}$ is an analytic function
on $\mathbb{R}_{+}^{d}$, and we also have $|\mat z(\veta)_{j}|\leq1=M_{R}$
for $\veta\geq0$ and for all $1\leq j\leq n$. In particular, it
is analytically continuable to the pollyellipse $E_{\mat U,\vbeta}$
with $\beta_{k}=2U_{k},\,\rho_{k}=1+\sqrt{2}$ . Now we bound $|\mat z(\veta)_{j}|^{2}$,
where here $\z$ denotes the analytic continuation. Notice that for
any $\x\in\mathbb{R}^{d},\,\y\in\R^{d}$, $\left|\mat z(\x+i\y)_{j}\right|=\left|\mat z(\x)_{j}\right|$
, so we need to bound $\left|\mat z(\x)_{j}\right|$ for each $x_{k}\in[U_{k}\frac{1-\sqrt{2}}{2},U_{k}\frac{1+\sqrt{2}}{2}],\,y_{k}\in[0,U_{k}]$:\textcolor{red}{{}
}
\begin{eqnarray*}
\left|\mat z(\x)_{j}\right| & = & e^{-\x^{\T}\x_{j}}=e^{-\sum_{k=1}^{d}x_{k}(\x_{j})_{k}}\leq e^{\sum_{k=1}^{d}(U_{k}\frac{\sqrt{2}-1}{2}R_{k})}=e^{\frac{\sqrt{2}-1}{2}\matU^{\T}\mat R}
\end{eqnarray*}
and each $\x_{j}$ satisfies $(\x_{j})_{k}\leq R_{k}$. Hence, $|\mat z(\veta)_{j}|\leq e^{\frac{\sqrt{2}-1}{2}\matU^{\T}\mat R}\eqqcolon M_{\matU,\vbeta}$.
So we set{\footnotesize{}
\[
{\normalcolor s_{k}\geq\left\lceil \frac{\frac{1}{d}\ln\left(2^{2d+2}\lambda_{0}\sigma_{n}^{-2}\sigma_{f}^{2}n^{2}\right)+\ln\left(\frac{\sqrt{2}-1}{2d}\ONorm{\matU}1\right)+\frac{\sqrt{2}-1}{d}\matU^{\T}\mat R+\ln\ln\left(\left(\frac{2\lambda_{0}^{1-d}\sigma_{f0}^{2}n^{2}}{\sigma_{n0}^{2}}\right)^{1/d}\right)-\ln(\sqrt{2})}{2\ln(1+\sqrt{2})}\right\rceil +1}
\]
}Since $\ONorm{\matU}1=O(\ln n)$ and assuming the bounding box $\matR$
is fixed, then $s_{k}=O(\ln n)$ suffice and $s=\prod_{k=1}^{d}s_{k}=O((\ln n)^{d})$
suffices. 

\section{\label{sec:Experiments}Numerical Experiments}

In this section we report experiments evaluating the performance of
our proposed quadrature based approach. Our goal is to show that indeed
if $\matU$ and $\s$ are set to be large enough, our method yields
results that are essentially indistinguishable from using the exact
kernel, while offering faster hyperparameter learning, training and
prediction. Clearly, from the theoretical results, our method predominately
applies to low-dimensional datasets (for example, such datasets are
prevalent in spatial statistics), so we experiment with one dimensional
and two dimensional datasets. We experiment both with the Gaussian
kernel or the Matèrn kernel.

In the graphs, we label our method as GLF-GPR (standing for Gauss-Legendre
Features Gaussian Process Regression). We use the following methods
as benchmark: exact GPR (labeled in the graphs as Exact-GPR) and GPR
based on random Fourier features (labeled RFF-GPR). As performance
metric we use the MSE error on a test set (as a function of number
of features) and the time to learn the hyperparameters. Training
and prediction time of both GLF-GPR and RFF-GPR are essentially the
same for the same number of features, and both are faster than Exact-GPR
if the number of features is smaller then the training set size. Thus,
when it comes to training and prediction time, it is sufficient to
explore the test error as function of the number of features. However,
hyperparameter learning time can vary considerably between GLF-GPR
and RFF-GPR, so we compare this quantity directly.

The various methods were implemented in MATLAB. Optimizing the hyperparameters
was conducted using the MATLAB function \texttt{fmincon} after transforming
the hyperparameters to a logarithmic scale. For each problem we defined
a hyperparameter domain, e.g. 
\[
\Theta=\{[\ell,\sigma_{n}^{2},\sigma_{f}^{2}]\,:\,\ell_{0}\leq\ell\leq\ell_{1},\sigma_{n0}^{2}\leq\sigma_{n}^{2}\leq\sigma_{n1}^{2},\sigma_{f1}^{2}\leq\sigma_{f}^{2}\leq\sigma_{f0}^{2}\}
\]
 and we take the initial hyperparameters for the optimization to be
$[\ell_{0},\sigma_{f0}^{2},\sigma_{n0}^{2}]$. Running times were
measured on a machine with two 3.2GHz Intel(R) Xeon(R) Gold 6134 CPUs,
each having 8 cores, and 256GB RAM.

\subsection{Synthetic Data}

In this subsection, we report experiments on synthetically generated
data. The data is generated by noisily sampling a predetermined function,
i.e. samples are generated from the formula 
\[
y_{i}=f^{\star}(x_{i})+\tau_{i}
\]
where $f^{\star}$ is the true function and $\{\tau_{i}\}$ are i.i.d
noise terms, distributed as normal variables with variance $\sigma_{\tau}^{2}=0.5^{2}$
(for 1D) or $\sigma_{\tau}^{2}=0.3^{2}\matI_{2}$ (for 2D). In these
experiments we use the isotropic Gaussian kernel.

First, we consider a one dimensional function:
\begin{equation}
f_{1}^{\star}(x)=\sin(2x)+\sin(6e^{x})\label{eq:wiggly-1D}
\end{equation}
The function was sampled equidistantly on $[-1,1]$ with $n=800$
samples. The results are reported in Figure~\ref{fig:wiggly-1D},
where we show how GLF-GPR with the number of quadrature points $s$
compared to RFF-GPR.

\begin{figure}[H]
\centering{}%
\begin{tabular}{cc}
\includegraphics[width=0.45\textwidth]{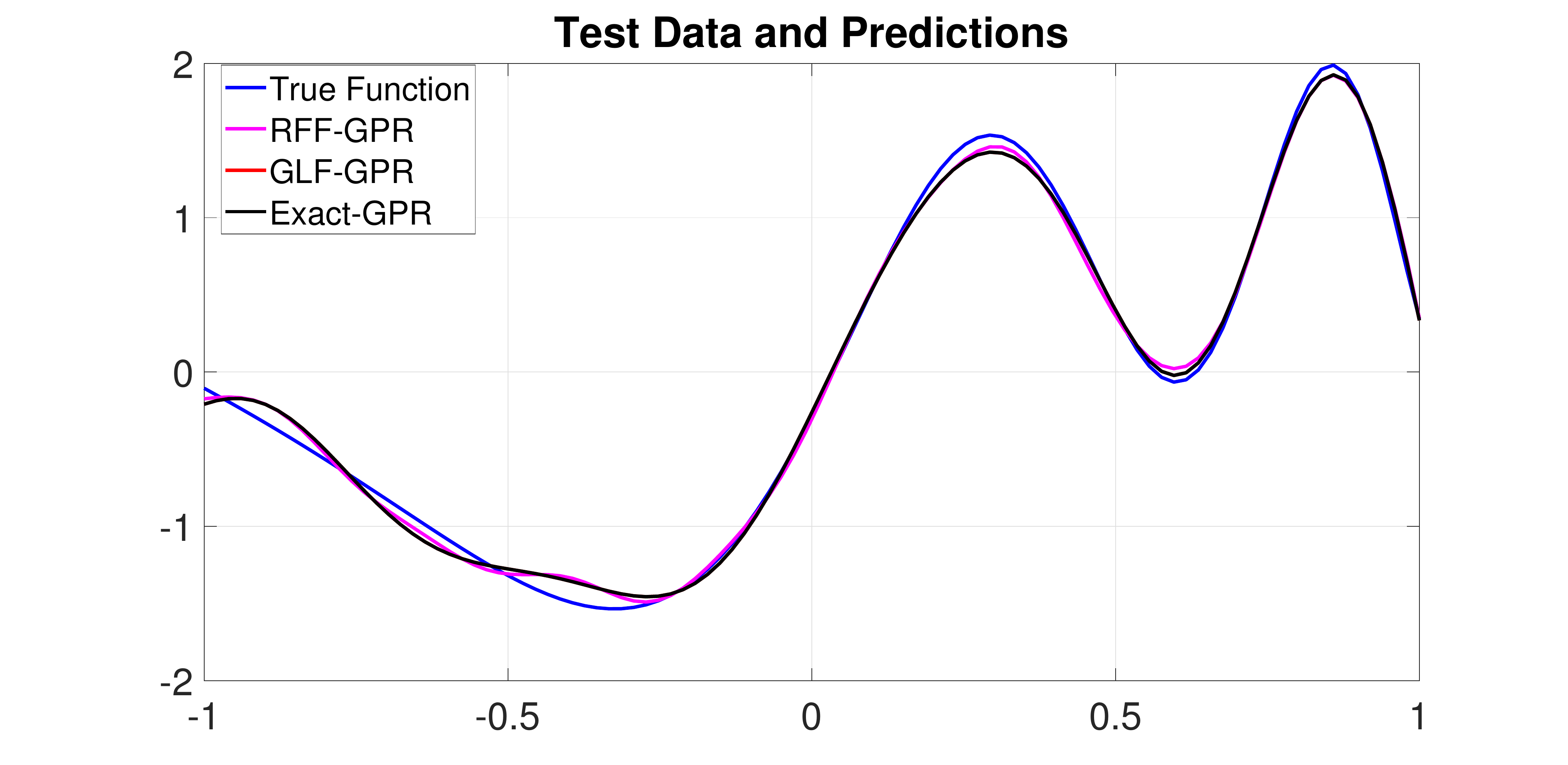} & \includegraphics[width=0.45\textwidth]{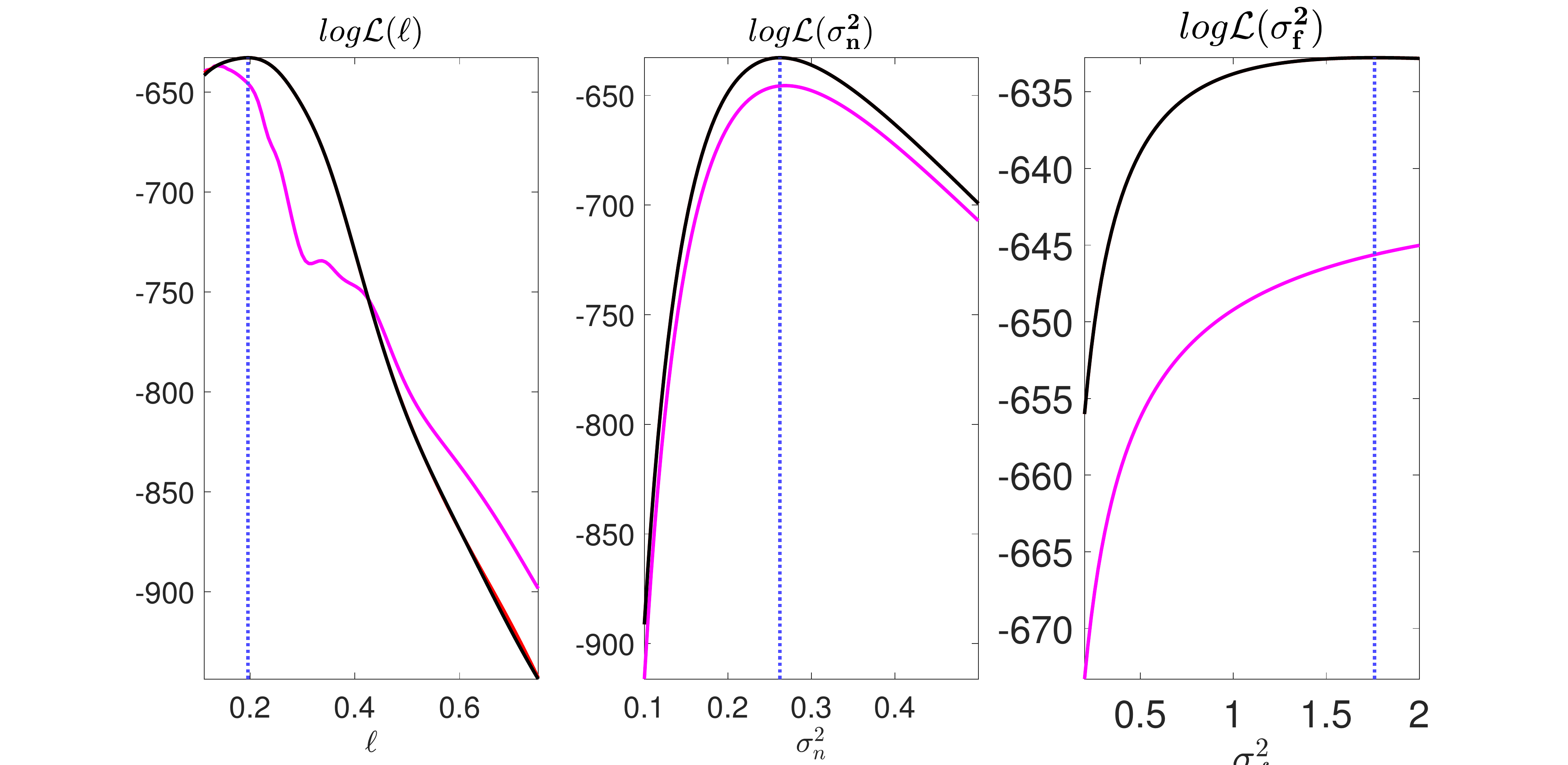}\tabularnewline
\includegraphics[width=0.45\textwidth]{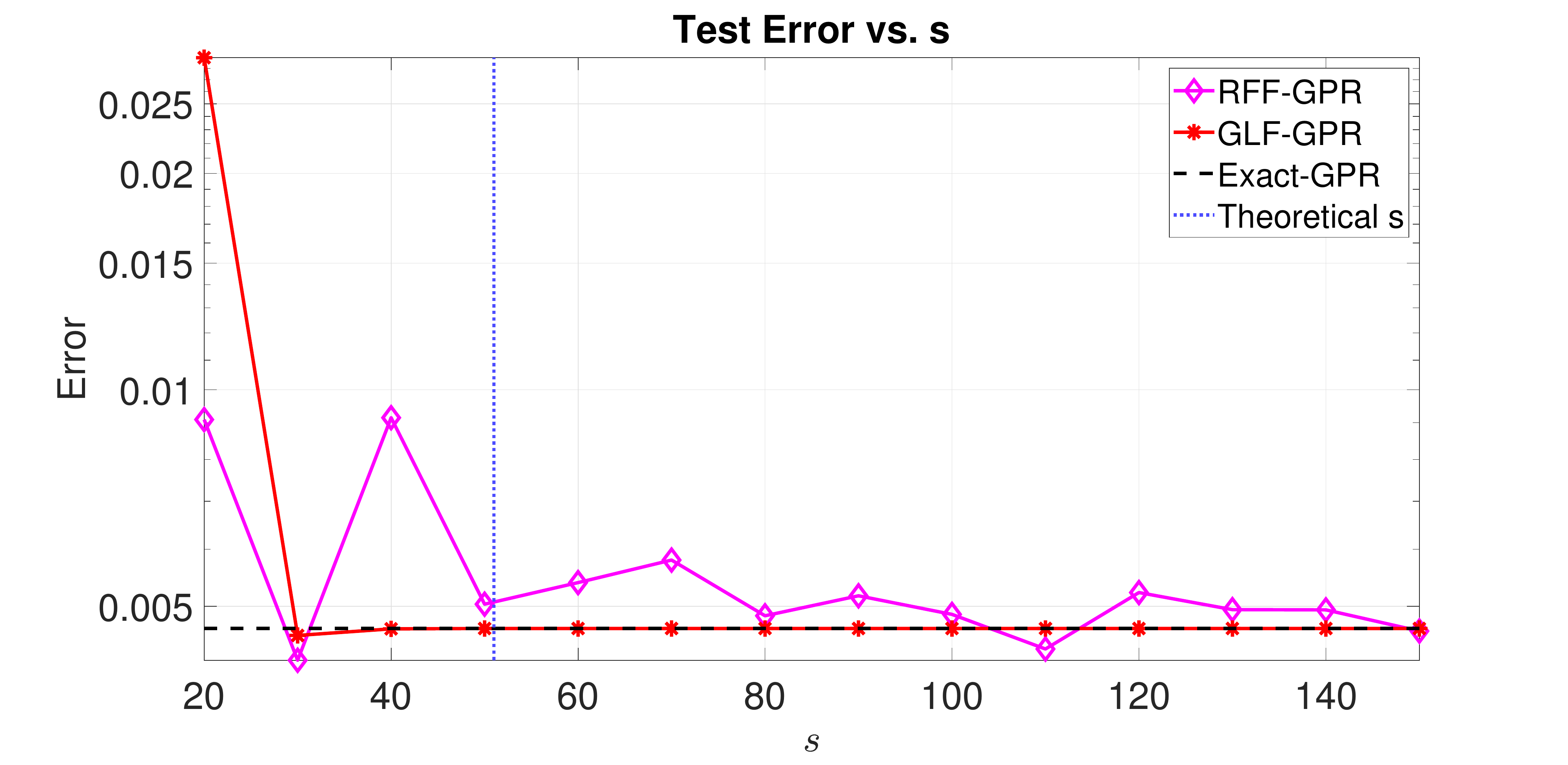} & \includegraphics[width=0.45\textwidth]{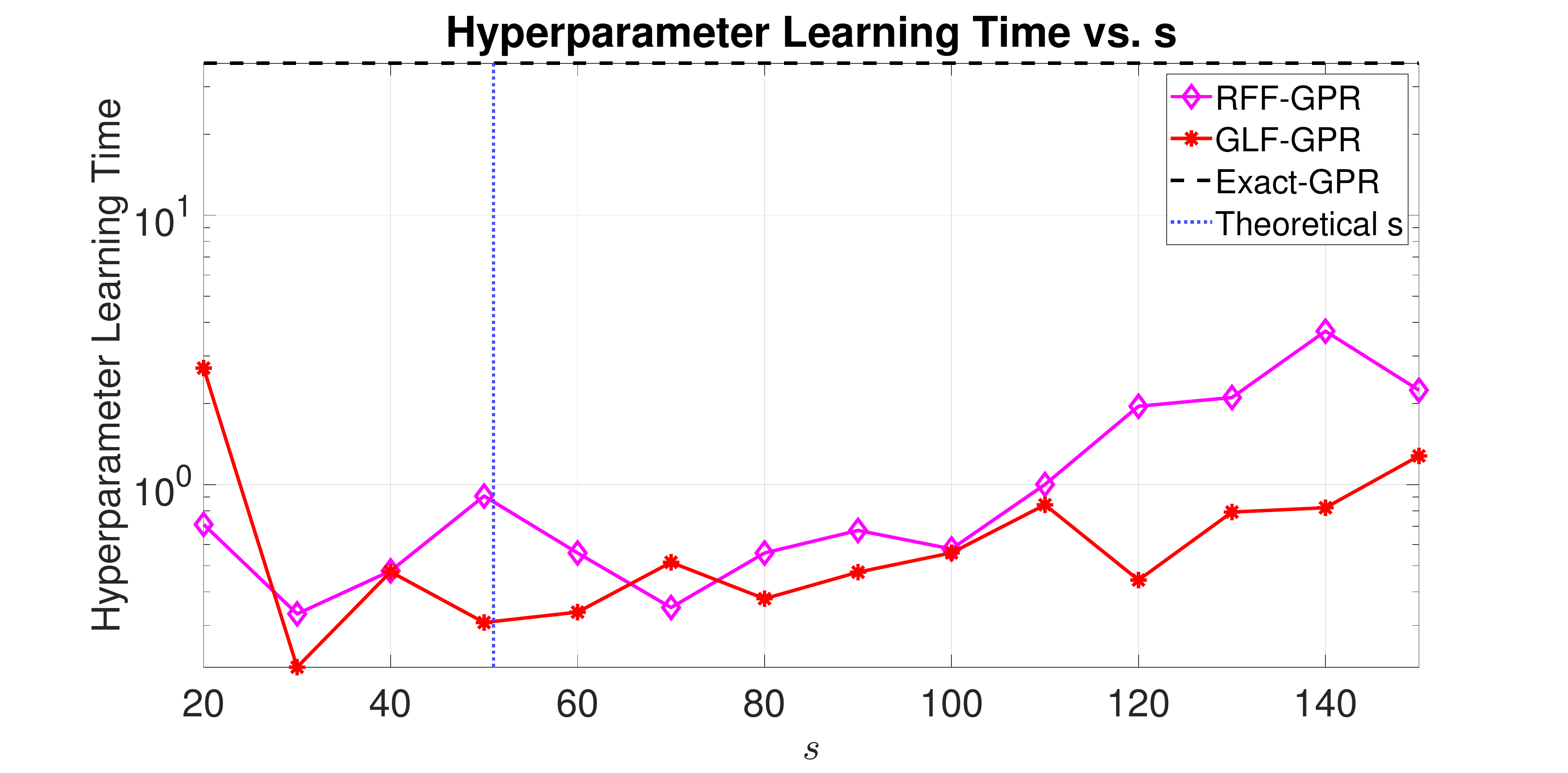}\tabularnewline
\end{tabular}\caption{\label{fig:wiggly-1D}Result for data generated using the function
$f_{1}^{\star}$. Top-left: true function, input data, and prediction.
Top-right: log-likelihood for various sections of the hyperparameter
range. Bottom-Left: Test error as a function of the number of features.
Bottom-Right: hyperparameter learning time.}
\end{figure}

In the top-right graph, we see that the log-likelihood of GLF-GPR
merges with the log-likelihood of Exact-GPR for each of the hyperparameters,
where the optimal hyperparameters are dashed in blue. RFF-GPR deviates
considerably. This graph exemplifies that GLF-GPR can yield good approximation
to the exact log-likelihoods, while RFF-GPR yields a poor approximation.
We also see that GLF-GPR optimizes hyperparameters that are much closer
to the exact values than RFF-GPR. The bottom-left plot shows the MSE
error on the same test points. We see that the GLF-GPR error stabilizes
on error of Exact-GPR even before the theoretical value of $s$. The
bottom-right graph shows the runtime of the hyperparameter learning
phase for different values of quadrature points $s$. GLF-GPR is clearly
more efficient than Exact-GPR and mostly more efficient than RFF-GPR.
As expected, as $s$ becomes larger, GLF-GPR learn the hyperparameters
much faster than Exact-GPR and RFF-GPR. Furthermore, GLF-GPR achieves
a low error rate with less features than RFF-GPR, and thus is able
to do training, prediction and hyperparameter learning much faster
than RFF-GPR.

Next, we consider a two dimensional function:
\begin{equation}
f_{2}^{\star}(x_{1},x_{2})=(\sin(x_{1})+\sin(10e^{x_{1}}))(\sin(x_{2})+\sin(10e^{x_{2}}))\label{eq:wiggly-2D}
\end{equation}
The function was sampled on an uniform grid on $[-1,1]\times[-1,1]$
with $n=4096$ samples. We consider $\ell_{1}=\ell_{2}$ so $U_{1}=U_{2}$
and $s_{1}=s_{2}$, i.e., $s=s_{1}^{2}$. The results are reported
in Figure~\ref{fig:wiggly-1D}.

\begin{figure}[H]
\centering{}%
\begin{tabular}{cc}
\includegraphics[width=0.45\textwidth]{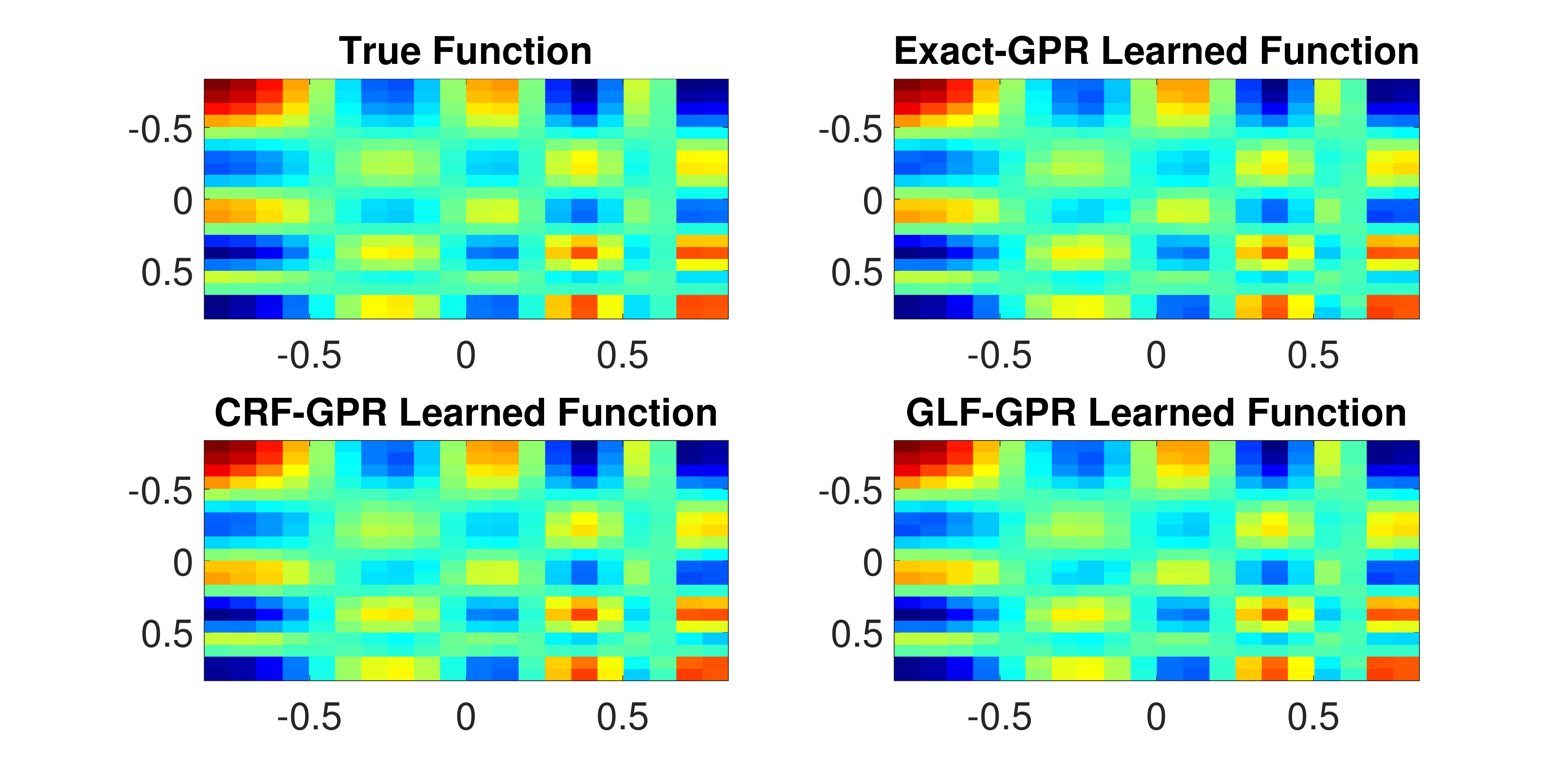} & \includegraphics[width=0.45\textwidth]{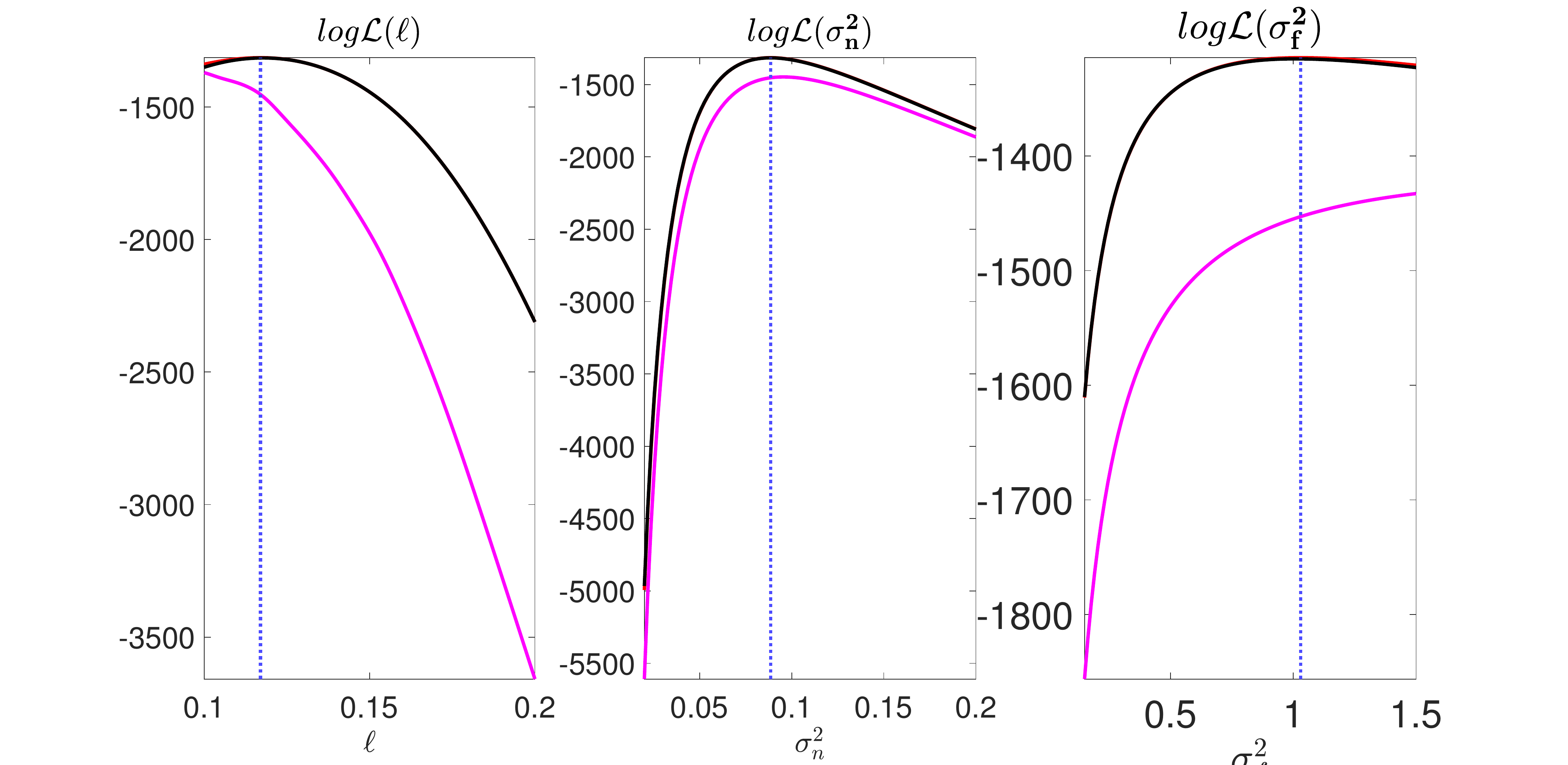}\tabularnewline
\includegraphics[width=0.45\textwidth]{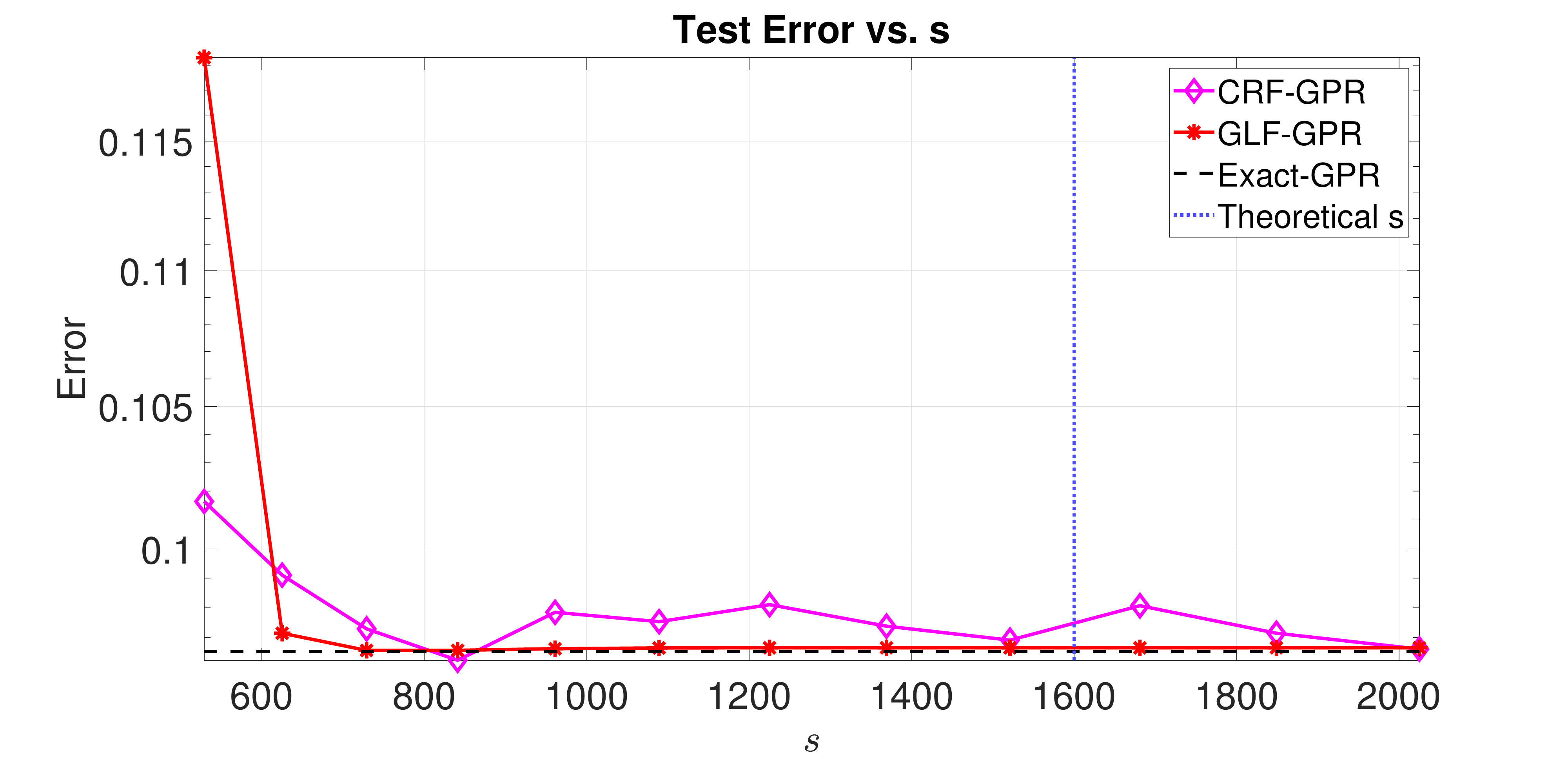} & \includegraphics[width=0.45\textwidth]{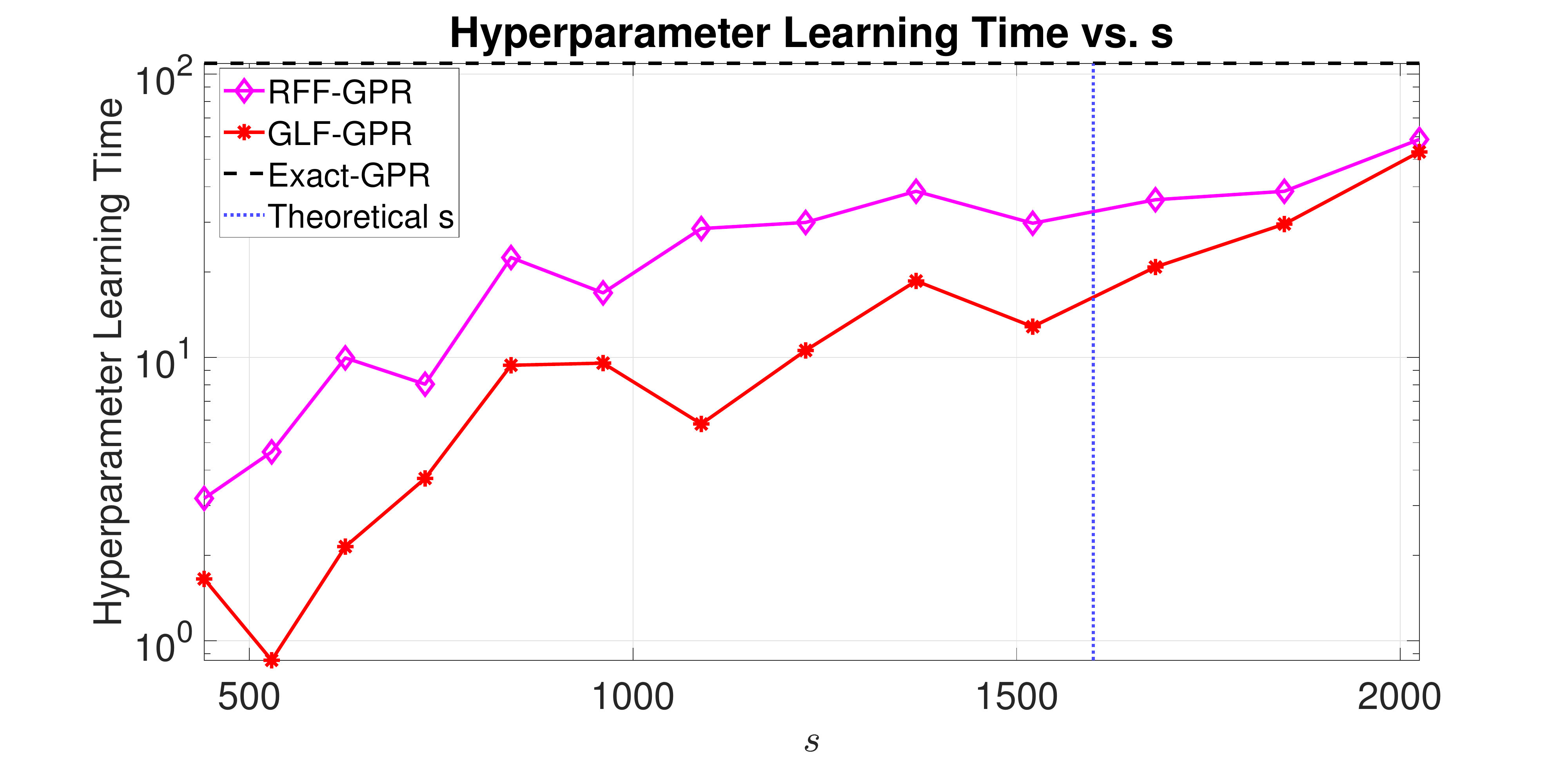}\tabularnewline
\end{tabular}\caption{\label{fig:wiggly-2D}Result for data generated using the function~$f_{2}^{\star}$.
Top-left:true function, input data, and prediction. Top-right: log-likelihood
for various sections of the hyperparameter range. Bottom-Left: Test
error as a function of the number of features. Bottom-Right: hyperparameter
learning time.}
\end{figure}

Similar to the synthetic 1D experiment, in the top-right plot we see
that GLF-GPR yields a good approximation to the exact log-likelihood.
In the bottom-left plot we see that shows the GLF-GPR error stabilizes
on error of Exact-GPR at a much smaller number of quadrature points
than the theoretical $s$.

\subsection{Natural Sound Modeling}

Next we consider the natural sound benchmark used in \cite{wilson2015kernel}
(without hyperparameter learning) and \cite{dong2017scalable} (with
hyperparameter learning). The data is shown in the top-left graph
of Figure~\ref{fig:sound-results}. The goal is to recover contiguous
missing regions in a waveform with $n=59309$ training points. The
test consists of 691 samples. The Gaussian kernel is used for learning.

\begin{figure}[H]
\begin{centering}
\begin{tabular}{cc}
\includegraphics[width=0.45\textwidth]{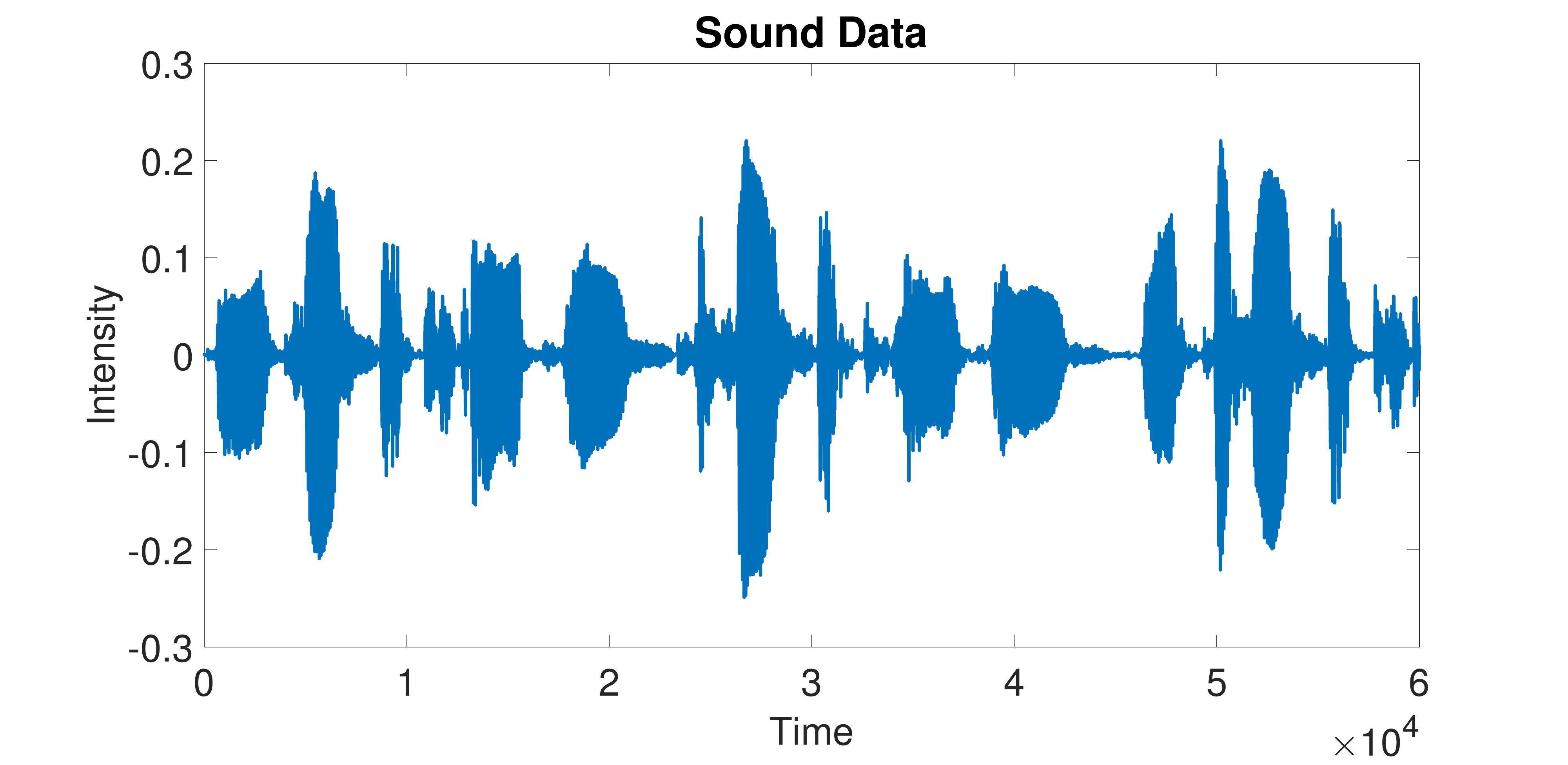} & \includegraphics[width=0.45\textwidth]{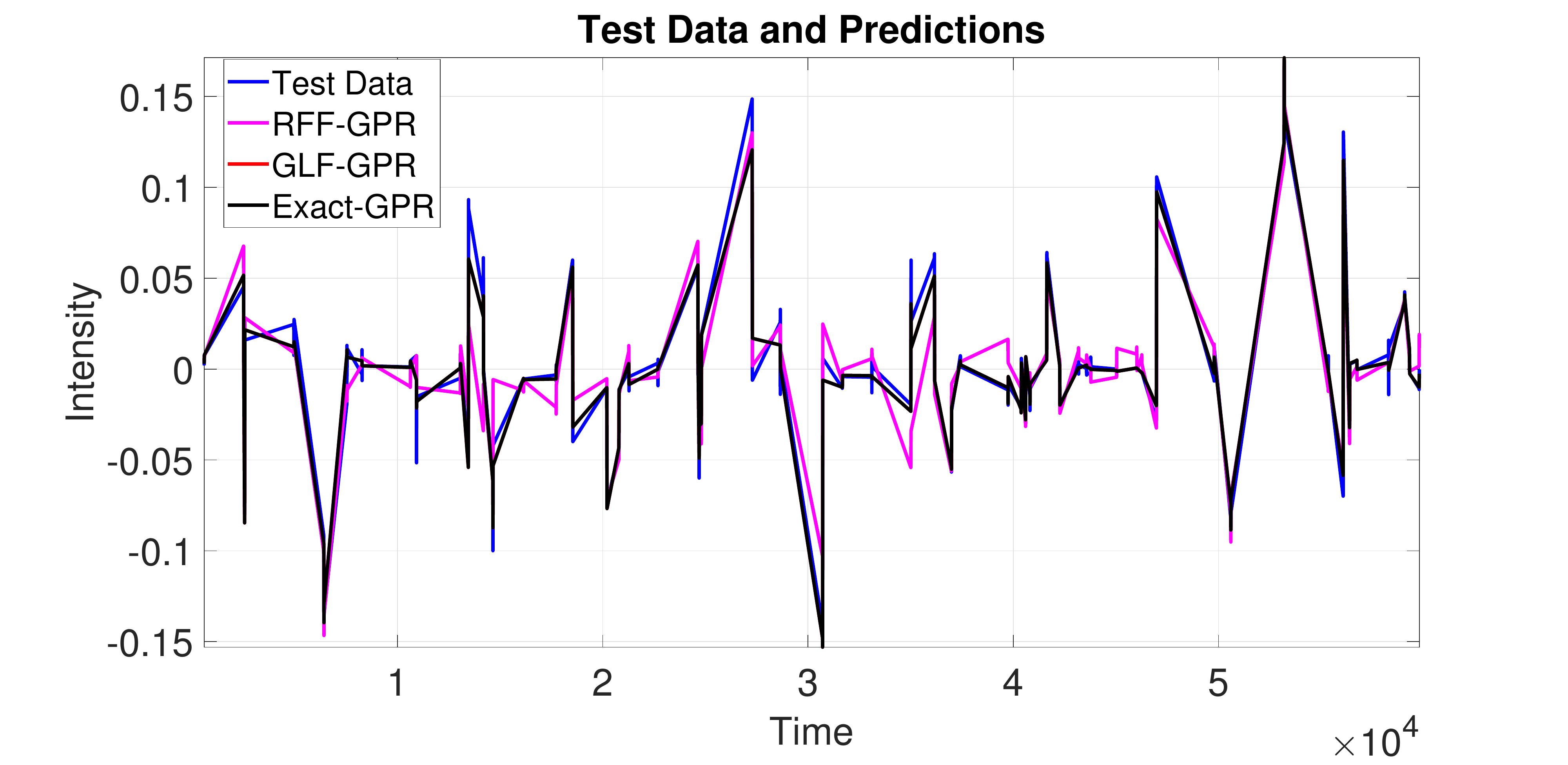}\tabularnewline
\includegraphics[width=0.45\textwidth]{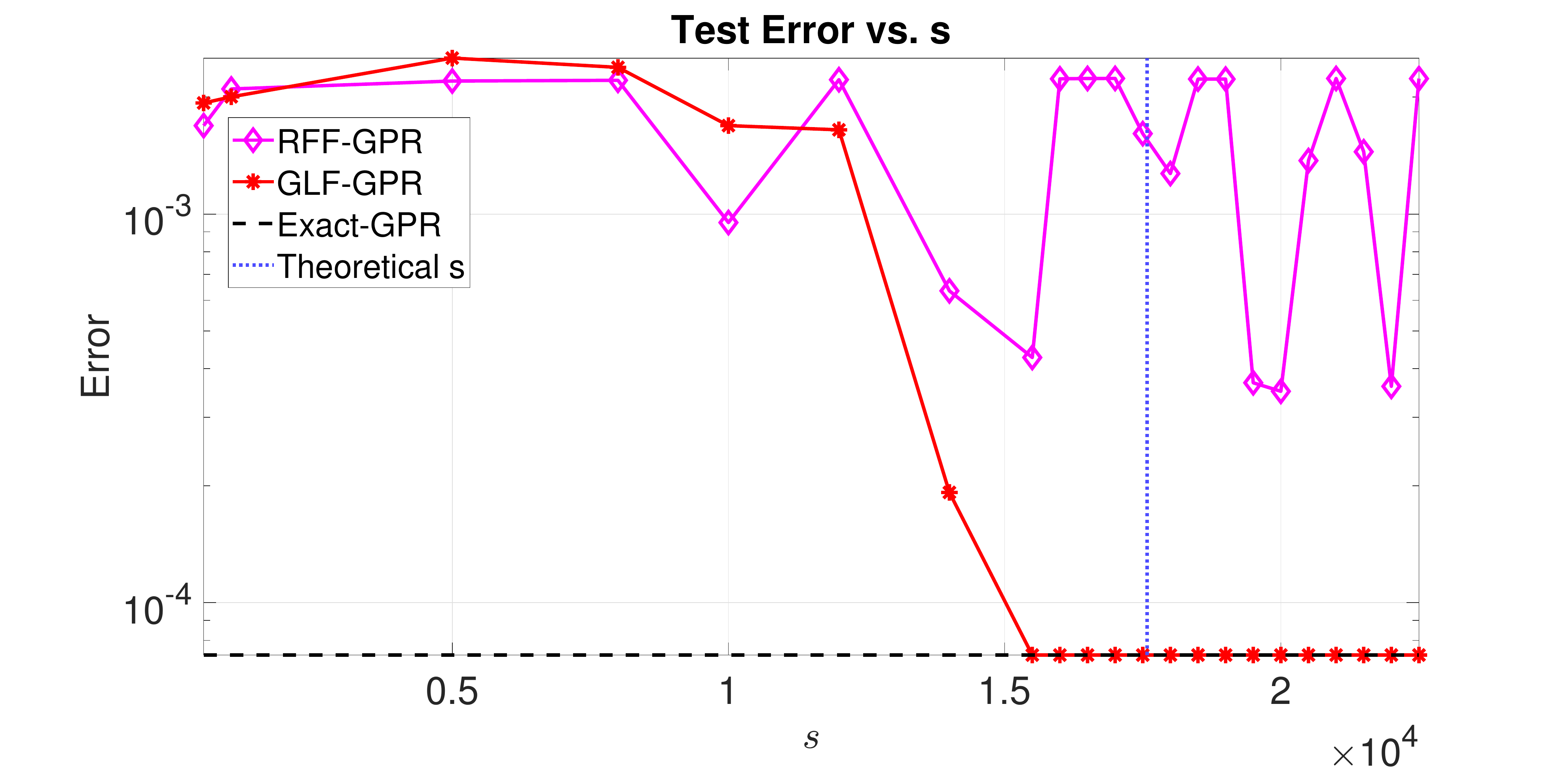} & \includegraphics[width=0.45\textwidth]{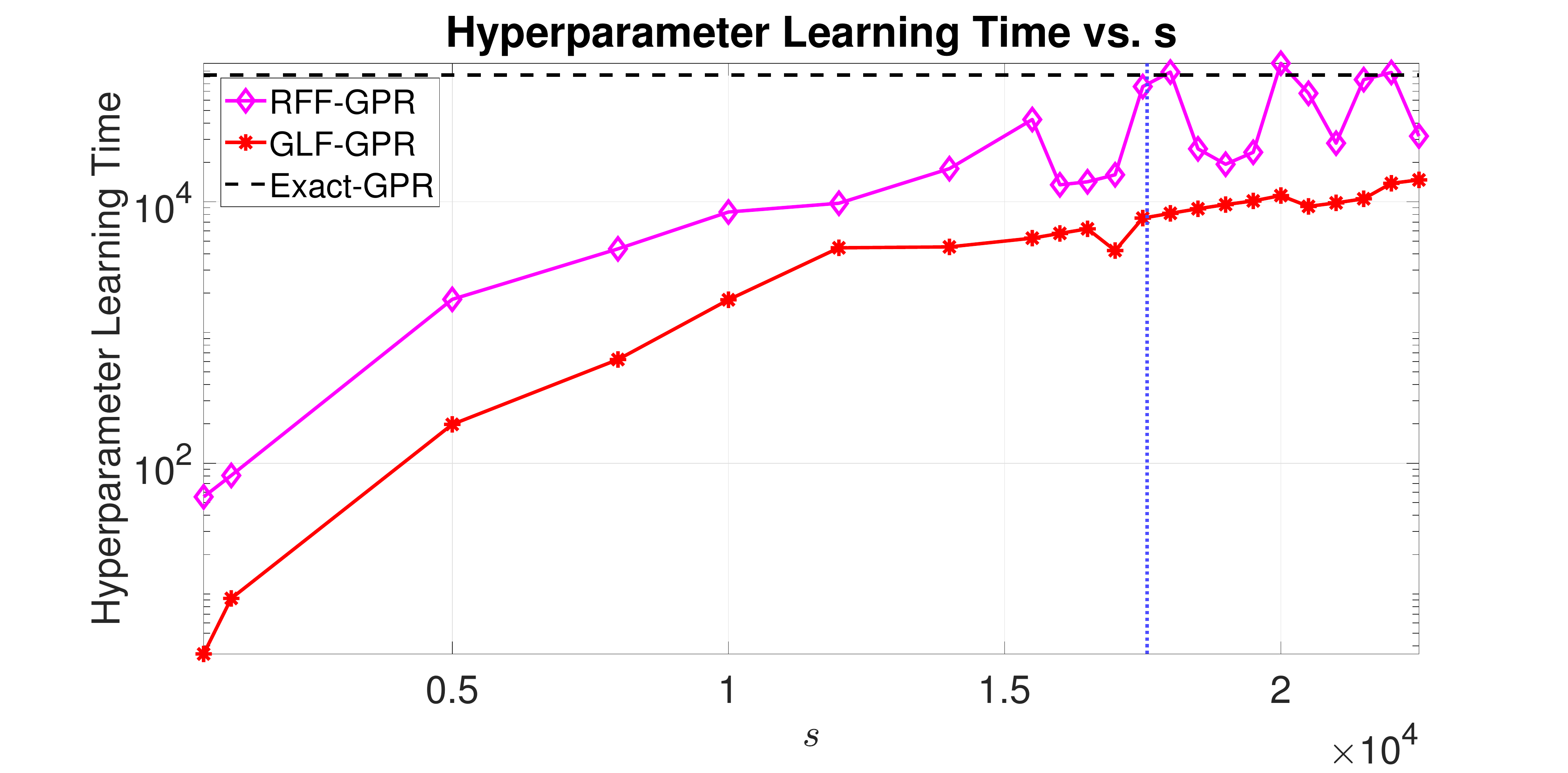}\tabularnewline
\end{tabular}
\par\end{centering}
\caption{\label{fig:sound-results}Results for the natural sound data. Top-left:
full dataset (training and test). Top-right: test data and predictions,
for the best smallest $s$ we tested. Bottom-left: MSE of the test
data for various sizes of $s$. The dashed vertical line is the value
of the theoretical minimum $s$ needed for spectral equivalence. Bottom-right:
running time for various sizes of $s$.}
\end{figure}

Results are reported in Figure~\ref{fig:sound-results}. In the bottom-left
graph we plot the test error as a function of the number of features.
Initially GLF-GPR produces poor results, but when $s$ is large enough,
the results are similar to the Exact-GPR (see also the top-right plot).
We see that even when the number of quadrature points is smaller than
the theoretical value required for spectral equivalence, GLF-GPR's
error is stabilizes on the Exact-GPR error. In contrast, RFF-GPR's
error oscillates above Exact-GPR's error. In the bottom-right graph
we see that the runtime of the hyperparameters learning phase is significantly
smaller for GLF-GPR.

\subsection{Google Daily High Stock Price}

We consider a time series data of the daily high stock price of Google
spanning 3797 days from 19th August 2004 to 19th September 2019. We
set the data as $x\in\{1,\dots,3797\}$ and $y=\log(Stock_{high})$.
The test is of size of 12\% of the data, i.e., consists of 502 days.
We use the Matèrn kernel with $\nu=5/2$.

We note that the theoretical number of quadrature features $s$ required
for spectral equivalence is bigger than the number of training points.
Possible reasons are: the hyperparameter $\sigma_{n}^{2}$ in these
dataset is very small and that increases our bound in Eq.~(\ref{eq:uni_bound_f}).
This increases $U$ which increases $s$. In addition, the weight
function of the Matèrn kernel has singularity point which leads to
the ellipse parameter to be pretty small. However in practice we see
that the approximation convergences around the value $s=1550$.

\begin{figure}[H]
\begin{centering}
\begin{tabular}{cc}
\includegraphics[width=0.45\textwidth]{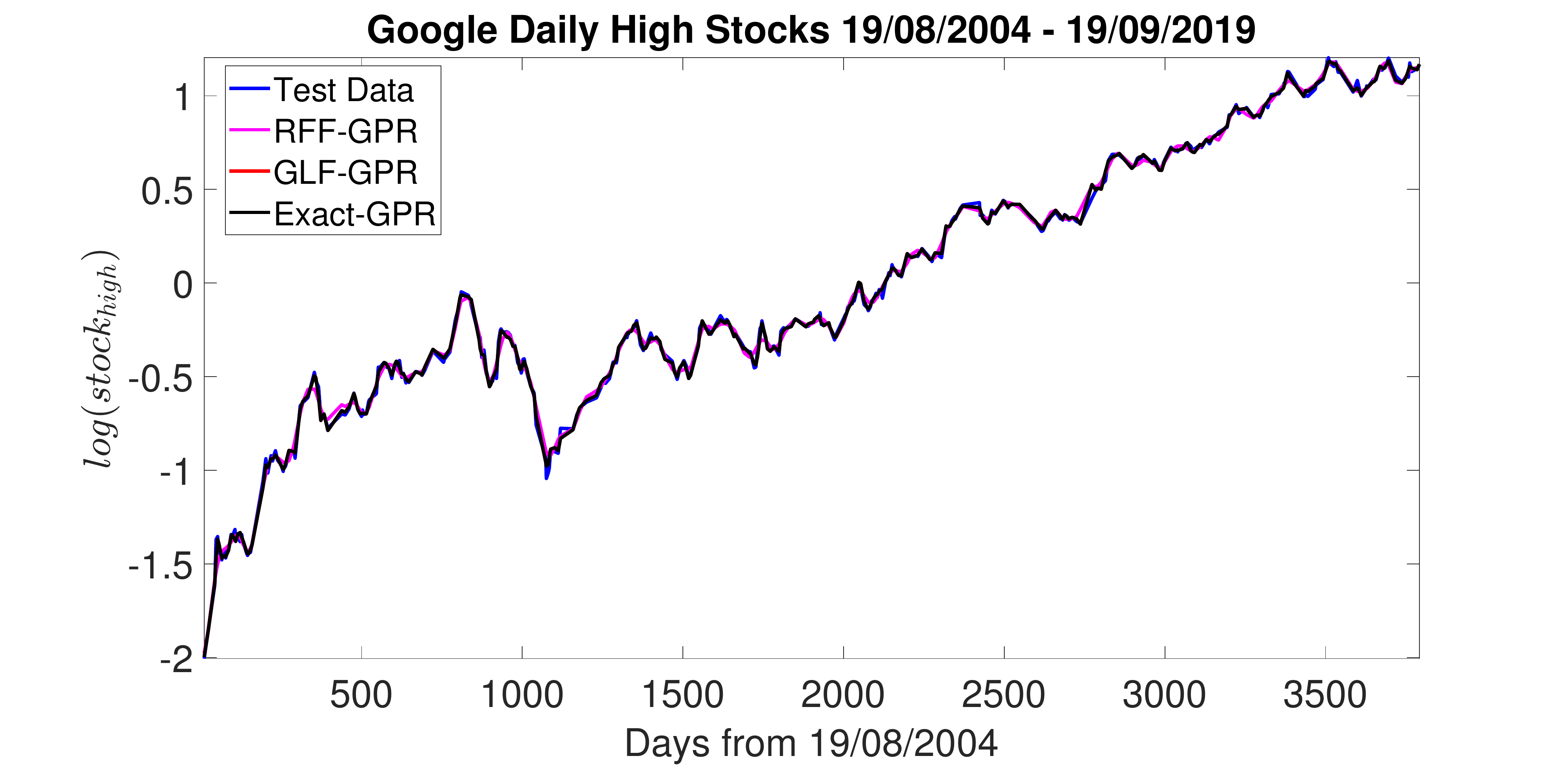} & \includegraphics[width=0.45\textwidth]{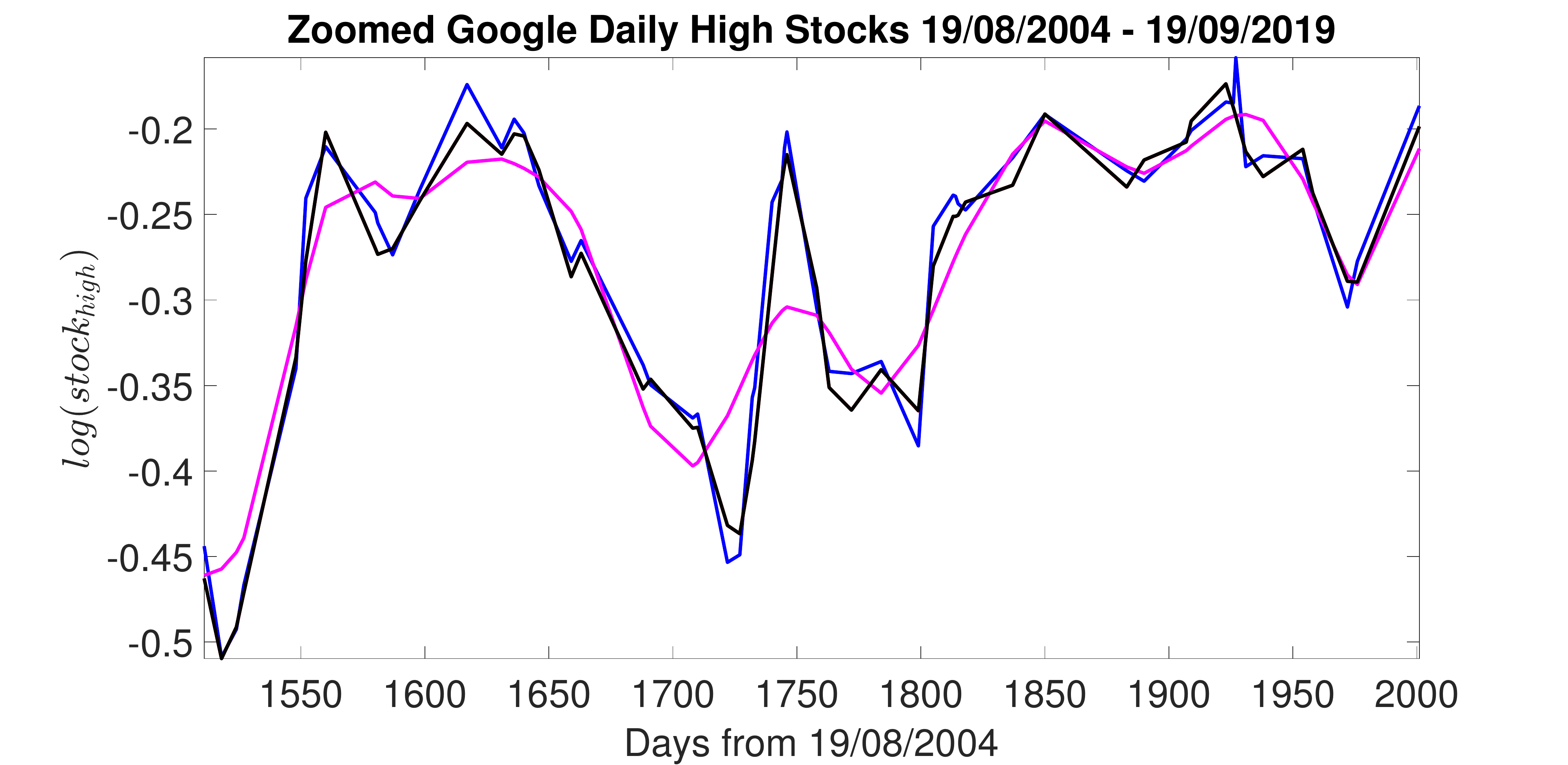}\tabularnewline
\includegraphics[width=0.45\textwidth]{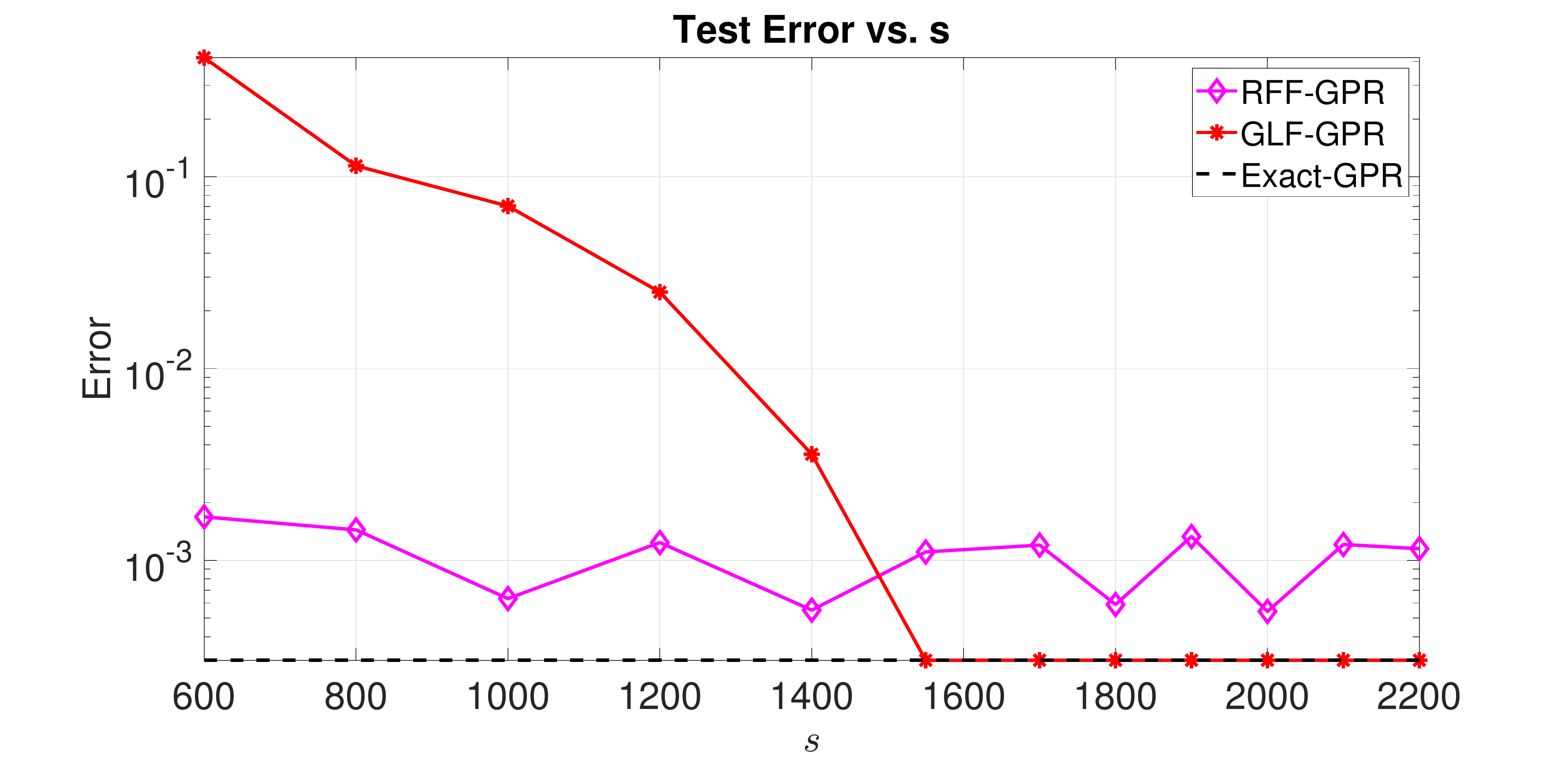} & \includegraphics[width=0.45\textwidth]{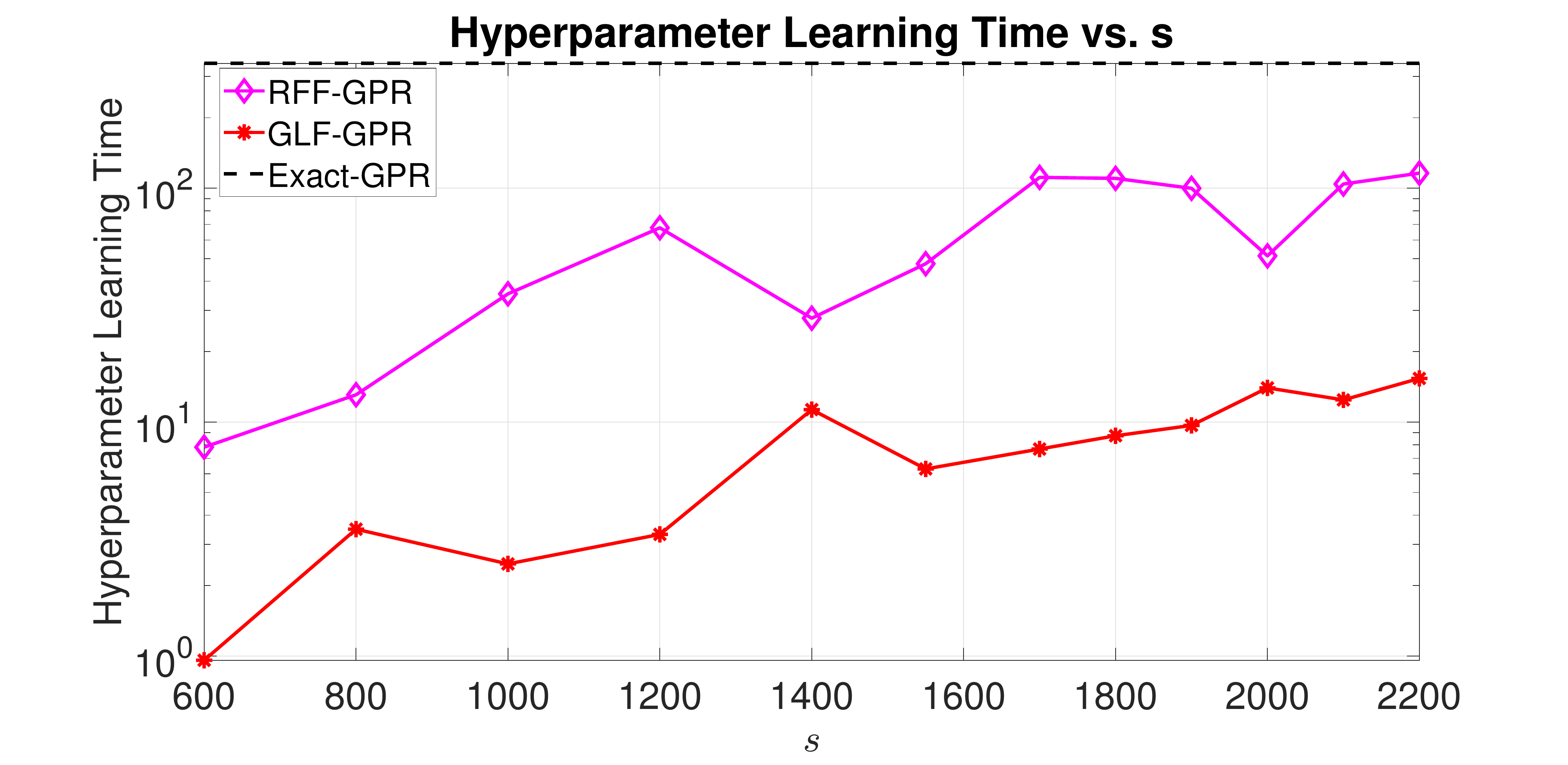}\tabularnewline
\end{tabular}
\par\end{centering}
\caption{\label{fig:google-results} Results for the google stocks data. Top-left:
test data and predictions. Top-right: zoomed test data and predictions
for the best smallest $s$ we tested. Bottom-left: MSE of the test
data for various sizes of $s$. Bottom-right: running time for various
sizes of $s$.}
\end{figure}
Results are reported in Figure~\ref{fig:google-results}. From the
top-right and top-left plots we see that RFF-GPR is producing a poor
approximation while GLF-GPR approximation merges with Exact-GPR approximation.
Also, from the bottom-left plot we see that initially GLF-GPR produces
poor errors, but when $s$ is large enough, the error stabilizes on
the Exact-GPR error (see also the top-right plot).

\subsection{Spatial Temperature Anomaly for East Africa in 2016}

Similar to \cite{ton2018spatial}, we consider MOD11A2 Land Surface
Temperature (LST) 8-day composite 2D data of synoptic yearly mean
for 2016 in the East Africa region. For the training set, we randomly
sample 77404 LST locations and set $\x\in\{\left(Longitude,\,Latitude\right)\}$
and $y=\{temperature\}$.We examine the MSE errors on the remaining
6005 locations, but use all 83409 data points to draw maps. We also
use the anisotropic Matèrn kernel with $\nu=1$. Again, theoretical
number of quadrature features for spectral equivalence is bigger than
then number of training points. However again in practice we see that
the approximation convergences with less features. Due to memory and
time constraints, we were unable to use Exact-GPR, and RFF-GPR results
are presented up to the computer's memory capacity.

\begin{figure}[H]
\begin{centering}
\begin{tabular}{c}
\includegraphics[width=0.9\textwidth]{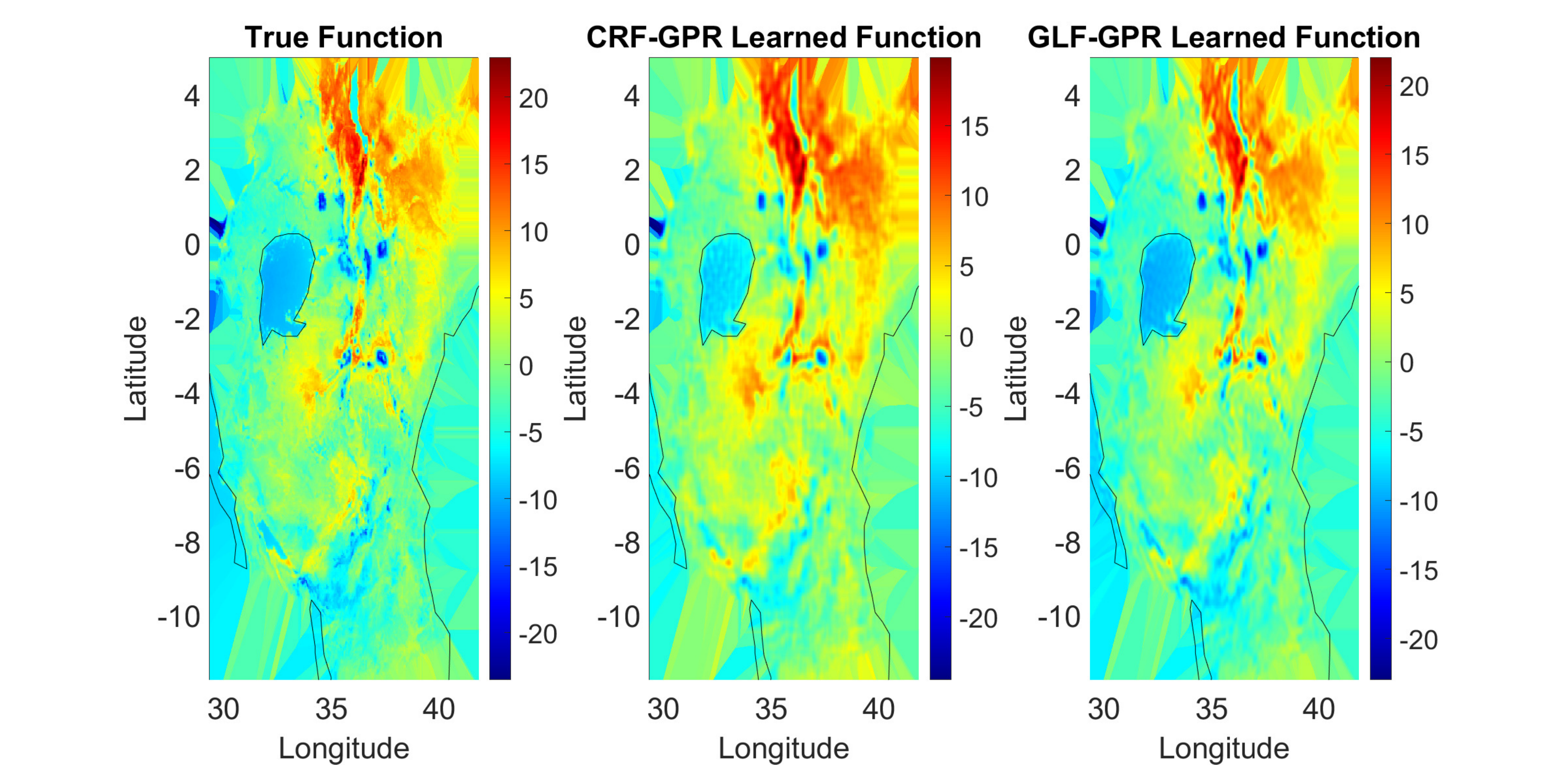}\tabularnewline
\end{tabular}
\par\end{centering}
~
\begin{centering}
\begin{tabular}{cc}
\includegraphics[width=0.45\textwidth]{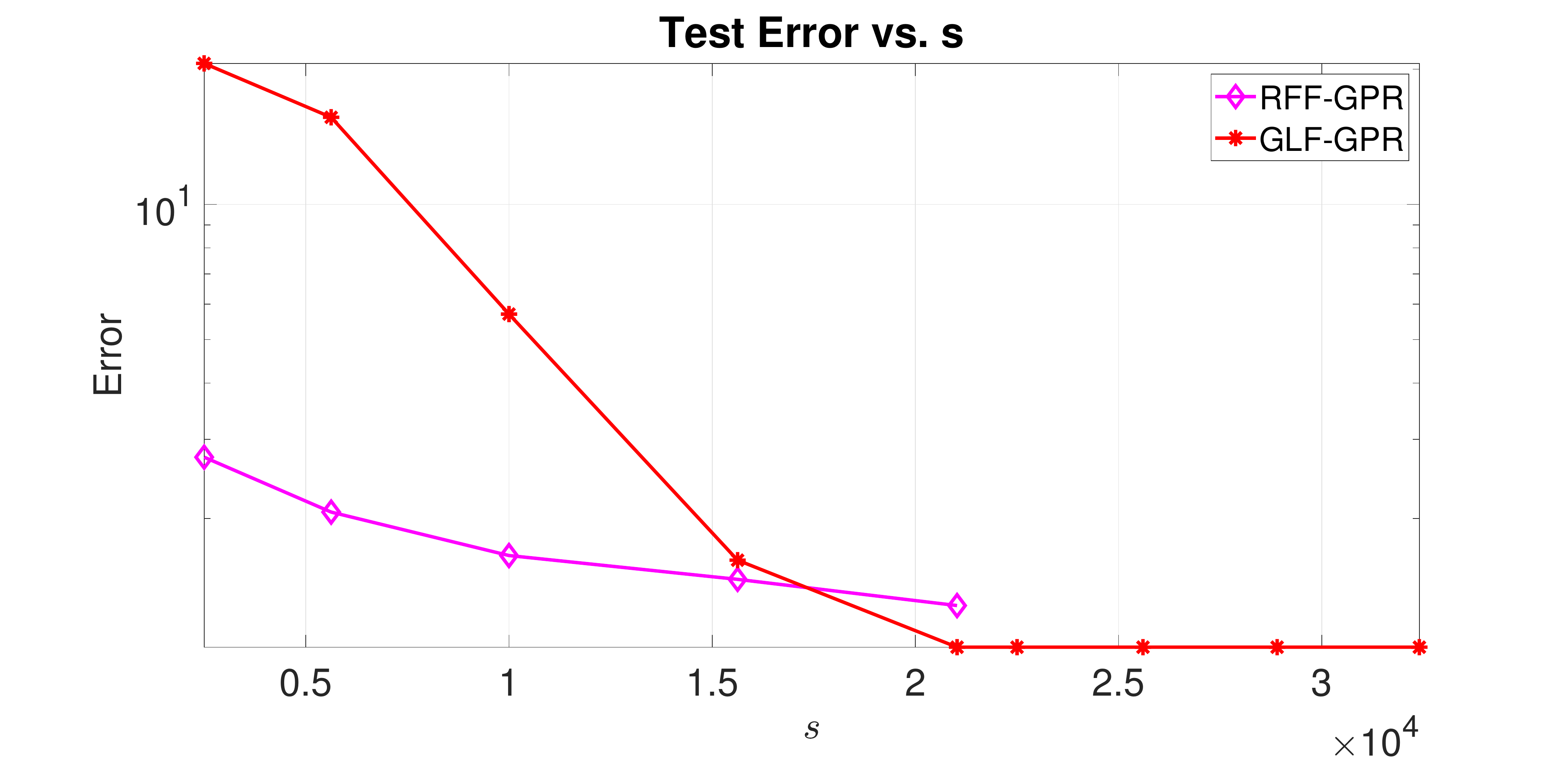} & \includegraphics[width=0.45\textwidth]{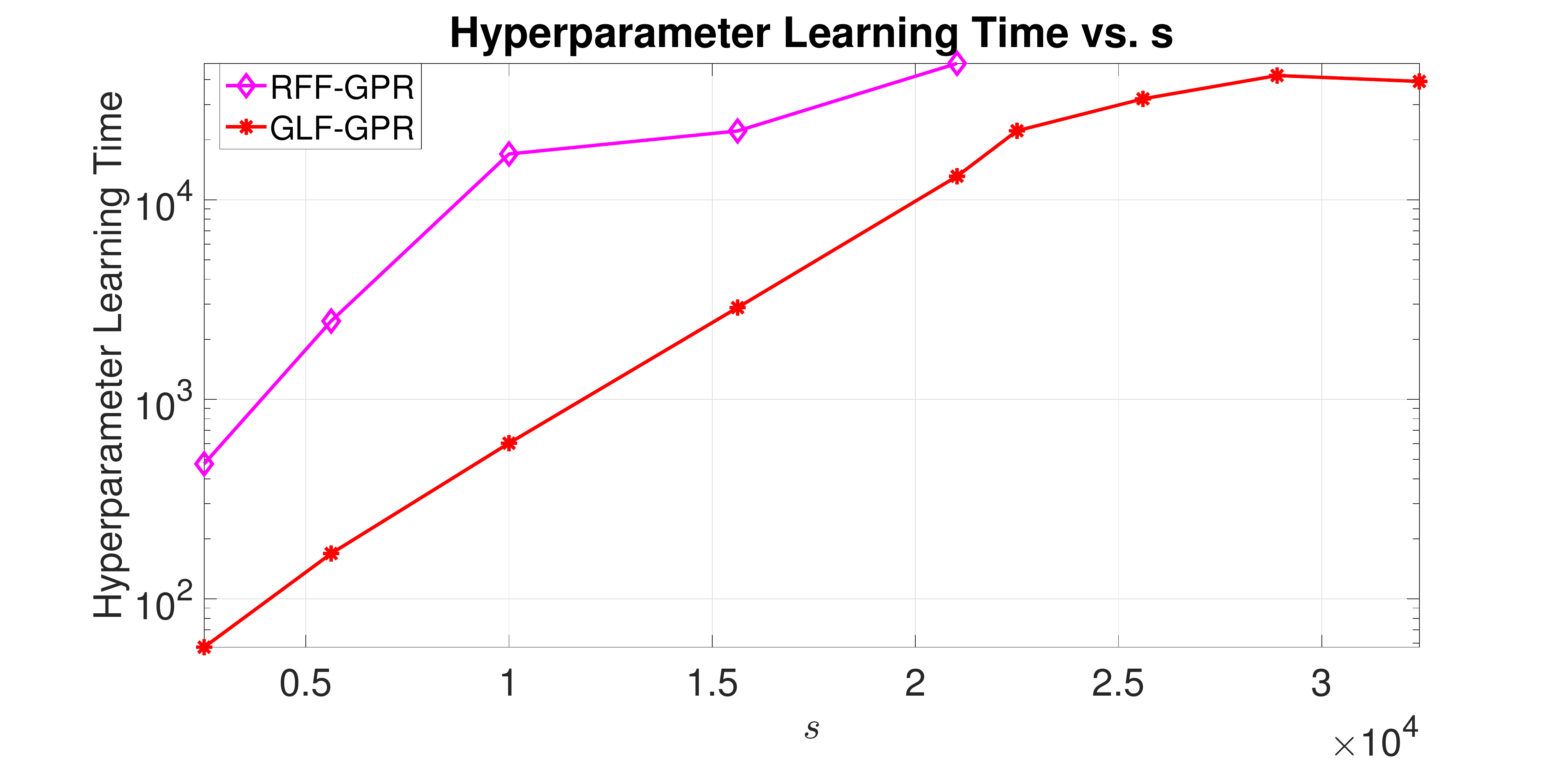}\tabularnewline
\end{tabular}
\par\end{centering}
\caption{\label{fig:Africa-results} Results for the east Africa data. Top:
all true data, and predictions. Bottom-left: MSE of the test data
for various sizes of $s$. Bottom-right: running time for various
sizes of $s$.}
\end{figure}
Results are reported in Figure~\ref{fig:Africa-results}. In the
top plot we see that GLF-GPR approximates the true function well,
unlike RFF-GPR. Also, from the bottom-left plot we see that around
$s=21025$, the GLF-GPR error stabilizes while RFF-GPR error is still
suboptimal.

\section{Conclusions and Future Work}

In this paper, we proposed the use of Gauss-Legendre feature for large-scale
Gaussian process regression. Our method is very much inspired by Random
Fourier Features~\cite{rahimi2008random}. However, our method replaces
Monte-Carlo integration in RFF with a Gauss-Legendre quadrature of
a truncated integral representation of the kernel function. With Gauss-Legendre
quadrature our method is able to build spectrally equivalent kernel
approximation with an amount of features which is asymptotically poly-logarithmic
in the training size. In contrast, with RFF the number of features
for spectral equivalence must be at least linear. Sublinear amount
of features can also be obtained using a Gaussian quadrature (suggested
in the context of kernel learning in~\cite{dao2017gaussian}). However,
this is problematic in the context of hyperparameter learning (see
Section~\ref{subsec:rff-problematic}). RFF has a similar issue.
In contrast, the use of Gauss-Legendre quadrature allows our method
to keep the quadrature nodes and weights fixed, leading to simplified
structural dependence of the kernel matrix on the hyperparameters
which is more amenable to hyperparameter learning. Finally, we demonstrate
the utility of our method on several real-world low-dimensional datasets.

We mention a few possible directions for future research:
\begin{itemize}
\item Asymptotically, our method requires a number of features that is poly-logarithmic
in the training size. Yet, for some moderately sized datasets our
theoretical results required a number of features larger than the
number of training points. However, in practice the number of features
required for high quality results was much smaller than the bound.
Closing this gap is an open problem.
\item Our method is able to handle a rich family of kernels (see Eq.~(\ref{eq:kernel-form}))
which includes stationary kernels and some non-stationary kernels
(see Section~\ref{subsec:semigroup}). Extending our method to arbitrary
non stationary kernels is an open problem.
\item The number of features needed by our method is exponential in the
dimension, i.e. we have not escaped from the curse of dimensionality.
A future research direction is to replace the tensorized multivariate
quadrature with sparse grids, and in doing so avoid the exponential
dependence on $d$.
\end{itemize}

\subsection*{Acknowledgements}

This research was supported by BSF grant 2017698.

\bibliographystyle{plain}
\bibliography{gp_bib}

\appendix

\section{\label{sec:Accelerating-Hyperparametes}Further Details on Hyperparameters
Learning}

\subsection{Derivation of Eq.~(\ref{eq:lik_trace})}

Recall that,
\[
\mat F(\vtheta)=\sigma_{f}^{2}(\sigma_{f}^{2}\mat{\mat Z}^{*}\mat{\mat Z}+\sigma_{n}^{2}\mat W(\vtheta)^{-1})^{-1}\mat{\mat Z}^{*}\mat{\mat Z}\,.
\]
First, from the Woodbury formula we obtain that
\begin{eqnarray*}
\Trace{\tilde{\matK}_{\vtheta}(\matX,\matX)^{-1}} & = & \Trace{\left(\sigma_{f}^{2}\mat{\mat Z}\mat W(\vtheta)\mat{\mat Z}^{*}+\sigma_{n}^{2}\mat I_{n}\right)^{-1}}\\
 & = & \sigma_{n}^{-2}\Trace{\mat I_{n}-\sigma_{f}^{2}\mat{\mat Z}\left(\sigma_{f}^{2}\mat{\mat Z}^{*}\mat{\mat Z}+\sigma_{n}^{2}\mat W(\vtheta)^{-1}\right)^{-1}\mat{\mat Z}^{*}}\\
 & = & \sigma_{n}^{-2}n-\sigma_{n}^{-2}\Trace{\sigma_{f}^{2}\mat{\mat Z}\left(\sigma_{f}^{2}\mat{\mat Z}^{*}\mat{\mat Z}+\sigma_{n}^{2}\mat W(\vtheta)^{-1}\right)^{-1}\mat{\mat Z}^{*}}\\
 & = & \sigma_{n}^{-2}n-\sigma_{n}^{-2}\Trace{\sigma_{f}^{2}\left(\sigma_{f}^{2}\mat{\mat Z}^{*}\mat{\mat Z}+\sigma_{n}^{2}\mat W(\vtheta)^{-1}\right)^{-1}\mat{\mat Z}^{*}\mat{\mat Z}}\\
 & = & \sigma_{n}^{-2}\left(n-\Trace{\mat F(\vtheta)}\right)\,.
\end{eqnarray*}
Now, for the first term in Eq.~(\ref{eq:lik_trace}) we have 
\begin{eqnarray*}
\Trace{\tilde{\matK}_{\vtheta}(\matX,\matX)^{-1}\frac{\partial\tilde{\matK}_{\vtheta}(\matX,\matX)}{\partial\sigma_{f}^{2}}} & = & \Trace{\tilde{\matK}_{\vtheta}(\matX,\matX)^{-1}\mat Z\mat W(\vtheta)\mat Z^{*}}\\
 & = & \Trace{\tilde{\matK}_{\vtheta}(\matX,\matX)^{-1}\sigma_{f}^{-2}\left(\sigma_{f}^{2}\mat Z\mat W(\vtheta)\mat Z^{*}+\sigma_{n}^{2}\mat I_{n}-\sigma_{n}^{2}\mat I_{n}\right)}\\
 & = & \sigma_{f}^{-2}\Trace{\mat I_{n}-\sigma_{n}^{2}\tilde{\matK}_{\vtheta}(\matX,\matX)^{-1}}\\
 & = & \sigma_{f}^{-2}n-\sigma_{f}^{-2}\sigma_{n}^{2}\Trace{\tilde{\matK}_{\vtheta}(\matX,\matX)^{-1}}\\
 & = & \sigma_{f}^{-2}n-\sigma_{f}^{-2}\sigma_{n}^{2}\left(\sigma_{n}^{-2}n-\sigma_{n}^{-2}\Trace{\mat F(\vtheta)}\right)\\
 & = & \sigma_{f}^{-2}\Trace{\mat F(\vtheta)}\,.
\end{eqnarray*}
For the next term in Eq.~(\ref{eq:lik_trace}) we have
\begin{eqnarray*}
\Trace{\tilde{\matK}_{\vtheta}(\matX,\matX)^{-1}\frac{\partial\tilde{\matK}_{\vtheta}(\matX,\matX)}{\partial\sigma_{n}^{2}}} & = & \Trace{\tilde{\matK}_{\vtheta}(\matX,\matX)^{-1}}\\
 & = & \sigma_{n}^{-2}(n-\Trace{\mat F(\vtheta)})
\end{eqnarray*}
Finally, for the third term in Eq.~(\ref{eq:lik_trace}) we have
\begin{eqnarray}
\Trace{\tilde{\matK}_{\vtheta}(\matX,\matX)^{-1}\frac{\partial\tilde{\matK}_{\vtheta}(\matX,\matX)}{\partial\theta_{i}}} & = & \Trace{\left(\sigma_{f}^{2}\mat Z\mat W(\vtheta)\mat Z^{*}+\sigma_{n}^{2}\mat I_{n}\right)^{-1}\sigma_{f}^{2}\mat{\mat Z}\frac{\partial\matW(\vtheta)}{\partial\theta_{i}}\mat{\mat Z}^{*}}\label{eq:trace3}\\
 & = & \sigma_{n}^{-2}\Trace{\sigma_{f}^{2}\mat{\mat Z}\frac{\partial\matW(\vtheta)}{\partial\theta_{i}}\mat{\mat Z}^{*}-\sigma_{f}^{4}\mat{\mat Z}\left(\sigma_{f}^{2}\mat{\mat Z}^{*}\mat{\mat Z}+\sigma_{n}^{2}\mat W(\vtheta)^{-1}\right)^{-1}\mat{\mat Z}^{*}\mat{\mat Z}\frac{\partial\matW(\vtheta)}{\partial\theta_{i}}\mat{\mat Z}^{*}}\nonumber \\
 & = & \sigma_{n}^{-2}\sigma_{f}^{2}\Trace{\mat{\mat Z}\frac{\partial\matW(\vtheta)}{\partial\theta_{i}}\mat{\mat Z}^{*}-\mat{\mat Z}\mat F(\vtheta)\frac{\partial\matW(\vtheta)}{\partial\theta_{i}}\mat{\mat Z}^{*}}\nonumber \\
 & = & \sigma_{n}^{-2}\sigma_{f}^{2}\Trace{\mat{\mat Z}\frac{\partial\matW(\vtheta)}{\partial\theta_{i}}\mat{\mat Z}^{*}}-\sigma_{n}^{-2}\sigma_{f}^{2}\Trace{\mat{\mat Z}\mat F(\vtheta)\frac{\partial\matW(\vtheta)}{\partial\theta_{i}}\mat{\mat Z}^{*}}\nonumber \\
 & = & \sigma_{n}^{-2}\sigma_{f}^{2}\Trace{\frac{\partial\matW(\vtheta)}{\partial\theta_{i}}\mat{\mat Z}^{*}\mat{\mat Z}}-\sigma_{n}^{-2}\sigma_{f}^{2}\Trace{\frac{\partial\matW(\vtheta)}{\partial\theta_{i}}\mat{\mat Z}^{*}\mat{\mat Z}\mat F(\vtheta)}\nonumber 
\end{eqnarray}
where in the second equality we use the Woodbury formula. 

\subsection{Efficient Gaussian Process Regression using QR Decomposition}

Here we present alternative formulas to the ones presented in Section~\ref{subsec:efficient-gpr}
and based on QR decomposition instead of the normal equations. Such
formulas are likely to be more numerically robust.

Let
\[
\mat Z=\mat Q_{\mat Z}\mat R_{\mat Z}
\]
be a thin QR decomposition of $\matZ,$ i.e. $\mat Q_{\mat Z}\in\mathbb{C}^{n\times s}$
is such that $\mat Q_{\mat Z}^{*}\mat Q_{\mat Z}=\mat I_{n}$ and
$\mat R_{\mat Z}\in\mathbb{C}^{s\times s}$ is an upper triangular
matrix. We suggest to compute the QR decomposition of $\matZ$ in
lieu of computing $\matZ^{*}\matZ$, and keeping only $\matR_{\matZ}$
and $\matQ_{\matZ}^{\conj}\y$ so still only $O(s^{2})$ is needed.
There is no asymptotic penalty in terms of arithmetic operation count
since the decomposition can be computed in $O(ns^{2})$ operations.
However, $\matZ^{*}\matZ$ tends to be ill-conditioned due to the
squaring of the condition number of $\matZ$, so it is best to avoid
computing it. Note that the QR decomposition is computed only once,
and not per iteration.

Let us consider a specific iteration, and for conciseness we omit
$\vtheta$ for the following formulas. Let
\[
\mat A\coloneqq\left[\begin{array}{c}
\mat R_{\mat Z}\\
\frac{\sigma_{n}}{\sigma_{f}}\mat W^{-1/2}
\end{array}\right]\in\mathbb{C}^{2s\times s}
\]
We compute a thin QR decomposition $\mat A=\mat Q_{\mat A}\mat R_{\mat A}$
of $\matA$ ($O(s^{3})$ operations) and write 
\[
\mat Q_{\mat A}=\left[\begin{array}{c}
\mat Q_{1}\\
\mat Q_{2}
\end{array}\right]
\]
where $\mat Q_{1},\mat Q_{2}\in\mathbb{C}^{s\times s}$, i.e.,
\[
\mat A=\left[\begin{array}{c}
\mat Q_{1}\\
\mat Q_{2}
\end{array}\right]\mat R_{\mat A}\,\text{.}
\]
Hence,
\[
\mat R_{\mat Z}=\mat Q_{1}\mat R_{\mat A}
\]
which implies that 
\[
\matZ=\mat Q_{\mat Z}\mat Q_{1}\mat R_{\mat A}\,.
\]
Therefore,
\[
\left[\begin{array}{c}
\mat Z\\
\frac{\sigma_{n}}{\sigma_{f}}\mat W^{-1/2}
\end{array}\right]=\left[\begin{array}{cc}
\mat Q_{\mat Z} & \mat 0\\
\mat 0 & \matI_{s}
\end{array}\right]\mat A=\left[\begin{array}{cc}
\mat Q_{\mat Z} & \mat 0\\
\mat 0 & \matI_{s}
\end{array}\right]\left[\begin{array}{c}
\matQ_{1}\\
\matQ_{2}
\end{array}\right]\matR_{\matA}=\left[\begin{array}{c}
\mat Q_{\mat Z}\matQ_{1}\\
\matQ_{2}
\end{array}\right]\matR_{\matA}
\]
which is a QR decomposition of 
\[
\matB\coloneqq\left[\begin{array}{c}
\mat Z\\
\frac{\sigma_{n}}{\sigma_{f}}\mat W^{-1/2}
\end{array}\right]
\]
 The crux is that given the QR decomposition of $\matZ$, we can compute
the QR decomposition of $\matB$ in $O(s^{3}$) arithmetic operations
instead of $O(ns^{2})$.

We now compute
\begin{eqnarray}
\w & = & \mat W^{-1}\left(\sigma_{f}^{2}\mat{\mat Z}^{*}\mat{\mat Z}+\sigma_{n}^{2}\mat W^{-1}\right)^{-1}\mat Z^{*}\y\nonumber \\
 & = & \sigma_{f}^{-2}\mat W^{-1}\left[\begin{array}{c}
\mat Z\\
\frac{\sigma_{n}}{\sigma_{f}}\mat W^{-1/2}
\end{array}\right]^{+}\left[\begin{array}{c}
\y\\
\mat 0_{s\times1}
\end{array}\right]\label{eq:w_QR}\\
 & = & \sigma_{f}^{-2}\mat W^{-1}\matR_{\matA}^{-1}\left[\matQ_{1}^{*}\matQ_{\matZ}^{*}\,\,\matQ_{2}^{*}\right]\left[\begin{array}{c}
\y\\
\mat 0_{s\times1}
\end{array}\right]\nonumber \\
 & = & \sigma_{f}^{-2}\mat W^{-1}\matR_{\matA}^{-1}\matQ_{1}^{*}\matQ_{\matZ}^{*}\y\nonumber 
\end{eqnarray}
which implies that $\w$ can be computed in $O(ns)$ operations (since
$\matW$ is diagonal).

Similarly, we also have
\begin{eqnarray*}
\mat F(\vtheta) & = & \sigma_{f}^{2}\left(\sigma_{f}^{2}\mat{\mat Z}^{*}\mat{\mat Z}+\sigma_{n}^{2}\mat W(\vtheta)^{-1}\right)^{-1}\mat{\mat Z}^{*}\mat{\mat Z}\\
 & = & \sigma_{f}^{2}\sigma_{f}^{-2}\mat R_{\mat A}^{-1}\mat Q_{1}^{*}\mat Q_{\mat Z}^{*}\mat{\mat Z}\\
 & = & \mat R_{\mat A}^{-1}\mat Q_{1}^{*}\mat Q_{\mat Z}^{*}\mat Q_{\mat Z}\mat R_{\mat Z}\\
 & = & \mat R_{\mat A}^{-1}\mat Q_{1}^{*}\mat R_{\mat Z}\\
 & = & \mat R_{\mat A}^{-1}\mat Q_{1}^{*}\mat Q_{1}\mat R_{\mat A}
\end{eqnarray*}
i.e., $\mat F(\vtheta)$ can be computed in $O(s^{3})$ operations.
This allows us to compute the first two formulas in Eq.~(\ref{eq:lik_trace})
in $O(s)$ time. As for the third formula, we have (\ref{eq:trace3}).
Since $\partial\matW(\vtheta)/\partial\theta_{i}$ is diagonal, we
need to compute only the diagonal of $\matZ^{*}\matZ$ and $\mat{\mat Z}^{*}\mat{\mat Z}\mat F(\vtheta)$.
The diagonal of $\matZ^{*}\matZ$ is just the square norms of the
columns of $\matZ$, and can be precomputed in $O(ns)$. Furthermore,
in some cases we know analytically the values of this norm. For example,
for shift-invariant kernels we use $\varphi(\x,\veta)=e^{-i\x^{\T}\veta}$
so the squared norms of the columns of $\matZ$ is equal to $n$.
As for $\mat{\mat Z}^{*}\mat{\mat Z}\mat F(\vtheta)$, we have:
\[
\mat{\mat Z}^{*}\mat{\mat Z}\mat F(\vtheta)=\matR_{\matA}^{*}(\matQ_{1}^{*}\matQ_{1})^{2}\matR_{\matA}=\matR_{\matZ}^{*}\matQ_{1}\matQ_{1}^{*}\matR_{\matZ}\,.
\]
Note that $\matQ_{1}^{*}\matR_{\matZ}$ has already been computed
for $\mat F(\vtheta)$. Since we only need the diagonal of $\mat{\mat Z}^{*}\mat{\mat Z}\mat F(\vtheta)$,
and this is an Hermitian matrix, the diagonal is just the squared
norms of the columns of $\matQ_{1}^{*}\matR_{\matZ}$. Thus, after
$O(s^{3})$ preprocessing, for every $\theta_{i}$ the first term
in Eq. (\ref{eq:diff_L}) can be computed in $O(s)$ operations .

As for the first identity in Eq.~(\ref{eq:diff_L}), we still have
\[
\valpha=\sigma_{n}^{-2}\left(\y-\sigma_{f}^{2}\mat Z\mat W\w\right)
\]
and $\mat Z^{*}\valpha=\w$. So, the second identity can be computed
as previously described using $\w$, which is computed according to
Eq.~(\ref{eq:w_QR}).

\section{\label{sec:multi_cheb} Analysis of Multivariate Chebyshev Approximation
in High Dimensions}

The following theorems are generalizations of similar one-dimensional
theorems. All the proofs rely on ideas similar to the ones presented
in \cite{mason1980near,mason1982minimal,mason2002chebyshev} and \cite[Theorem 3.1, Theorem 8.1]{ftrefethen2013approximation}.
Note that a similar generalization can found in \cite{trefethen2017multivariate,wang2020analysis}.
For completeness, we present our own proof, which is based on a different
technique.

For convenience, denote:

\[
E_{\mat{\rho}}\coloneqq\left\{ \z\in\mathbb{C}^{d}:\,\left|z_{k}+\sqrt{z_{k}^{2}-1}\right|<\rho_{k}\quad\quad\forall k=1,\ldots,d\right\} 
\]
as the polyellipse, where each $E_{\rho_{k}}$ is Bernstein ellipse
with foci at $\pm1$ and the sum of of major and minor semiaxis lengths
of the ellipse is $\rho_{k}$ . Also denote
\[
{\cal C}_{\mathbf{r}}\coloneqq\left\{ \z\in\mathbb{C}^{d}:\,\left|z_{k}\right|=r_{k}\quad\quad\forall k=1,\ldots,d\right\} 
\]
as the polycircle centered at the origin, and simply denote ${\cal C}_{1}$
in the case where $r_{1}=\ldots=r_{d}=1$. Finally, denote 
\[
{\cal A}_{\mathbf{r,R}}\coloneqq\left\{ \z\in\mathbb{C}^{d}:\,0<r_{k}<\left|z_{k}\right|<R_{k}\,\forall k=1,\ldots,d\right\} 
\]
as the polyannulus centered at the origin. \textcolor{black}{The following
proposition appears in \cite{scheidemann2005introduction} as Theorem
1.5.26. See also \cite[Pages 32, 90-91]{bochnerseveral} for further
details.} 
\begin{prop}
\label{prop:multi_Laurent}\textcolor{black}{Let ${\cal A}_{\mathbf{r,R}}$
be the polyannulus centered at the origin, and let $f$ be an analytic
complex function on ${\cal A}_{\mathbf{r,R}}$. Also, let $r_{k}<s_{k}<R_{k},\,k=1,\dots,d.$
Then $f$ has a multivariate Laurent expansion 
\[
f(\z)=\sum_{j_{1},\ldots,j_{d}=-\infty}^{\infty}b_{j_{1}\ldots j_{d}}z_{1}^{j_{1}}\cdots z_{d}^{j_{d}}
\]
converging uniformly on ${\cal A}_{\mathbf{r,R}}$. The coefficients
$b_{j_{1}\ldots j_{d}}$ are given by 
\[
b_{j_{1}\ldots j_{d}}=\frac{1}{\left(2i\pi\right)^{d}}\oint_{{\cal C}_{\s}}\frac{f\left(\z\right)}{\z^{\alpha+1}}d\z
\]
}
\end{prop}

Note that the one dimensional Chebyshev polynomials in the complex
plane are defined by
\[
T_{j}(z)=\frac{w^{j}+w^{-j}}{2}
\]
where 
\[
z=\frac{w+w^{-1}}{2}\,.
\]
In addition, for a function $f$ that is analytic in the interior
and on the boundary of $E_{\rho}$ in the complex plane, the complex
Chebyshev series of $f$ is $\sum_{j=0}^{\infty}a_{j}T_{j}(z)$ where
\[
\forall j\neq0:\,a_{j}=\frac{2}{\pi\left(\rho^{2j}+\rho^{-2j}\right)}\oint_{E_{\rho}}f(z)\overline{T_{j}(z)}\left|\frac{dz}{\sqrt{1-z^{2}}}\right|,\quad a_{0}=\frac{1}{2\pi}\oint_{E_{\rho}}f(z)\overline{T_{j}(z)}\left|\frac{dz}{\sqrt{1-z^{2}}}\right|\,.
\]
The last definition was introduced in \cite{mason2002chebyshev},
and we generalize it to multivariate functions. The multivariate complex
tensorized Chebyshev series of a multivariate complex function $f$
that is analytic in the polyellipse $E_{\vrho}$ can be defined by
\[
\sum_{j_{1},\ldots,j_{d}=0}^{\infty}a_{j_{1}\ldots j_{d}}T_{j_{1}}(z_{1})\cdots T_{j_{d}}(z_{d})\,.
\]

\begin{prop}
\label{prop:complex_coef}Let $f$ be a multivariate complex function
that is analytic in the polyellipse $E_{\mat{\rho}}$, where $\vrho=(\rho_{1},\dots,\rho_{d}),\,\rho_{1},\dots,\rho_{d}>0$.
Then, the coefficients of its multivariate complex tensorized Chebyshev
series
\begin{equation}
\sum_{j_{1},\ldots,j_{d}=0}^{\infty}a_{j_{1}\ldots j_{d}}T_{j_{1}}(z_{1})\cdots T_{j_{d}}(z_{d})\label{eq:multi_Cheb}
\end{equation}
are given by
\[
a_{j_{1}\ldots j_{d}}=\frac{2^{d-m}}{\pi^{d}\left(\rho_{1}^{2j_{1}}+\rho_{1}^{-2j_{1}}\right)\cdots\left(\rho_{d}^{2j_{d}}+\rho_{d}^{-2j_{d}}\right)}\oint_{E_{\mat{\rho}}}f\left(z_{1},\ldots,z_{d}\right)\overline{T_{j_{1}}(z_{1})\cdots T_{j_{d}}(z_{d})}\left|\frac{dz_{1}\cdots dz_{d}}{\sqrt{1-z_{1}^{2}\cdots}\sqrt{1-z_{d}^{2}}}\right|
\]
where $m:=\#\{j_{k}:\,j_{k}=0\}$.
\end{prop}

We remark that this Chebyshev series converges to $f$ uniformly,
as claimed in \cite[Theorem 9.1]{mason1982minimal}.
\begin{proof}
We begin by mapping $f(\z)$ on the contour of $E_{\mat{\rho}}$ into
$g(\w)$ on ${\cal C}_{\mat{\rho}}$. For $k=1,\dots,d$, define

\[
z_{k}=\frac{w_{k}+w_{k}^{-1}}{2}
\]
 such that $g\left(w_{1},\ldots,w_{d}\right)=f\left(z_{1},\ldots,z_{d}\right)=f\left(\frac{w_{1}+w_{1}^{-1}}{2},\ldots,\frac{w_{d}+w_{d}^{-1}}{2}\right)$.
It follows that 
\begin{equation}
g\left(w_{1},\ldots,w_{d}\right)=g\left(w_{1}^{\alpha_{1}},\ldots,w_{d}^{\alpha_{d}}\right),\quad\alpha_{1},\ldots,\alpha_{d}\in\left\{ -1,1\right\} \,.\label{eq:g_sym}
\end{equation}
The equation for each $w_{k}$ has two solutions
\[
w_{k}=z_{k}\pm\sqrt{z_{k}^{2}-1}\,.
\]
We choose the solutions $w_{k}=z_{k}+\sqrt{z_{k}^{2}-1}\,$, so $\left|w_{k}\right|=\rho_{k}>1$
and thus the second solution for each $k=1,\ldots,d$ is essentially
$w_{k}^{-1}$. These relations imply that $g$ is analytic in the
polyannulus between ${\cal C}_{\mat{\rho}}$ and ${\cal C}_{\mat{\rho}^{-1}}$.
We also have for each $k=1,\ldots,d$ 
\[
T_{j_{k}}(z_{k})=\frac{w_{k}^{j_{k}}+w_{k}^{-j_{k}}}{2}
\]
Therefore, and since $f$ is analytic in $E_{\vrho}$, we have
\begin{align*}
g(\w) & =f(\z)\\
 & =\sum_{j_{1},\ldots,j_{d}=0}^{\infty}a_{j_{1}\ldots j_{d}}T_{j_{1}}(z_{1})\cdots T_{j_{d}}(z_{d})\\
 & =\sum_{j_{1},\ldots,j_{d}=0}^{\infty}\frac{a_{j_{1}\ldots j_{d}}}{2^{d}}\left(w_{k}^{j_{1}}+w_{k}^{-j_{1}}\right)\cdots\left(w_{d}^{j_{d}}+w_{d}^{-j_{d}}\right)
\end{align*}
That is, the series given in Eq.~(\ref{eq:multi_Cheb}) can be written
as the Laurent series of $g$. Thus, by Proposition \ref{prop:multi_Laurent},
the coefficients are given by
\[
\frac{a_{j_{1}\ldots j_{d}}}{2^{d-m}}=\frac{1}{\left(2i\pi\right)^{d}}\oint_{{\cal C}_{\mat{\rho}}}w_{1}^{-1-j_{1}}\cdots w_{d}^{-1-j_{d}}g\left(w_{1},\ldots,w_{d}\right)dw_{1}\cdots dw_{d}\,.
\]
which implies
\begin{equation}
a_{j_{1}\ldots j_{d}}=\frac{1}{2^{m}\left(i\pi\right)^{d}}\oint_{{\cal C}_{\mat{\rho}}}w_{1}^{-1-j_{1}}\cdots w_{d}^{-1-j_{d}}g\left(w_{1},\ldots,w_{d}\right)dw_{1}\cdots dw_{d}\,.\label{eq:coeff_w}
\end{equation}
The last integral can be written also as 
\begin{eqnarray}
\oint_{{\cal C}_{\mat{\rho}}}w_{1}^{-1-j_{1}}\cdots w_{d}^{-1-j_{d}}g\left(w_{1},\ldots,w_{d}\right)dw_{1}\cdots dw_{d} & =\nonumber \\
\oint_{{\cal C}_{\mat{\rho}}}w_{1}^{-1-j_{1}}\cdots w_{d}^{-1-j_{d}}g\left(w_{1}^{\alpha_{1}},\ldots,w_{d}^{\alpha_{d}}\right)dw_{1}\cdots dw_{d} & =\label{eq:g_powers}\\
\oint_{\left|\tilde{w}_{1}\right|=\rho_{1}^{\alpha_{1}}}\ldots\oint_{\left|\tilde{w}_{d}\right|=\rho_{d}^{\alpha_{d}}}\widetilde{w}_{1}^{-1+\alpha_{1}j_{1}}\cdots\widetilde{w}_{d}^{-1+\alpha_{d}j_{d}}g\left(\widetilde{w}_{1},\ldots,\widetilde{w}_{d}\right)d\widetilde{w}_{1}\cdots d\widetilde{w}_{d} & =\nonumber \\
\oint_{{\cal C}_{\mat{\rho}}}\widetilde{w}_{1}^{-1+\alpha_{1}j_{1}}\cdots\widetilde{w}_{d}^{-1+\alpha_{d}j_{d}}g\left(\widetilde{w}_{1},\ldots,\widetilde{w}_{d}\right)d\widetilde{w}_{1}\cdots d\widetilde{w}_{d}\nonumber 
\end{eqnarray}
where the first equality follows from Eq.~(\ref{eq:g_sym}), the
second equality is changing of variables from $w_{k}$ to $\widetilde{w}_{k}=w_{k}^{\alpha_{k}},\,\alpha_{k}\in\left\{ -1,1\right\} $,
and the last equality is due to Definition \ref{prop:multi_Laurent}
which means the integral also can be considered on ${\cal C}_{\mat{\rho}^{-1}}$.
Now we show that 
\begin{eqnarray}
\oint_{{\cal C}_{\mat{\rho}}}w_{1}^{-1-j_{1}}\cdots w_{d}^{-1-j_{d}}g\left(w_{1},\ldots,w_{d}\right)dw_{1}\cdots dw_{d} & =\label{eq:g_power_ellipse}\\
\prod_{k=1}^{d}\frac{1}{\rho_{k}^{2j_{k}}+\rho_{k}^{-2j_{k}}}\oint_{{\cal C}_{\mat{\rho}}}\prod_{k=1}^{d}\left(\rho_{k}^{2j_{k}}w_{k}^{-j_{k}}+\rho_{k}^{-2j_{k}}w_{k}^{j_{k}}\right)g\left(w_{1},\ldots,w_{d}\right)\frac{dw_{1}\cdots dw_{d}}{w_{1}\cdots w_{d}}\nonumber 
\end{eqnarray}
by induction on the number of changes of variables.

\textbf{The base case: }apply the change of variables only for one
of the variables. Without loss of generality, we show it for $w_{1}$:\textbf{
\begin{eqnarray*}
\oint_{{\cal C}_{\mat{\rho}}}w_{1}^{-1-j_{1}}\cdots w_{d}^{-1-j_{d}}g\left(w_{1},\ldots,w_{d}\right)dw_{1}\cdots dw_{d} & =\\
\frac{\rho_{1}^{2j_{1}}+\rho_{1}^{-2j_{1}}}{\rho_{1}^{2j_{1}}+\rho_{1}^{-2j_{1}}}\oint_{{\cal C}_{\mat{\rho}}}w_{1}^{-j_{1}}\cdots w_{d}^{-j_{d}}g\left(w_{1},\ldots,w_{d}\right)\frac{dw_{1}\cdots dw_{d}}{w_{1}\cdots w_{d}} & =\\
\frac{1}{\rho_{1}^{2j_{1}}+\rho_{1}^{-2j_{1}}}\oint_{{\cal C}_{\mat{\rho}}}\rho_{1}^{2j_{1}}w_{1}^{-j_{1}}\cdots w_{d}^{-j_{d}}g\left(w_{1},\ldots,w_{d}\right)\frac{dw_{1}\cdots dw_{d}}{w_{1}\cdots w_{d}} & +\\
+\frac{1}{\rho_{1}^{2j_{1}}+\rho_{1}^{-2j_{1}}}\oint_{{\cal C}_{\mat{\rho}}}\rho_{1}^{-2j_{1}}w_{1}^{-j_{1}}\cdots w_{d}^{-j_{d}}g\left(w_{1},\ldots,w_{d}\right)\frac{dw_{1}\cdots dw_{d}}{w_{1}\cdots w_{d}} & =\\
\frac{1}{\rho_{1}^{2j_{1}}+\rho_{1}^{-2j_{1}}}\oint_{{\cal C}_{\mat{\rho}}}\rho_{1}^{2j_{1}}w_{1}^{-j_{1}}\cdots w_{d}^{-j_{d}}g\left(w_{1},\ldots,w_{d}\right)\frac{dw_{1}\cdots dw_{d}}{w_{1}\cdots w_{d}} & +\\
+\frac{1}{\rho_{1}^{2j_{1}}+\rho_{1}^{-2j_{1}}}\oint_{{\cal C}_{\mat{\rho}}}\rho_{1}^{-2j_{1}}w_{1}^{j_{1}}w_{2}^{-j_{2}}\cdots w_{d}^{-j_{d}}g\left(w_{1},\ldots,w_{d}\right)\frac{dw_{1}\cdots dw_{d}}{w_{1}\cdots w_{d}} & =\\
\frac{1}{\rho_{1}^{2j_{1}}+\rho_{1}^{-2j_{1}}}\oint_{{\cal C}_{\mat{\rho}}}\left(\rho_{1}^{2j_{1}}w_{1}^{-j_{1}}+\rho_{1}^{-2j_{1}}w_{1}^{j_{1}}\right)w_{2}^{-j_{2}}\cdots w_{d}^{-j_{d}}g\left(w_{1},\ldots,w_{d}\right)\frac{dw_{1}\cdots dw_{d}}{w_{1}\cdots w_{d}}
\end{eqnarray*}
}where in the third equality we use Eq.~(\ref{eq:g_powers}) with
$\alpha_{1}=1$.

\textbf{The inductive step: }suppose that for \textbf{$1<n-1<d$}
changes of variables, the following holds:
\begin{eqnarray}
\oint_{{\cal C}_{\mat{\rho}}}w_{1}^{-1-j_{1}}\cdots w_{d}^{-1-j_{d}}g\left(w_{1},\ldots,w_{d}\right)dw_{1}\cdots dw_{d} & =\label{eq:ind_assum}\\
\prod_{k=1}^{n-1}\frac{1}{\rho_{k}^{2j_{k}}+\rho_{k}^{-2j_{k}}}\oint_{{\cal C}_{\mat{\rho}}}\prod_{k=1}^{n-1}\left(\rho_{k}^{2j_{k}}w_{k}^{-j_{k}}+\rho_{k}^{-2j_{k}}w_{k}^{j_{k}}\right)w_{n}^{-j_{n}}\cdots w_{d}^{-j_{d}}g\left(w_{1},\ldots,w_{d}\right)\frac{dw_{1}\cdots dw_{d}}{w_{1}\cdots w_{d}}\nonumber 
\end{eqnarray}
Then, we show that this is also true for $n$ changes of variables:\textbf{
\begin{eqnarray*}
\oint_{{\cal C}_{\mat{\rho}}}w_{1}^{-1-j_{1}}\cdots w_{d}^{-1-j_{d}}g\left(w_{1},\ldots,w_{d}\right)dw_{1}\cdots dw_{d} & =\\
\frac{\rho_{n}^{2j_{n}}+\rho_{n}^{-2j_{n}}}{\rho_{n}^{2j_{n}}+\rho_{n}^{-2j_{n}}}\oint_{{\cal C}_{\mat{\rho}}}w_{1}^{-j_{1}}\cdots w_{d}^{-j_{d}}g\left(w_{1},\ldots,w_{d}\right)\frac{dw_{1}\cdots dw_{d}}{w_{1}\cdots w_{d}} & =\\
\prod_{k=1}^{n}\frac{1}{\rho_{k}^{2j_{k}}+\rho_{k}^{-2j_{k}}}\left(\rho_{n}^{2j_{n}}+\rho_{n}^{-2j_{n}}\right)\oint_{{\cal C}_{\mat{\rho}}}\left(\prod_{k=1}^{n-1}\left(\rho_{k}^{2j_{k}}w_{k}^{-j_{k}}+\rho_{k}^{-2j_{k}}w_{k}^{j_{k}}\right)\right)w_{n}^{-j_{n}}\cdots w_{d}^{-j_{d}}g\left(w_{1},\ldots,w_{d}\right)\frac{dw_{1}\cdots dw_{d}}{w_{1}\cdots w_{d}} & =\\
\prod_{k=1}^{n}\frac{1}{\rho_{k}^{2j_{k}}+\rho_{k}^{-2j_{k}}}\oint_{{\cal C}_{\mat{\rho}}}\left(\prod_{k=1}^{n-1}\left(\rho_{k}^{2j_{k}}w_{k}^{-j_{k}}+\rho_{k}^{-2j_{k}}w_{k}^{j_{k}}\right)\right)\rho_{n}^{2j_{n}}w_{n}^{-j_{n}}\cdots w_{d}^{-j_{d}}g\left(w_{1},\ldots,w_{d}\right)\frac{dw_{1}\cdots dw_{d}}{w_{1}\cdots w_{d}} & +\\
+\prod_{k=1}^{n}\frac{1}{\rho_{k}^{2j_{k}}+\rho_{k}^{-2j_{k}}}\oint_{{\cal C}_{\mat{\rho}}}\left(\prod_{k=1}^{n-1}\left(\rho_{k}^{2j_{k}}w_{k}^{-j_{k}}+\rho_{k}^{-2j_{k}}w_{k}^{j_{k}}\right)\right)\rho_{n}^{-2j_{n}}w_{n}^{-j_{n}}\cdots w_{d}^{-j_{d}}g\left(w_{1},\ldots,w_{d}\right)\frac{dw_{1}\cdots dw_{d}}{w_{1}\cdots w_{d}} & =\\
\prod_{k=1}^{n}\frac{1}{\rho_{k}^{2j_{k}}+\rho_{k}^{-2j_{k}}}\oint_{{\cal C}_{\mat{\rho}}}\left(\prod_{k=1}^{n-1}\left(\rho_{k}^{2j_{k}}w_{k}^{-j_{k}}+\rho_{k}^{-2j_{k}}w_{k}^{j_{k}}\right)\right)\rho_{n}^{2j_{n}}w_{n}^{-j_{n}}\cdots w_{d}^{-j_{d}}g\left(w_{1},\ldots,w_{d}\right)\frac{dw_{1}\cdots dw_{d}}{w_{1}\cdots w_{d}} & +\\
+\prod_{k=1}^{n}\frac{1}{\rho_{k}^{2j_{k}}+\rho_{k}^{-2j_{k}}}\oint_{{\cal C}_{\mat{\rho}}}\left(\prod_{k=1}^{n-1}\left(\rho_{k}^{2j_{k}}w_{k}^{-j_{k}}+\rho_{k}^{-2j_{k}}w_{k}^{j_{k}}\right)\right)\rho_{n}^{-2j_{n}}w_{n}^{j_{n}}\cdots w_{d}^{-j_{d}}g\left(w_{1},\ldots,w_{d}\right)\frac{dw_{1}\cdots dw_{d}}{w_{1}\cdots w_{d}} & =\\
\prod_{k=1}^{n}\frac{1}{\rho_{k}^{2j_{k}}+\rho_{k}^{-2j_{k}}}\oint_{{\cal C}_{\mat{\rho}}}\left(\prod_{k=1}^{n}\left(\rho_{k}^{2j_{k}}w_{k}^{-j_{k}}+\rho_{k}^{-2j_{k}}w_{k}^{j_{k}}\right)\right)w_{n+1}^{-j_{n+1}}\cdots w_{d}^{-j_{d}}g\left(w_{1},\ldots,w_{d}\right)\frac{dw_{1}\cdots dw_{d}}{w_{1}\cdots w_{d}}
\end{eqnarray*}
}where in the second equality we use Eq.~(\ref{eq:ind_assum}), and
in the fourth equality we use Eq.~(\ref{eq:g_powers}) with $\alpha_{n}=1$.
Therefore, by induction, for $n=d$ we obtain Eq.~(\ref{eq:g_power_ellipse}),
and by Eq.~(\ref{eq:coeff_w}):
\[
a_{j_{1}\ldots j_{d}}=\frac{1}{2^{m}\left(i\pi\right)^{d}}\prod_{k=1}^{d}\frac{1}{\rho_{k}^{2j_{k}}+\rho_{k}^{-2j_{k}}}\oint_{{\cal C}_{\mat{\rho}}}\prod_{k=1}^{d}\left(\rho_{k}^{2j_{k}}w_{k}^{-j_{k}}+\rho_{k}^{-2j_{k}}w_{k}^{j_{k}}\right)g\left(w_{1},\ldots,w_{d}\right)\frac{dw_{1}\cdots dw_{d}}{w_{1}\cdots w_{d}}\,.
\]

Now, for each $k=1,\ldots,d$ we have
\begin{equation}
\left|\frac{dz_{k}}{\sqrt{1-z_{k}^{2}}}\right|=\frac{dw_{k}}{iw_{k}}\label{eq:dz_to_dx}
\end{equation}
and
\begin{equation}
\overline{T_{j_{k}}(z_{k})}=\frac{\overline{w_{k}}^{j_{k}}+\overline{w_{k}}^{-j_{k}}}{2}=\frac{\rho_{k}^{2j_{k}}w_{k}^{-j_{k}}+\rho_{k}^{-2j_{k}}w_{k}^{j_{k}}}{2}\label{eq:Tj_to_w_power}
\end{equation}
where $w_{k}=\rho_{k}e^{i\theta_{k}}$. Therefore, replacing $g\left(w_{1},\ldots,w_{d}\right)$
by $f\left(z_{1},\ldots,z_{d}\right)$ and recall that each $w_{k}$
on ${\cal C}_{\rho_{k}}$ maps $z_{k}$ on $E_{\rho_{k}}$, we obtain
\[
a_{j_{1}\ldots j_{d}}=\frac{2^{d-m}}{\pi^{d}\left(\rho_{1}^{2j_{1}}+\rho_{1}^{-2j_{1}}\right)\cdots\left(\rho_{d}^{2j_{d}}+\rho_{d}^{-2j_{d}}\right)}\oint_{E_{\mat{\rho}}}f\left(z_{1},\ldots,z_{d}\right)\overline{T_{j_{1}}(z_{1})\cdots T_{j_{d}}(z_{d})}\left|\frac{dz_{1}\cdots dz_{d}}{\sqrt{1-z_{1}^{2}\cdots}\sqrt{1-z_{d}^{2}}}\right|\,.
\]
We proceed to the main theorem:
\end{proof}
\begin{thm}
\label{thm:trefethen_coeff_2D} Let $f(x_{1},\ldots,x_{d})$ be an
analytic function in $\left[-1,1\right]^{d}$ and analytically continuable
to the polyellipse $E_{\mat{\rho}}$ where it satisfies $|f(x_{1},\ldots,x_{d})|\leq M$
for some $M>0$. Let $T_{j}(x):=\cos(j\cos^{-1}(x))$ be the $j$
degree one dimensional Chebyshev polynomial, and $E_{\rho_{1}},\ldots,E_{\rho_{d}}$
are the open Bernstein ellipses with major and minor semiaxis lengths
correspondingly summing to $\rho_{1},\ldots,\rho_{d}>1$. Then:
\begin{enumerate}
\item The multivariate (real) Chebyshev coefficients of $f$ are given by
\[
a_{j_{1}\ldots j_{d}}:=\frac{2^{d-m}}{\pi^{d}}\int_{\left[-1,1\right]^{d}}\frac{f(x_{1},\ldots,x_{d})T_{j_{1}}(x_{1})\cdots T_{j_{d}}(x_{d})}{\sqrt{1-x_{1}^{2}}\cdots\sqrt{1-x_{d}^{2}}}dx_{1}\cdots dx_{d}
\]
where $m:=\#\{j_{k}:\,j_{k}=0\}$.
\item The coefficients satisfy
\[
|a_{j_{1}\ldots j_{d}}|\leq\frac{2^{d-m}M}{\rho_{1}^{j_{1}}\cdots\rho_{d}^{j_{d}}}\,.
\]
\end{enumerate}
\end{thm}

Some versions of this theorem appear in\textcolor{black}{{} \cite[Pages 32, 94-95]{bochnerseveral}
and \cite{trefethen2017multivariate}}, however without an explicit
bound.
\begin{proof}
As in the proof of Proposition \ref{prop:complex_coef}, consider
the analytic continuation $f(\z)$ on the contour of $E_{\mat{\rho}}$
which we map into $g(\w)$ on ${\cal C}_{\mat{\rho}}$, by defining
for $k=1,\dots,d$
\[
z_{k}=\frac{w_{k}+w_{k}^{-1}}{2}\,.
\]
Then, we saw that 
\[
a_{j_{1}\ldots j_{d}}=\frac{2^{d-m}}{\pi^{d}\left(\rho_{1}^{2j_{1}}+\rho_{1}^{-2j_{1}}\right)\cdots\left(\rho_{d}^{2j_{d}}+\rho_{d}^{-2j_{d}}\right)}\oint_{E_{\mat{\rho}}}f\left(z_{1},\ldots,z_{d}\right)\overline{T_{j_{1}}(z_{1})\cdots T_{j_{d}}(z_{d})}\left|\frac{dz_{1}\cdots dz_{d}}{\sqrt{1-z_{1}^{2}\cdots}\sqrt{1-z_{d}^{2}}}\right|
\]
In particular, since $f(z_{1},\ldots,z_{d})$ is a continuation of
$f(x_{1},\ldots,x_{d})$ to the complex plane, replacing each $z_{k}$
with $x_{k}=Re(z_{k})$ for $k=1,\dots d$ gives
\[
a_{j_{1}\ldots j_{d}}=\frac{2^{d-m}}{\pi^{d}}\int_{\left[-1,1\right]^{d}}\frac{f\left(x_{1},\ldots,x_{d}\right)T_{j_{1}}(x_{1})\cdots T_{j_{d}}(x_{d})}{\sqrt{1-x_{1}^{2}}\cdots\sqrt{1-x_{d}^{2}}}dx_{1}\cdots dx_{d}\,.
\]
This completes the first part of the proof. For the second part of
the proof, we use the bound on $f$ representation in Eq.~(\ref{eq:coeff_w})
for the coefficients to obtain:
\begin{eqnarray*}
\left|a_{j_{1}\ldots j_{d}}\right| & = & \frac{2^{-m}}{\pi^{d}}\left|\oint_{{\cal C}_{\mat{\rho}}}w_{1}^{-j_{1}-1}\cdots w_{d}^{-j_{d}-1}g\left(w_{1},\ldots,w_{d}\right)dw_{1}\cdots dw_{d}\right|\\
 & \leq & \frac{2^{-m}M}{\pi^{d}}\oint_{{\cal C}_{\mat{\rho}}}\left|w_{1}\right|^{-j_{1}-1}\cdots\left|w_{d}\right|^{-j_{d}-1}dw_{1}\cdots dw_{d}\\
 & = & \frac{2^{-m}M}{\pi^{d}\rho_{1}^{j_{1}+1}\cdots\rho_{d}^{j_{d}+1}}\oint_{{\cal C}_{\mat{\rho}}}dw_{1}\cdots dw_{d}\\
 & = & \frac{2^{-m}M}{\pi^{d}\rho_{1}^{j_{1}+1}\cdots\rho_{d}^{j_{d}+1}}\cdot2^{d}\pi^{d}\rho_{1}\cdots\rho_{d}\\
 & = & \frac{2^{d-m}M}{\rho_{1}^{j_{1}}\cdots\rho_{d}^{j_{d}}}
\end{eqnarray*}
\end{proof}

\end{document}